\documentclass[11pt,leqno]{amsart}

\usepackage[top=2.5 cm,bottom=2 cm,left=2.5 cm,right=2 cm]{geometry}
\usepackage[utf8]{inputenc}

\usepackage{bbm}
\usepackage{esint}
\usepackage{graphicx}
\usepackage{hyperref}
\usepackage{color}
\usepackage{amsfonts}
\usepackage{psfrag}
\newcounter{stepnb}

\newtheorem{theorem}{Theorem}[section]
\newtheorem{lemma}[theorem]{Lemma}
\newtheorem{proposition}[theorem]{Proposition}
\newtheorem{corol}[theorem]{Corollary}
\newtheorem{hyp}{Hypothesis}

\newtheorem{remark}{Remark}[theorem]
\numberwithin{equation}{section}

\newcommand{\R}{\mathbb{R}}

\newcommand{\unpo}{\mathcal{O}(1)}

\DeclareMathOperator*{\monconc}{monconc}
\DeclareMathOperator*{\monconv}{monconv}

\newcommand{\ee}{\varepsilon}
\newcommand{\mf}{\mathbf}

\newcommand{\be}{\begin{equation}}
\newcommand{\eq}{\end{equation}}

\newcommand{\comment}{\color{blue}}

\begin{document}

\title[]{Characteristic boundary layers for mixed hyperbolic-parabolic systems in one space dimension, and applications to the Navier-Stokes and MHD equations}

\author[S.~Bianchini]{Stefano Bianchini}
\address{S.B. SISSA, via Bonomea 265, I-34136, Trieste, Italy}
\email{bianchin@sissa.it}
\author[L.~V.~Spinolo]{Laura V.~Spinolo}
\address{L.V.S. IMATI-CNR, via Ferrata 5, I-27100 Pavia, Italy.}
\email{spinolo@imati.cnr.it}
\maketitle
{
\rightskip .85 cm
\leftskip .85 cm
\parindent 0 pt
\begin{footnotesize}

{\sc Abstract.}
We provide a detailed analysis of the boundary layers for mixed hyperbolic-parabolic systems in one space dimension and  small amplitude regimes. As an application of our results, we describe the solution of the so-called boundary Riemann problem recovered as the zero viscosity limit of the physical viscous approximation. In particular, we tackle the so called \emph{doubly characteristic} case, which is considerably more demanding from the technical viewpoint and occurs when the boundary is characteristic for both the mixed hyperbolic-parabolic system and for the hyperbolic system obtained by neglecting the second order terms. Our analysis applies in particular to the compressible Navier-Stokes and MHD equations in Eulerian coordinates, with both positive and null conductivity. In these cases, the  
doubly characteristic case occurs when the velocity is close to $0$.  The analysis extends to non-conservative systems.



\end{footnotesize}
}
\section{Introduction and main results}
We deal with boundary layers for mixed hyperbolic-parabolic systems in the form 
\be 
\label{e:viscouscl}
    \mf w_{t}^{ \ee}  +
     \mf f (\mf w^{ \ee})_x = \ee  
   [\mf D(\mf w^{ \ee}) \mf w_x^{ \ee}]_{x}.
\eq 
In the previous expression, the unknown $\mf w$ attains values in $\R^N$, $\mf f: \R^N \to \R^N$ is a regular function, the viscosity matrix $\mf D$ attains values in $\mathbb M^{N \times N}$ and is 
positive semi-definite, $\ee$ is a positive parameter and the subscripts $_t$ and $_x$ denote the partial derivatives with respect to the time and space variable, respectively. We later discuss the precise hypotheses we impose on $\mf f$ and $\mf D$, here we only refer to the fundamental work by Kawashima and Shizuta~\cite{KawashimaShizuta1} and we point out that in the present work we handle the case when $\mf D$ is singular and the boundary $x=0$ is \emph{doubly characteristic}, i.e. is characteristic for both~\eqref{e:viscouscl} and the conservation law  
 \be 
\label{e:cl}
   \mf w_t  +
    [ \mf f (\mf w)]_x = \mf 0,
\eq
which is formally recovered in the limit $\ee \to 0^+$. We specify in the following the notion of characteristic boundary. Our analysis applies to the Navier-Stokes and MHD equations and in these cases the doubly characteristic case occurs when the fluid velocity is close to $0$ at the boudary. As an application of our results, we provide a precise description of the solution of~\eqref{e:cl} that is recovered in the limit $\ee \to 0^+$ of~\eqref{e:viscouscl} in the case of  small amplitude, Riemann-type data. 

There is an extremely large number of works devoted to the analysis of the boundary layers for parabolic and mixed hyperbolic-parabolic systems. Here we only quote the 
works~\cite{BenzoniGavageSerreZumbrun,Gisclon,GrenierRousset,GuesMetivierWilliamsZumbrun,JosephLeFloch,JosephLeFloch2,MatsumuraNishida,MetivierZumbrun,NakamuraNishibata,Rousset2,Rousset,SerreZumbrun,Xin} and we refer to the review paper by Grenier~\cite{Grenier} for a more extended discussion and list of references. We also point out that the analysis in the case where the viscosity matrix $\mf D$ in~\eqref{e:viscouscl} is singular involves severe technical challenges, but it is the most interesting from the physical viewpoint because it is the case of the compressible Navier-Stokes and MHD equations. The characteristic and doubly characteristic case involve further technical challenges that we outline in the following. 
Note that the analysis in previous works like~\cite{BianchiniSpinolo:ARMA,Rousset} applies to the 
compressible Navier-Stokes and MHD equations written in Lagrangian coordinates, but does not directly apply to the same equations written in Eulerian coordinates. To the best of our knowledge, the present work is the first to provide a complete boundary layers analysis for mixed hyperbolic-parabolic systems that  directly applies to the compressible Navier-Stokes and MHD equations in Eulerian coordinates with fluid velocity close to $0$.  Our analysis applies to small amplitude, one-dimensional boundary layers.  

Concerning the applications to~\eqref{e:cl}, we refer to the classical books by Dafermos and Serre~\cite{Dafermos,Serre1} for a comprehensive introduction to systems of conservation laws. Existence and uniqueness results for~\emph{admissible} solutions of~\eqref{e:cl} are presently only available under the assumptions that the data have sufficiently small total variation, or in the case of special system, see~\cite{Bressan:book,Dafermos,Serre1}. In the present paper we focus on the so-called boundary Riemann problem, i.e. we assume that the initial and boundary data are constant. The Riemann {problem} and boundary Riemann problem play a key role in the analysis of conservation laws from both the theoretical and numerical viepoint, see~\cite{Bressan:book,LeVeque}. 

The analysis of the viscous approximation~\eqref{e:viscouscl} is extremely relevant for~\eqref{e:cl} and it is  
particularly interesting in the case of boundary problems  because it turns out that the limit \emph{depends} on the choice of the viscosity matrix $\mf D$, i.e. in general it changes when we change $\mf D$, see Gisclon~\cite{Gisclon}. Remarkably, this can happen even in the simplest possible case where $\mf f$ is a linear function and $\mf D$ is a constant matrix, provided $N \ge 2$. 
Note that fact that the solution of an initial-boundary value problem depends on the underlying viscous mechanism has also numerical implications, see~\cite{MishraSpinolo}. 
On the other hand, note that, in the case of the Riemann problem (with no boundary), the analysis in~\cite{Bianchini} strongly suggests that the limit of~\eqref{e:viscouscl} does not depend on $\mf D$.  Finally, we point out that establishing the convergence $\ee \to 0^+$ of~\eqref{e:viscouscl} is presently a challenging open problem for both the initial value (Cauchy) problem and the initial-boundary value problem. 
There are, however, partial results, see in particular~\cite{AnconaBianchini,Gisclon,Rousset2,Rousset,Spinolo:2bvp}
 for the initial-boundary value problem.     

In a previous work~\cite{BianchiniSpinolo:ARMA} the authors provided a detailed description of the solution of the boundary Riemann problem obtained by taking the limit $\ee \to 0^+$ in~\eqref{e:viscouscl}. The analysis in~\cite{BianchiniSpinolo:ARMA} applies to the Navier-Stokes and MHD equation written in Lagrangian coordinates, but does not apply to the same equations written in Eulerian coordinates\footnote{This was first pointed out to us by Fr\'ed\'eric Rousset} owing to severe technical obstructions. In the present paper we overcome these obstructions by relying on a very careful study of the structure of boundary layers and traveling waves profiles. Note that in the following we will often use the Euler equations written in Eulerian coordinates as a guiding example, but our goal is to develop an analysis that applies to general systems (as general as possible).  Note furthermore that, besides the applications to the boundary Riemann problem, the boundary layers analysis is of independent interest as it provides very precise information on the transient behavior $\ee \to 0^+$ and on the limit.

We now highlight the main technical challenges we have to tackle in our analysis. First, we rely on the analysis by Kawashima and Shizuta~\cite{KawashimaShizuta1} and we realize that, under physically sounded assumptions (see Theorem~\ref{t:ks} in~\S~\ref{t:ks} here), there is a change of variables  
$\mf w^{\ee} \longleftrightarrow \mf u^{\ee}$ such that in the new dependent variables~\eqref{e:viscouscl} rewrites as 
\be
\label{e:symmetricee}
   \mf E (\mf u^{ \ee}) \mathbf{u}_t^{ \ee}  +
    \mf A(\mf u^{ \ee})  \mf u_x^{ \ee}=  {\ee} \mf B(\mf u^{\ee}) \mf u_{xx}^{\ee} +
    \mf G (\mf u^{ \ee}, { \ee} \mf u_x^{ \ee} )  \mf u_x^{ \ee}
\eq 
for suitable matrices $\mf E$, $\mf A$, $\mf B$ and $\mf G$ satisfing the properties described in~\S~\ref{s:h}. In particular, the matrix $\mf A$ is symmetric, the matrix $\mf B$ is positive semi-definite and block diagonal and the matrix $\mf E$ is positive definite. Next, by using the change of variables $(x, t) \mapsto (\ee x, \ee t)$ we reduce to the case where $\ee =1$ and we arrive at 
\be
\label{e:symmetric}
   \mf E (\mf u) \mathbf{u}_{t}  +
    \mf A(\mf u)  \mf u_{ x}=  \mf B(\mf u) \mf u_{ xx} +
    \mf G (\mf u,  \mf u_{x} )  \mf u_{ x} {\comment .}
\eq 
To highlight the heart of the matter and avoid some technicalities, we now focus on the case where~\eqref{e:symmetric} are the Navier-Stokes equations for a polytropic gas written in Eulerian coordinates, but our considerations apply in much greater generality. In the case of the Navier-Stokes equations~\eqref{e:symmetric} is a system of $3$ equations and the components of the unknown $\mf u$  
are the fluid density $\rho>0$, the fluid velocity $u$ and the temperature $\theta>0$. It turns out (see
~\S~\ref{ss:h:ns} for the explicit computations) that the matrices in~\eqref{e:symmetric} are 
\be
\label{e:matrici}
    \mf E (\mf u)  \! \!= \! \! 
    \left(
    \begin{array}{cc}
    R \theta /\rho^2 & \mf 0_2^t \\
    \mf 0_2 & \mf E_{22} \\
    \end{array}
    \right) \! \!,
    \; \;
    \mf A (\mf u)  \! \!= \! \! 
    \left(
    \begin{array}{cc}
    R \theta u /\rho^2 & \mf a_{21}^t \\
    \mf a_{21} & \mf A_{22} \\
    \end{array}
    \right) 
    \! \!,
    \; \;
    \mf B (\mf u)  \! \!= \! \! 
    \left(
    \begin{array}{cc}
    0 &  \mf 0_2^t \\
    \mf 0_2^t & \mf B_{22} \\
    \end{array}
    \right) 
    \! \!,
    \; \;
    \mf G (\mf u, \mf u_x)  \! \!= \! \! 
    \left(
    \begin{array}{cc}
    0 & \mf 0_2^t \\
    \mf g_1 & \mf G_{22} \\
    \end{array}
    \right). 
\eq
In the previous expression, $R$ is the universal gas constant, $\mf 0_2$ denotes the column vector $(0, 0)^t$ and the symbol $^t$ denotes the transpose. The explicit expression of the vectors $\mf a_{21}, \mf g_1 \in \R^2$ and of the matrices $\mf E_{22}, \mf A_{22}, \mf B_{22}, \mf G_{22} \in \mathbb{M}^{2 \times 2}$ is not important here, but one should keep in mind that $\mf E_{22}$ and $\mf B_{22}$ are both positive definite, and that $\mf g_1$ and $\mf G_{22}$ represent higher order terms that vanish when $\mf u_x = \mf 0$. We set $\mf z_2 : = (u_x, \theta_x)^t$ and we conclude that the equation at the first line of~\eqref{e:symmetric} reads
\be 
\label{e:intro:ns1}
\frac{R \theta}{\rho^2} \big( \rho_t + u \rho_x )+ \mf a_{21}^t \mf z_2  =0.  
\eq 
In other words, since the matrix $\mf B$ is singular, then~\eqref{e:symmetric} is a mixed hyperbolic-parabolic system and contains the ``hyperbolic part" given by equation~\eqref{e:intro:ns1}.  Note that~\eqref{e:intro:ns1} implies that we cannot assign a boundary condition on the component $\rho$ in the case where $u \leq 0$. When $u=0$, the number of boundary conditions we can impose on~\eqref{e:symmetric} changes and this is the reason why we say that the boundary is characteristic for the mixed hyperbolic-parabolic system~\eqref{e:symmetric}.  

From the technical viewpoint, the main challenge in the case where $u$ can attain the value $0$ is the following.  Let us focus on the boundary layers, i.e. on the steady solutions of~\eqref{e:symmetric}: in the case of the Navier-Stokes equations, by plugging $\mf u_t= \mf 0$ in~\eqref{e:symmetric}, using~\eqref{e:matrici} and recalling that $\mf z_2 : = (u_x, \theta_x)^t$ we arrive at  
$$ 
     \left\{ 
     \begin{array}{ll}
     \displaystyle{\frac{R \theta}{\rho^2}  u \rho_x + \mf a_{21}^t \mf z_2  =0} {,}\\
     \mf a_{21} \rho_x + \mf A_{22} \mf z_2 = \mf B_{22} \mf z_{2x} + \mf g_1 \rho_x + \mf G_2 \mf z_{2}. \\
     \end{array}
     \right. 
$$
If $u \neq 0$, we can solve for $\rho_x$ the first line and arrive at 
\be 
\label{e:intro:ns2}
\left\{ 
     \begin{array}{ll}
\rho_x = -\displaystyle{\frac{\rho^2}{R \theta}   \frac{ \mf a_{21}^t \mf z_2}{u}} {,}\\
      \mf z_{2x} = \mf B_{22}^{-1} \left[ - \displaystyle{\frac{\rho^2}{R \theta} \frac{\mf a_{21} \mf a_{21}^t  \mf z_2}{u}}
      +  \mf A_{22} \mf z_2 + \displaystyle{\frac{\rho^2}{R \theta}   \frac{ \mf g_1 \mf a_{21}^t \mf z_2}{u}} 
      - \mf G_2 \mf z_{2}\right] {.}
 \end{array}
     \right. 
\eq
Note, however, that the above equations are singular at $u=0$. As a matter of fact, the boundary layers of the MHD equations are also singular at $u=0$. This is the reason why the analysis in~\cite{BianchiniSpinolo:ARMA} does not apply the Navier-Stokes and MHD equations written in Eulerian coordinates. To tackle this challenge we rely 
on invariant manifold techniques introduced in~\cite{BianchiniSpinolo:JDE}. We refer to~\S~\ref{ss:roadmap} for an overview of the main ideas involved  in the analysis. 

Before stating our main results we make a further observation. Let us consider the hyperbolic equations obtained by setting $\ee =0$ in~\eqref{e:symmetricee}, i.e.
\be
\label{e:hyperbolic}
   \mf E (\mf u) \mathbf{u}_{t}  +
    \mf A(\mf u)  \mf u_{x}= \mf 0 {.}
\eq
In the case where~\eqref{e:symmetricee} are the Navier-Stokes equations, we obtain the Euler equations. The number of boundary conditions we impose on $\mf u$ is then equal to the number of strictly positive eigenvalues of the matrix $\mf A$. In the case where~\eqref{e:hyperbolic} are the Euler equations, one eigenvalue of $\mf A$ vanishes when $u$ vanishes. This implies that  when $u=0$ the boundary is characteristic not only for~\eqref{e:symmetric}, but also for~\eqref{e:hyperbolic} and this is the reason why we term this case  \emph{doubly characteristic}\footnote{We owe this name to Denis Serre}. Note that an analogous (but more complicated) situation occurs  in the case of the MHD equations. From the technical viewpoint, the fact that the boundary is characteristic also at the hyperbolic level implies severe complications. In particular, one has to take into account the possibility of contact discontinuities with $0$ or slightly positive speed. In other words, one cannot separate the analysis of boundary layers from the analysis of traveling waves, but has to {consider possible interactions among them}.

We now informally discuss the hypotheses we impose on~\eqref{e:viscouscl} and we refer to~\S~\ref{s:h} for the precise statement, which requires some heavy notation. As a matter of fact, our hypotheses are all satisfied by the Navier-Stokes and MHD equations (with both positive and null conductivity), see~\S~\ref{ss:h:ns} and~\S~\ref{ss:h:mhd}. First, we require that the hypotheses of a theorem due to Kawashima and Shizuta~\cite{KawashimaShizuta1}, which is Theorem~\ref{t:ks} in~\S~\ref{s:h}, are satisfied. Theorem~\ref{t:ks} states that, under physically sounded assumptions,~\eqref{e:viscouscl} can be re-written as~\eqref{e:symmetricee} and~\eqref{e:symmetricee}  is in the \emph{normal form}, in the Kawashima-Shizuta sense~\cite{KawashimaShizuta1}, i.e. the coefficients $\mf E$, $\mf A$, $\mf B$ and $\mf G$ satisfy a set of properties that in~\S~\ref{s:h} we collect as Hypothesis~\ref{h:normal}. 
In particular, Hypothesis~\ref{h:normal} says that the matrix $\mf A$ is symmetric, the matrix $\mf B$ is positive semi-definite and block diagonal and the matrix $\mf E$ is positive definite.  

Hypothesis~\ref{h:ks} is the so-called \emph{Kawashima-Shizuta} condition. Very loosely speaking, it is a coupling condition that rules out the possibility of decomposing~\eqref{e:symmetricee} into a purely hyperbolic and a viscous part. Hypothesis~\ref{h:sh} states that system~\eqref{e:hyperbolic} is \emph{strictly hyperbolic}, i.e. $\mf E^{-1} \mf A$ has $N$ real and distinct eigenvalues. Hypothesis~\ref{h:sing1d} states that the number of boundary conditions we have to impose on the mixed hyperbolic-parabolic equation~\eqref{e:symmetricee} depends on the sign of a scalar function $\alpha (\mf u)$. In the case of the Navier-Stokes equations, $\alpha (\mf u) =u$, but Hypothesis~\ref{h:sing1d} is actually trivial in the case where the kernel of $\mf B$ is one dimensional, which is the case of the Navier-Stokes equations and the MHD equations with positive conductivity. Hypothesis~\ref{h:sing1d} is only meaningful in the case where the kernel of $\mf B$ has larger dimension, as in the case of the MHD equations with null conductivity. Hypothesis~\ref{h:jde} is a technical assumption that pops out when we construct the invariant manifolds containing the traveling waves and boundary layers of~\eqref{e:symmetricee}. Remarkably, Hypothesis~\ref{h:jde} is the very condition that allows us to assign in a consistent and fairly natural way the boundary condition on~\eqref{e:symmetricee}. We refer to~\S~\ref{ss:bc} for the technical details concerning the boundary condition, but loosely speaking we can construct a function $\boldsymbol{\beta}: \R^N \times \R^N \to \R^N$, depending on system~\eqref{e:symmetricee} and satisfying the following. Given $\mf u_b \in \R^N$, by imposing the condition $\boldsymbol{\beta} ( \mf u^{\ee}(x=0), \mf u_b) = \mf 0$ we are imposing the correct boundary conditions: for instance, if~\eqref{e:symmetricee} are the Navier-Stokes equations we are imposing
\begin{eqnarray*}
     \rho(x=0) - \rho_b=0, \; u(x=0)- u_b=0,  \;  \theta (x=0)-\theta_b =0 & \text{if $u_b>0$}, \\
   u(x=0)- u_b=0,  \;  \theta (x=0)-\theta_b =0 & \text{if $u_b \leq 0$.} 
\end{eqnarray*}
We can now state our main result. We first provide the statement in the conservative case and hence we use the dependent variable $\mf w$. A remark about notation: since the map $\mf w \longleftrightarrow \mf u$ is invertible, by assigning a boundary condition on $\mf w$ through the function $\boldsymbol{\beta}$ we assign a boundary condition on $\mf u$, and viceversa. By a slight abuse of notation, to simplify the exposition in the following we write 
$\boldsymbol{\beta}( \mf w(x=0),\mf w_b {)}$ instead of $\boldsymbol{\beta}( \mf u(\mf w(x=0)), \mf u(\mf w_b))$. Also, we recall that shocks and contact discontinuities are termed \emph{Liu admissible} if they satisfy the admissibility criterion introduced in~\cite{Liu1}. 
\begin{theorem} \label{t:main}
Assume that the hypotheses of Theorem~\ref{t:ks} in~\S~\ref{s:h}  are satisfied. Assume furthermore that Hypotheses~\ref{h:ks},$\dots$,~\ref{h:jde} in~\S~\ref{s:h} hold true. Then there are constants $C>0$ and $\delta>0$ such that the following holds. For every $\mf w_i$, $\mf w_b \in \R^N$  such that $| \boldsymbol{\beta} (\mf w_i, \mf w_b )| <  \delta$, there is a self-similar function $\mf w$ such that 
\begin{itemize}
\item[a)] $\mathrm{TotVar} \, \mf w(t, \cdot) \leq C \delta$  for a.e. $t\in [0, + \infty[$. 
\item[b)] $\mf w$ is a distributional solution of~\eqref{e:cl} and satisfies the initial condition $\mf w(0, \cdot) = \mf w_i$.  
\item[c)] $\mf w$ contains at most countably many shocks and contact discontinuities, each of them admissible in the sense of Liu.
\item[d)] Let $\bar{\mf w}$ be the trace of $\mf w$ at $x=0$. Then there is $\underline {\mf w} \in \R^N$ such that 
\begin{itemize}
\item[d1)] $\mf f(\bar{\mf w})= \mf f (\underline {\mf w})$  and the shock or contact discontinuity between $\bar{\mf w}$ (on the right) and $\underline {\mf w}$ (on the left) is admissible in the sense of Liu; 
\item[d2)] there is a so-called ``boundary layer" $\boldsymbol{\varphi}: \R_+ \to \R$ such that 
\be 
\label{e:boundaryl}
\left\{
\begin{array}{ll}
  \mf f (\boldsymbol{\varphi})' = 
   [\mf D(\boldsymbol{\varphi}) \boldsymbol{\varphi}']' , \\
  \boldsymbol{\beta}( \boldsymbol{\varphi}(0), \mf w_b)= \mf 0_N,  \quad \lim_{x \to + \infty} \boldsymbol{\varphi}(x)= \underline{\mf w}  \\
   \end{array}
\right.
\eq 
\end{itemize}
\end{itemize}
\end{theorem}
Some remarks are in order:
\begin{itemize}
\item assume that the boundary is not characteristic at the hyperbolic level, i.e. all the eigenvalues of the Jacobian matrix $\mf D \mf f$ are bounded away from $0$. Then from condition d1) we get $\bar{\mf w} = \underline{\mf w}$. Conversely, in the characteristic case it may happen that $\bar{\mf w} \neq \underline{\mf w}$ and that there is a $0$ speed shock or contact discontinuity at the boundary. 
\item Condition~\eqref{e:boundaryl} means that $\boldsymbol{\varphi}$ is a steady solution of the mixed-hyperbolic system~\eqref{e:viscouscl}, and satisfies the boundary condition.  
\item In the proof of Theorem~\ref{t:main} we provide an explicit construction of $\mf w$, and this construction has a unique outcome once the Cauchy and boundary data are fixed. Loosely speaking this is basically the same situation as in the paper by Lax~\cite{Lax}, where one explicitly constructs a solution of the Riemann problem and the construction has a unique outcome once the left and right state are fixed.  
\item We are actually confident that one could apply the same techniques as in~\cite{ChristoforouSpinolo} and show that there is a unique function satisfying conditions a), $\dots$, d) above. Note that this would not contradict the fact that the limit of~\eqref{e:viscouscl} depends on $\mf D$ because condition d2) involves the viscosity matrix $\mf D$.  
\item The proof of Theorem~\ref{t:main} was established in~\cite{BianchiniSpinolo:ARMA} (see also~\cite[\S 6.3]{Spinolo:2bvp}) under more restrictive assumptions. In particular, in~\cite{BianchiniSpinolo:ARMA}  we introduced a condition of so-called \emph{block linear degeneracy} which rules out the possibility that the boundary layers and the traveling wave satisfy a singular ODE like~\eqref{e:intro:ns2}.  
The condition of block linear degeneracy is violated by the Navier-Stokes and MHD equations written in Eulerian coordinates when the fluid velocity is close to $0$.
\item  The only reason why we have to assume that the hypotheses of  Theorem~\ref{t:ks} are satisfied is because Theorem~\ref{t:ks} implies that the properties collected as Hypothesis~\ref{h:normal} in~\S~\ref{s:h} are satisfied. What we actually need in the analysis is Hypothesis~\ref{h:normal}. 
The reason why we assume the hypotheses of  Theorem~\ref{t:ks} is because we want to provide a statement in the conservative case and hence we need to pass from~\eqref{e:viscouscl} to~\eqref{e:symmetricee}. In Proposition~\ref{p:nc} we state our main results in the general non-conservative case and hence we replace the hypotheses of Theorem~\ref{t:ks} with Hypothesis~\ref{h:normal}.
\item Our analysis does not require that the vector fields are either linearly degenerate or genuinely nonlinear, which is a common assumption in the analysis of conservation laws, see~\cite{Bressan:book,Dafermos,Serre1}.
As a matter of fact, we do not impose any assumption on the number of inflection points of the $i$-th eigenvalue along the $i$-th characteristic vector field. 
\item As we mentioned before, the proof of Theorem~\ref{t:main} is based on a careful analysis of the boundary layer structure, which we feel is of independent interest and allows to establish several corollaries. As an example, {in~\S~\ref{ss:bc}} we state Corollary~\ref{c:traccia}, which concerns the sign of $\alpha$ evaluated at the hyperbolic trace $\bar{\mf w}$. We recall that the sign of $\alpha$ determines the number of boundary conditions 
we can impose on the mixed hyperbolic-parabolic system~\eqref{e:symmetric} and $\alpha =u$ in the case of the Navier-Stokes equations. The precise statement of Corollary~\ref{c:traccia} requires some technical preliminary consideration and this is why we postpone it to~\S~\ref{ss:bc}. 
\end{itemize}
In the proof of Theorem~\ref{t:main} we always use formulation~\eqref{e:symmetricee} and hence our analysis directly applies to the non conservative case. We now discuss how to modify the statement of Theorem~\ref{t:main} in the non conservative case. First, we point out that condition b) does not make sense, since in general we cannot provide a distributional formulation of the quasilinear system~\eqref{e:hyperbolic}. Loosely speaking, our analysis provides a characterization of the limit $\ee \to 0^+$ of~\eqref{e:symmetricee}, provided the limit exists. Note that convergence results for viscous approximation of non conservative systems have been established in some special case, see for instance~\cite{BianchiniBressan,Spinolo:2bvp}. The main issue in extending the statement of Theorem~\ref{t:main} to the non conservative case 
is that we cannot use the Rankine-Hugoniot conditions. However,  the analysis in~\cite{Bianchini} implies that one can still define a notion of \emph{shock curve} and extend the definition of Liu admissible discontinuity. 
The main difference with respect to the conservative case is that now in general the shock curve \emph{depends} on the approximation, i.e. on the viscosity matrix $\mf B$.
We can now provide the statement of Theorem~\ref{t:main} in the non-conservative case. 
\begin{proposition}
\label{p:nc}
Assume that Hypotheses~\ref{h:normal},$\dots$,~\ref{h:jde} in~\S~\ref{s:h} hold true. Then there are constants $C>0$ and $\delta>0$ such that the following holds. For every $\mf u_i$, $\mf u_b \in \R^N$  such that $| \boldsymbol{\beta} (\mf u_i, \mf u_b)| <  \delta$, there is a self-similar function $\mf u$ such that 
\begin{itemize}
\item[i)] $\mathrm{TotVar} \, \mf u(t, \cdot) \leq C \delta$  for a.e. $t \in [0, + \infty[$ and $\mf u(0, \cdot) \equiv \mf u_i$. 
\item[ii)] $\mf u$ contains at most countably many shocks and contact discontinuities, each of them admissible in the sense of Liu. 
\item[iii)] Let $\bar{\mf u}$ be the trace of $\mf u$ at $x=0$. Then there is $\underline {\mf u} \in \R^N$ such that 
\begin{itemize}
\item[iii)$_1$] either $\underline {\mf u}=\bar{\mf u}$ or there is a $0$-speed Liu admissible discontinuity joining $\underline {\mf u}$ (on the left) and 
$\bar{\mf u}$ (on the right) {;}
\item[iii)$_2$] there is a so-called ``boundary layer" $\boldsymbol{\varphi}: \R_+ \to \R$ such that 
\be 
\label{e:boundaryl2}
\left\{
\begin{array}{ll}
   \mf A(\boldsymbol{\varphi})  \boldsymbol{\varphi}'=  \mf B(\boldsymbol{\varphi}) \boldsymbol{\varphi}'' +
    \mf G (\boldsymbol{\varphi},  \boldsymbol{\varphi}' )  \boldsymbol{\varphi}' {,} \\
  \boldsymbol{\beta}( \boldsymbol{\varphi}(0), \mf u_b)= \mf 0_N,  \quad \lim_{x \to + \infty} \boldsymbol{\varphi}(x)= \underline{\mf u} {.} \\
   \end{array}
\right.
\eq 
\end{itemize}
\end{itemize}
\end{proposition}
The exposition is organized as follows.  {In~\S~\ref{ss:roadmap}} we provide an overview of the proof of Theorem~\ref{t:main} and Proposition~\ref{p:nc}. {In~\S~\ref{s:h}} we rigorously state Hypotheses~\ref{h:normal},$\dots$,~\ref{h:jde} and we show that they are satisfied by the Navier-Stokes and MHD equations. We also provide the rigorous formulation of the boundary condition. 
To simplify the exposition, in~\S~\ref{ss:emme0},$\dots$,\S~\ref{s:brie} we focus on the case where the dimension of the kernel of the matrix $\mf B$ in~\eqref{e:symmetric} is $1$ and we provide the the proof of Theorem~\ref{t:main} and Proposition~\ref{p:nc} in this case. 
More precisely, in ~\S~\ref{ss:emme0},$\dots$,~\S~\ref{s:complete} we discuss the analysis of the equations satisfied by the traveling waves and the boundary layers  of~\eqref{e:symmetric} and in \S~\ref{s:brie} we complete the proof. For the reader's convenience, at the end of each of~\S~\ref{ss:emme0},$\dots$,~\S~\ref{s:complete} we explicitly discuss how the analysis in the section applies to the Navier-Stokes equations and to the MHD equations with positive conductivity.
In~\S~\ref{s:bigk} we discuss the extension of the analysis at the previous sections to the case where the dimension of the kernel of the matrix $\mf B$ is bigger than $1$ and we discuss the applications to the MHD equations with null conductivity.
Finally, in~\S~\ref{s:fattorizzo},~\S~\ref{s:proof} and~\S~\ref{s:slaving} we provide the proof of some technical results we need in the previous sections. For the reader's convenience, we conclude the introduction by collecting the main notation used in the paper. 

\subsection{Proof outline} \label{ss:roadmap}
We can now discuss the basic ideas underpinning the proof of 
Theorem~\ref{t:main} and Proposition~\ref{p:nc}. The most important point is the analysis of the equations satisfied by the traveling waves and the boundary layers. Boundary layers are steady solutions of~\eqref{e:symmetric} satisfying~\eqref{e:boundaryl2}. A traveling wave with speed $\sigma$ is a solution of 
\be
\label{e:tw} 
    -\sigma \mf E(\mf u)  \boldsymbol{\varphi}' +  \mf A(\boldsymbol{\varphi})  \boldsymbol{\varphi}'=  \mf B(\boldsymbol{\varphi}) \boldsymbol{\varphi}'' +
    \mf G (\boldsymbol{\varphi},  \boldsymbol{\varphi}' )  \boldsymbol{\varphi}' 
\eq
that converges at both $- \infty$ and $+ \infty$.   As we mentioned before, the main novelty of the present paper is that we tackle the case where the boundary layers and traveling waves equations are singular, as in the case of the Navier-Stokes equations, see equation~\eqref{e:intro:ns2}. In~\S~\ref{ss:twbl} we derive the precise expression of the equations satisfied by the boundary layers and the traveling waves. For simplicity, in this introduction we only consider boundary layers, but the equation satisfied by the traveling waves is analogous. It turns out that the equation satisfied by the boundary layers has the form 
\be
\label{e:boilsd}
\mf v' = \mf h(\mf v) / \alpha (\mf v),
\eq
where $\mf v$ is the unknown and $\alpha$ is a function that can attain the value $0$. In the case of the Navier-Stokes equations, $\alpha$ is the fluid velocity. In the general case, $\alpha$ is defined in~\S~\ref{s:h} by~\eqref{e:sing1d} and the sign of $\alpha$ is what determines the number of boundary conditions we can impose on~\eqref{e:symmetric}.  

To tackle the challenges coming from the fact {that~\eqref{e:boilsd}} is a singular equation when $\alpha=0$, we rely on dynamical systems techniques. In particular, our analysis uses the idea of slaving manifold.  
We refer to the book~\cite{KatokHasselblatt} for an extended introduction to dynamical systems. 
We also refer to the lecture notes~\cite{Bressan:cm}  by Bressan on the center manifold and to the book by Perko~\cite{Perko} for a discussion about the stable manifold. 

The first remark we make {about~\eqref{e:boilsd}} is that we do not need to study \emph{all} solutions, but only those satisfying suitable conditions. Indeed, we are interested in either traveling waves or boundary layers. Traveling waves satisfy~\eqref{e:tw} and, since we focus on small total variation regimes, we can study solutions that are confined in a neighborhood of some fixed state and hence lie on a so-called \emph{center manifold}. Boundary layers satisfy~\eqref{e:boundaryl2} and hence we have to study solutions that decay to some limit as $x \to + \infty$. Note that this does not imply that the boundary layers lie on the stable manifold, because in the boundary characteristic case there might be boundary layers that are very slowly decaying. However, for the time being we neglect this technical problem and we come back to it in the following. 

To illustrate our argument to {handle~\eqref{e:boilsd}} we introduce a toy model, and we refer to~\cite[\S~1.1]{BianchiniSpinolo:JDE} for a discussion of the general linear case. Consider system~\eqref{e:boilsd} in the case where 
\be 
\label{e:toy}
\mf v : = (v_1, v_2, v_3, v_4, v_5)^t, \quad \alpha (\mf v) = v_5, \quad \mf h(\mf v)= (-v_1, - v_2 v_5, v_3 v_5, 0, 0)^t. 
\eq 
Note that $\alpha=v_5$ is actually constant in this case. We are interested in solutions with $v_5$ close, but different from, $0$. For simplicity, we also assume that $v_5>0$. By explicitly computing the solutions, we see that $v_1$ and $v_2$ are two exponentially decaying solutions and hence can be regarded as ``stable components". Note however, that there is an important difference between $v_1$ and $v_2$: $v_1$ decays to $0$ like  $\exp (-x/v_5)$, whereas $v_2$ decays to $0$ like $\exp (-x)$. For this reason, we term $v_1$ the ``fast stable component" and $v_2$ the ``slow stable component". Finally, $v_4$ and $v_5$ are globally bounded (they are constant) and hence can be regarded as ``center components". 
Note furthermore that to single out the ``fast stable component", the ``slow stable component" and the ``center component" we could proceed as follows. \\
{\sc Step 1:} we consider the equation 
\be
\label{e:odefast}
\dot{\mf v}= \mf h(\mf v),
\eq
which is obtained from~\eqref{e:boilsd} thought the change of variables $x =\alpha y$. We linearize at $(0, 0,0,0,0)$ and we consider the stable space $M^f:= \{ \mf v \in \R^5: \; v_2= v_3=v_4 =v_5 = 0 \}$ and the center space $ M^0:= \{ \mf v \in \R^5: \; v_1 =0  \}$. 
Note that the stable space provides the ``fast stable component" $v_1$. \\
{\sc Step 2:} we point out that the \emph{original} system~\eqref{e:boilsd} restricted on $M^0$ is not singular as it reads 
\be  
\label{e:boilres}
   v_2'= - v_2, \quad v_3' = v_3, \quad v_4' =0, \quad v_5'=0.  
\eq 
Note that this is due to the particular structure of the system. We study~\eqref{e:boilres}, i.e.~\eqref{e:boilsd} restricted on $M^0$. We linearize at $(0,0,0,0)$ and we consider the stable space 
$M^s: = \{ \mf v \in \R^4: \; v_3=v_4 =v_5 = 0 \}$ and the center space $M^{00}: = \{ \mf v \in \R^4: \; v_2= v_3=0 \}$. Note that $M^{00}$ provides the center component of the original system~\eqref{e:boilsd} and that $M^s$ provides the ``slow stable component".\\
{\sc Step 3:} to construct the stable component we ``sum" the fast stable component and the slow stable component, i.e. {in this case} we take 
$M^f \oplus M^s$. 
\vspace{0.2cm}

Very loosely speaking, the basic idea underpinning the analysis in~\S~\ref{ss:emme0},~\S~\ref{s:center},~\S~\ref{s:blemme0},
~\S~\ref{s:complete} is that we want to proceed according to {\sc Step 1}, $\dots$, {\sc Step 3} above in the case where~\eqref{e:boilsd} is the equation for the boundary layers of~\eqref{e:symmetric}. We now  briefly discuss the extension of each of the above steps. \\
{\sc Step 1:} we consider the equation~\eqref{e:odefast}, which is now \emph{formally} obtained
 from~\eqref{e:boilsd} thought the change of variables ${dx/dy =\alpha(\mf v)}$. The derivation at this stage is only formal because we are not yet able to show that ${dx/dy =\alpha}$ is actually a change of variables (for instance, it may in principle happen that $\alpha \equiv 0$ on some interval). Recall that $\alpha$ is not constant now: for instance, in the case of the Navier-Stokes equations $\alpha$ is the fluid velocity. 

Next, we linearize~\eqref{e:boilsd} at an equilibrium point where $\alpha (\mf v)=0$ and we construct the stable manifold and the center manifold (the center manifold is actually not unique, but we can arbitrarily fix one). The stable manifold will loosely speaking provide 
the ``fast stable component". We term the center manifold $\mathcal M^0$. Note that $\mathcal M^0$ is a center manifold for~\eqref{e:odefast}, not for~\eqref{e:boilsd}. \\
{\sc Step 2:} by using Hypothesis~\ref{h:jde}, in~\S~\ref{ss:emme0} we show that  the \emph{original} system~\eqref{e:boilsd} restricted on $\mathcal M^0$  is not singular. Next, we study system~\eqref{e:boilsd} restricted on $\mathcal M^0$ and we construct the center manifold $\mathcal M^{00}$ and the stable manifold $\mathcal M^{s}$. The stable manifold $\mathcal M^s$ will provide the ``slow stable component". From the technical viewpoint the 
analysis of system~\eqref{e:boilsd} restricted on $\mathcal M^0$ is actually quite demanding.  \\
{\sc Step 3:}  we construct the ``complete stable component" by combining the ``fast stable component" and the ``slow stable component". In the case of the toy model~\eqref{e:toy} we could simply add them. Owing to the nonlinearities, in the general case we have to take into account possible interactions. As a matter of fact, the argument is completed by using the notion of \emph{slaving manifold}, which is recalled in~\S~\ref{s:slaving}.  There is actually one issue that we are left to tackle: we have constructed  the ``fast stable component" by using equation~\eqref{e:odefast}, and the change of variables from~\eqref{e:boilsd} to~\eqref{e:odefast} has to be rigorously justified. We do so in~\S~\ref{ss:back} by relying on Hypothesis~\ref{h:jde}. In particular, we show that, if the solution lies on the manifolds we have constructed and $\alpha > 0$ at $x=0$, then $\alpha > 0$ for every $x>0$. 
\vspace{0.2cm}

To complete the overview of the proof of Theorem~\ref{t:main} and Proposition~\ref{p:nc} we have to tackle one last issue we have so far neglected: the boundary is characteristic for the hyperbolic system~\eqref{e:hyperbolic} and hence we have to take into account the possibility that there are boundary layers that do not lie on the stable manifold because they very slowly decay to their asymptotic state.  
Also, we have to take into account that there might be traveling waves~\eqref{e:tw} having speed $\sigma$ positive but very close to $0$, and that by slightly perturbing them we may obtain traveling waves with negative speed, which are not admissible since the domain is $x \in [0, + \infty[$. We address these challenges in~\S~\ref{s:center} by using the same approach as in~\cite{AnconaBianchini,BianchiniSpinolo:ARMA}. As a matter of fact, the analysis greatly simplifies if the characteristic field is \emph{linearly degenerate}, which is the case of the Navier-Stokes and MHD equations. For this reason we first provide the analysis in the linearly degenerate case and next we consider the general case.

\subsection{Notation}
We use standard characters to denote real numbers, bold characters  to denote vectors and 
capital bold letters to denote matrices. In other words $c \in \R, \mf c \in \R^d, \mf C \in \mathbb{M}^{d\times d}$. 
In general, we regard vectors as \emph{column} vectors: $\mf c$ is a column vector, $\mf c^t$ is a row vector.
\subsubsection{General mathematical symbols}
\begin{itemize}
\item $\mf c^t, \mf C^t$: the transpose of the vector $\mf c$, of the matrix $\mf C$. 
\item $\mf 0_d$: the zero vector in $\R^d$. 
\item $\mathbb{M}^{a \times b}$: the space of $a \times b$ matrices
\item $\mf 0_{a \times b}:$ the matrix in $\mathbb{M}^{a \times b}$ having all entries equal to $0$.  
\item $\mf I_d$: the identity matrix in $\mathbb{M}^{d \times d}$.
\item $\mathrm{B}_r^d(\mf x_0)$: the ball of radius $r$ and center $\mf x_0$ in $\R^d$. 
\item $\monconc_{[0, s_k ]} f$: the monotone concave envelope of the function $f$ on the interval $[0, s_k]$, which is defined as in~\eqref{e:cosaemonconc}.
\end{itemize}

\subsubsection{Symbols introduced in the present paper}
\begin{itemize}
\item $\mf f$: the flux function in~\eqref{e:cl}.
\item $N$: the number of unknowns in~\eqref{e:cl} and~\eqref{e:symmetric}.  
\item $\mf E, \mf A, \mf B, \mf G$: the matrices in~\eqref{e:symmetric}. 
\item $h$: the dimension of the kernel of $\mf B$, see~\eqref{e:B}.   
\item $\mf B_{22}, \mf E_{11}, \mf E_{22}$: see~\eqref{e:B} and~\eqref{e:e}. 
\item $\mf g_1, \mf G_2$: see~\eqref{e:G}.
\item $\lambda_1(\mf u), \dots, \lambda_N(\mf u)$: the eigenvalues of the matrix $\mf E^{-1} \mf A(\mf u)$. 
\item $k$: $\lambda_k$ is the characteristic eigenvalue, i.e. $\lambda_k (\mf u) \sim 0$. 
\item $\mf u^\ast$: a fixed state satisfying~\eqref{e:uast}. 
\item $\boldsymbol{\beta}$: see~\eqref{e:esse}. 
\item $\mf v'$: the derivative of $\mf v$ with respect to the slow variable $x$, see~\eqref{e:odeslow}. 
\item $\dot{\mf v}$: the derivative of $\mf v$ with respect to the fast variable $y$, see~\eqref{e:odefast}. 
\item $\mf R_0$: see Lemma~\ref{l:slow}.
\item $\mf d$: see~\eqref{e:definiamod}
\item $\mf \Theta_0$: see Lemma~\ref{l:suemme0} and equation~\eqref{e:cosaeLambda0}. 
\item $\mf r_{00}:$ see Lemma~\ref{l:center}. 
\item $\theta_{00}$: see Lemma~\ref{l:suemme00} and equation~\eqref{e:cosaelambda00}. 
\item $f$: the function in~\eqref{e:cosaeffe}. 
\item $\tilde c$: the same constant as in~\eqref{e:ci}. 
\item $\boldsymbol{\psi}_{sl}$: see Theorem~\ref{t:slowmanifold}.
\item $\boldsymbol{\psi}_{b}$: the projection onto $\R^N$ of the map $\tilde{\boldsymbol{\psi}_{b}}$, which attains values in $\R^N \times \R^{N-1} \times \R$. The map  $\tilde{\boldsymbol{\psi}_{b}}$ is defined in the statement of Lemma~\ref{l:fastvaria}.
\end{itemize}
\section{Hypotheses and formulation of the boundary condition}
In this section we state our {assumptions} and we provide the rigorous formulation of the boundary condition we impose on~\eqref{e:symmetric}. We also state Corollary~\ref{c:traccia}. To conclude, in~\S~\ref{ss:h:ns} and~\S~\ref{ss:h:mhd} we show that our assumptions are satisfied by the Navier-Stokes and the MHD equations, respectively.  
\subsection{Hypotheses}\label{s:h}
In this paragraph we state the assumptions we impose on system~\eqref{e:symmetric}, and we comment on them. We first make a preliminary remark:  to simplify the notation in the statement of Hypothesis~\ref{h:normal},$\dots$,~\ref{h:jde} we write that the various assumptions on the coefficients $\mf E$, $\mf A$, $\mf B$ and $\mf G$ must hold ``for every $\mf u$". However, in the whole paper we actually perform a \emph{local} analysis, and hence as a matter of fact it suffices that our assumptions hold in some ball of $\R^N$. 
\begin{hyp}
\label{h:normal} 
System~\eqref{e:symmetric} is of the normal form, in the Kawashima Shizuta sense, see~\cite{KawashimaShizuta1}.  More precisely, the coefficients in~\eqref{e:symmetric} are smooth and satisfy the following assumptions:
\begin{itemize}
\item[i)] The matrix $\mf B$ satisfies the block decomposition 
\begin{equation}
\label{e:B}
    \mf B(\mf u) = 
    \left(
    \begin{array}{cc}
    \mf 0_{h \times h} & \mf 0_{N-h}^t \\
    \mf 0_{N-h} & \mf B_{22} (\mf u) \\
    \end{array}
    \right)
\end{equation}
for some symmetric and positive definite matrix $\mf B_{22} \in \mathbb M^{(N-h) \times (N-h)}$.
\item[ii)] For every $\mf u$, the matrix $\mf A (\mf u)$ is symmetric.
\item[iii)] For every $\mf u$, the matrix $\mf E (\mf u)$ is symmetric, positive definite and block diagonal, namely
\be 
\label{e:e}
    \mf E(\mf u) = 
    \left(
    \begin{array}{cc}
    \mf E_{11} (\mf u) & \mf 0_{N-h}^t \\
    \mf 0_{N-h} & \mf E_{22} (\mf u) \\
    \end{array}
    \right),
\end{equation}
where $\mf E_{11} \in \mathbb{M}^{h \times h}$ and $\mf E_{22} \in \mathbb{M}^{(N-h) \times (N- h)}$. 
\item[iv)] The second order term can be written in the form $\mf G(\mf u, \mf u_x) \mf u_x$, with  
$\mf G$ satisfying
\be
\label{e:G}
        \mf G(\mf u,  \mf u_x) = 
    \left(
    \begin{array}{cc}
    \mf 0_{h \times h} & \mf 0_{N-h}^t \\
    \mf G_1 & \mf G_2 \\
    \end{array}
    \right), \qquad 
\eq
for suitable functions $\mf G_1 \in \mathbb M^{h \times (N-h)}$, $\mf G_2
\in \mathbb  M^{(N-h) \times (N-h)}$
such that for every $i, j$ the component $(\mf G_1)_{ij}$
satisfies 
\be
\label{e:gduedue}
    (\mf G_1)_{ij}(\mf u,  \mf u_x)  = \mf g^t_{ij}(\mf u)  \mf u_x
\eq 
for some function $\mf g^{ij}: \R^N \to \R^{N}$ (note that $\mf g^{ij}$ only depends on $\mf u$). 
An analogous property holds for the coefficients of $\mf G_2$. 
 \end{itemize}
\end{hyp}
Two remarks are in order. First, Hypothesis~\ref{h:normal} comes from the works by Kawashima and Shizuta, see in particular~\cite[\S~3]{KawashimaShizuta1}. The condition on the second order term $\mf G$ is written in a slightly more precise form than in~\cite[\S~3]{KawashimaShizuta1}, but 
property iv) straightforwardly follows from the conditions in~\cite[\S~3]{KawashimaShizuta1}. 
Second, Hypothesis~\ref{h:normal} is satisfied under fairly reasonable assumptions, as the next theorem states.    
\begin{theorem}[Kawashima and Shizuta~\cite{KawashimaShizuta1}]
\label{t:ks}
Assume that system~\eqref{e:viscouscl} has an entropy function, in the sense of~\cite[Definition 2.1]{KawashimaShizuta1}. Assume furthermore that the kernel of the matrix $\mf D(\mf w)$ in~\eqref{e:viscouscl} does not depend on $\mf w$. Then there is a diffeomorphism 
$\mf w \longleftrightarrow \mf u$ such that by rewriting~\eqref{e:viscouscl} with $\mf u$ as a dependent variable we arrive at~\eqref{e:symmetric} for some matrices $\mf E, \mf A, \mf B, \mf G$ satisfying~Hypothesis~\ref{h:normal}.
\end{theorem}
We now introduce the celebrated Kawashima-Shizuta condition, which in our case reads as follows 
\begin{hyp}[Kawashima-Shizuta condition]
\label{h:ks}
For every $\mf u$ the matrices $\mf E$, $\mf A$ and $\mf B$ satisfy
\be 
\label{e:ks}
   \Big\{ \text{eigenvectors of} \;  \mf E^{-1}(\mf u) \mf A (\mf u) \Big\} \cap  \text{kernel} \; \mf B (\mf u) = \emptyset. 
   \eq
\end{hyp}
Since the matrix $\mf A$ is symmetric by Hypothesis~\ref{h:normal}, owing to Lemma~\ref{l:counting} in~\S~\ref{s:proof} we can conclude that the matrix $\mf E^{-1} (\mf u) \mf A(\mf u)$
has  $N$ real eigenvalues, provided each eigenvalue is counted according to its multiplicity. We now introduce the standard hypothesis that the system is \emph{strictly hyperbolic}. 
\begin{hyp}
\label{h:sh}
    For every $\mf u$ the matrix $\mf E^{-1} (\mf u) \mf A(\mf u)$ has $N$ distinct eigenvalues. 
\end{hyp}
We term $\lambda_1 (\mf u), \dots, \lambda_N (\mf u)$ the eigenvalues of $\mf E^{-1} (\mf u) \mf A(\mf u)$. In the following, we mostly focus on the \emph{boundary characteristic case}. More precisely, we assume that 
\begin{equation}
\label{e:sh}
    \lambda_1 (\mf u) < \dots < \lambda_{k-1} (\mf u)
    < 0  < \lambda_{k+1} (\mf u) < \dots 
    < \lambda_N(\mf u), \quad \text{for every $\mf u$}
\end{equation}
for some $k=1, \dots, N$, and that eigenvalue $\lambda_k(\mf u)$ can attain the value $0$. We now introduce the block decomposition of $\mf A$ corresponding to the block decomposition~\eqref{e:B}, i.e.
\be
\label{e:blockae}
        \mf A (\mf u) = 
    \left(
    \begin{array}{cc}
    \mf A_{11} (\mf u)  & \mf A_{21}^t (\mf u) \\
    \mf  A_{21} (\mf u) & \mf A_{22} (\mf u) \\
    \end{array}
    \right) \! \!. 
    \eq
In the following we focus on the case where there is $\mf u^\ast \in \R^N$ such that 
\be
\label{e:char} 
\lambda_k (\mf u^\ast) =0, \quad \mf A_{11} (\mf u^\ast) = \mf 0_{h \times h} \in \mathbb M^{h \times h}.   
\eq
We term this case \emph{doubly characteristic} because the first condition in~\eqref{e:char} means that the boundary is characteristic at the \emph{hyperbolic} level, while the second condition implies (as we will see in~\S~\ref{ss:bc}) that the  boundary is characteristic at the \emph{viscous} level. In the following we mostly focus on the doubly characteristic case because it is the most challenging from the technical viewpoint and the fact that we handle it is the main contribution of the present paper. If~\eqref{e:char} does not hold,  loosely speaking we can either apply a simplified version of the analysis of the present paper, or the analysis in~\cite{BianchiniSpinolo:ARMA}, see~\S~\ref{ss:nonchar} for a related discussion. We now introduce two new hypotheses: we need them to tackle the doubly characteristic case~\eqref{e:char}.
\begin{hyp}
\label{h:sing1d}
We have 
\be
\label{e:sing1d}
   \mf A_{11} (\mf u) = \alpha (\mf u) \mf E_{11}(\mf u)
\eq
for some \emph{scalar} function $\alpha$.  
\end{hyp}
Note that Hypothesis~\ref{h:sing1d} is only needed in the case $h>1$. Indeed, if $h=1$
the block decompositions~\eqref{e:e} and~\eqref{e:blockae} boil down to  
\be
\label{e:blockae2}
 \mf A (\mf u) = 
    \left(
    \begin{array}{cc}
     a_{11} (\mf u)  & \mf a_{21}^t (\mf u) \\
    \mf  a_{21} (\mf u) & \mf A_{22} (\mf u) \\
    \end{array}
    \right), \qquad 
    \mf E(\mf u) = 
    \left(
    \begin{array}{cc}
     e_{11} (\mf u)  & \mf 0\\
    \mf  0 & \mf E_{22} (\mf u) \\
    \end{array}
    \right),
\eq
where $a_{11}$ and $e_{11}$ are scalar functions and $e_{11}$ is strictly positive because $\mf E$ is positive definite by Hypothesis~\ref{h:normal}. Equation~\eqref{e:sing1d} is then satisfied provided $\alpha= a_{11}/ e_{11}$.  Also, note that in general $\mf E_{11}$ is positive definite by Hypothesis~\ref{h:normal} and hence by using Hypothesis~\ref{h:sing1d} we get that~\eqref{e:char} is satisfied if and only if 
\be
\label{e:uast}
    \lambda_{k} (\mf u^\ast)=0, \quad \alpha(\mf u^\ast) =0. 
\eq  
We now state two useful consequences of Hypothesis~\ref{h:sing1d}.
\begin{lemma}
\label{l:kaw}
       Under Hypotheses~\ref{h:normal},~\ref{h:ks} and~\ref{h:sing1d}, 
       the columns of $\mf A_{21} (\mf u)$ are linearly independent vectors. 
       \end{lemma}
\begin{proof}
We argue by contradiction: assume that there is $\mf u\in \R^N$ such that the columns of $\mf A_{21}(\mf u)$ are not linearly independent, so that there is $\mf z \in \R^h$ such that $\mf z \neq \mf 0_h$, $\mf A_{21} \mf z = \mf 0_{N-h}$. We now introduce the vector $\boldsymbol{\zeta}: = (\mf z^t, \mf 0_{N-h}^t)^t$ and we point out that $[\mf A  -\alpha  \mf E] \boldsymbol{\zeta} = \mf 0_N= \mf B \boldsymbol{\zeta}$. Since $\boldsymbol{\zeta} \neq \mf 0_N$, this contradicts the Kawashima-Shizuta condition~\eqref{e:ks}. 
\end{proof}
By recalling that $\mf A_{21} \in \mathbb M^{(N-h)\times h}$ we get 
\begin{lemma}
\label{l:dimensioni}
 Under Hypotheses~\ref{h:normal},~\ref{h:ks} and~\ref{h:sing1d} we have $h \leq N/2$.        
\end{lemma}
We now introduce our last hypothesis. 
\begin{hyp}
\label{h:jde} 
For every $\mf u$ such that $\alpha (\mf u) =0$, we have  $\nabla \alpha (\mf u)  \neq \mf 0_N$ and  furthermore
\begin{equation}
\label{e:jde}
       \nabla \alpha (\mf u)= (\mf 0_h^t, \boldsymbol{\xi}^t)    
       \end{equation}
 for some $ \boldsymbol{\xi} \in \R^{N-h}$ which can be written as a linear combination of the columns of $\mf A_{21}(\mf u)$.       
\end{hyp}
Note that we impose~\eqref{e:jde} only at points where $\alpha (\mf u)=0$.
Owing to Hypothesis~\ref{h:jde} we are able to assign in a very natural way the boundary conditions for the solution of the hyperbolic-parabolic system~\eqref{e:symmetric}, see~\S~\ref{ss:bc}. Note furthermore that Hypothesis~\ref{h:jde} is violated by the Navier-Stokes equations written in Lagrangian coordinates. However, in that case we can apply the analysis in~\cite{BianchiniSpinolo:ARMA}. 
\begin{remark}
\label{r:acca1}
In the case when $h=1$, we have $\mf a_{21} \in \R^{N-1}$ and $\mf a_{21} \neq \mf 0_{N-1}$ owing to Lemma~\ref{l:kaw}. The block decomposition~\eqref{e:G} boils down to 
\be 
\label{e:G2}
 \mf G (\mf u,  \mf u_x) = 
    \left(
    \begin{array}{cc}
     0 & \mf \mf 0 \\
    \mf  g_{1}  & \mf G_2 \\
    \end{array}
    \right)
\eq
where $\mf g_1 \in \R^{N-1}$ and $\mf G_2 \in \mathbb{M}^{(N-1) \times (N-1)}$. 
Hypothesis~\ref{h:sing1d} is trivially satisfied and Hypothesis~\ref{h:jde} means that $\nabla \alpha$ is as in~\eqref{e:jde} for some nonzero vector $\boldsymbol{\xi} $ parallel to $\mf a_{21} (\mf u)$.  
\end{remark}
Since we need it in the following sections, we state here a result concerning the signature of the matrix 
$\mf B_{22}^{-1} \mf A_{21} \mf E_{11}^{-1} \mf A_{21}^t$. 
\begin{lemma}
\label{l:signature} The the signature of the square matrix $-\mf B_{22}^{-1} \mf A_{21} \mf E_{11}^{-1} \mf A_{21}^t (\mf u)$ is as follows:
\begin{itemize}
\item the eigenvalue $0$ has multiplicity $N-2h$;
\item there are $h$ strictly negative eigenvalues (each of them is counted according to its multiplicity). 
\end{itemize}
\end{lemma}
\begin{proof}
In the proof we assume that all matrices are evaluated at the same point $\mf u$.
First, we point out that the matrices $\mf A_{21} \mf E_{11}^{-1} \mf A_{21}^t$ 
and $\mf B_{22}$ are both symmetric owing to Hypothesis~\ref{h:normal}. We apply Lemma~\ref{l:sym} in~\S~\ref{s:proof} and we conclude that to establish Lemma~\ref{l:signature} it suffices to show that $\mf A_{21} \mf E_{11}^{-1} \mf A_{21}^t$  has $h$ strictly positive eigenvalues (each of them counted according to its multiplicity) and the eigenvalue $0$ with multiplicity $N-2h$. Next, we recall that $\mf E_{11}^{-1}$ is symmetric and positive definite. This implies that we can decompose it as $\mf E_{11}^{-1} = \mf P^2$ for some 
symmetric and positive definite matrix $\mf P \in \mathbb M^{h \times h}$. This implies that $\mf A_{21} \mf E_{11}^{-1} \mf A_{21}^t= \mf M \mf M^t$ provided $\mf M: = \mf A_{21} \mf P \in \mathbb M^{(N-h) \times h}$.  

We now establish the following implication: given $\mf b \in \R^{N-h}$, 
\be 
\label{e:iffm}
        \mf M \mf M^t \mf b = \mf 0_{N-h} \Longleftrightarrow \mf M^t \mf b = \mf 0_{h}
\eq
The implication $\mf M^t \mf b = \mf 0_{h} \implies \mf M \mf M^t \mf b = \mf 0_{N-h}$ is trivial. To establish the opposite implication, we point out that the equality $\mf M \mf M^t \mf b = \mf 0_{N-h}$ implies that 
$|\mf M^t \mf b|^2=\mf b^t \mf M \mf M^t \mf b =0$, i.e. that $\mf M^t \mf b= \mf 0_{h}.$

Next, we point out that the columns of $\mf M$ are linearly independent. Indeed, assume by contradiction there is $\mf a \in \R^{h}$, $\mf a \neq \mf 0_h$ such that $\mf M \mf a = \mf A_{21} \mf P \mf a = \mf 0_{N-h}$. Owing to Lemma~\ref{l:kaw}, this implies that $\mf P \mf a = \mf 0_h$ and since $\mf P$ is positive definite this implies $\mf a = \mf 0_h$, which contradicts our assumption. 

We term $\mf m_1, \dots, \mf m_h$ the columns of $\mf M$. Since they are linearly independent, then the subspace 
$
  V:= \big\{ \mf b \in \mathbb \R^{N-h}: \; \text{$\mf b^t \mf m_i=0$ for every $i=1, \dots, h$} \big\}
$
has dimension $N-2h$. Owing to~\eqref{e:iffm}, ${\mf M \mf M^t \mf b = \mf 0_{N-h}}$ if and only if $ \mf b^t \mf M = \mf 0^t_{h}$, which in turn is equivalent to $\mf b \in  V$. This implies that $V$ is the kernel of $\mf M \mf M^t$, and hence that the multiplicity of $0$ as an eigenvalue of $\mf M \mf M^t$ is exactly $N-2h.$

To conclude, we are left to show that, if $\lambda$ is a nonzero eigenvalue of $\mf M \mf M^t$, then $\lambda >0$. We point out that
$\mf M \mf M^t$ is symmetric and hence has $N-h$ linearly independent eigenvectors. We fix $\lambda \neq 0$ such that $\mf M \mf M^t \mf c = \lambda \mf c$ for some $\mf c \in \R^{N-h}$, $\mf c \neq \mf 0_{N-h}$. By left multiplying the previous inequality times $\mf c^t$ we get $|\mf M^t \mf c|^2 = 
\lambda |\mf c|^2$, which implies $\lambda \ge 0$ and hence concludes the proof. 
\end{proof}
\subsection{Boundary conditions for the mixed hyperbolic-parabolic system}
\label{ss:bc}
In this paragraph we precisely define the boundary conditions for~\eqref{e:symmetric}. 
We first decompose $\mf u$ as 
\begin{equation}
\label{e:blocku}
      \mf u : =
      \left( 
      \begin{array}{cc} 
      \mf u_1 \\
      \mf u_2 \\
      \end{array}
      \right), \qquad \mf u_1 \in \R^h, \mf u_2 \in \R^{N-h}{,}
\end{equation}
and we introduce a preliminary result. We recall that $\alpha$ is the same function as in Hypothesis~\ref{h:sing1d}. 
\begin{lemma}
\label{l:sign}
The sign of $\alpha(\mf u_1, \mf u_2)$ does not depend on $\mf u_1$ and only depends on $\mf u_2$. 
\end{lemma}
\begin{proof}
We fix $\mf u_1, \underline{\mf u_1} \in \R^h$, $\mf u_2 \in \R^{N-h}$. We want to show that 
$\alpha (\mf u_1, \mf u_2)$ has the same sign as $\alpha (\underline{\mf u_1}, \mf u_2)$.
We consider the function $a: \R \to \R$ defined by setting 
$
    a(t) : = \alpha \big( (1-t) \underline{\mf u_1} + t \mf u_1,{\mf u}_2 \big)
$
and note that $a$ satisfies  
\begin{equation}
\label{e:auode}
   \frac{d}{dt} a (t)
   \stackrel{\eqref{e:jde}}{=}0 \quad \text{if $a(t) =0$}.
   \end{equation}
In other words, $a=0$ is an equilibrium for the ODE satisfied by $a$. By the uniqueness part of the Cauchy Lipschitz Picard Lindel\"of Theorem on the Cauchy problem, either $a$ is identically $0$, or it is always different from $0$. This concludes the proof.  
\end{proof}
Owing to the above lemma, the following function is well defined. 
$$
   \zeta(\mf u_2 )= \left\{
   \begin{array}{lll}
   1 & \text{if} \; \alpha (\mf u_1, \mf u_2)  > 0 \; \text{for some $\mf u_1${,}} \\
   0 &  \text{if} \; \alpha (\mf u_1, \mf u_2) \leq 0 \; \text{for some $\mf u_1${.}} \\
   \end{array}
   \right.
$$
We can eventually define the function $ \boldsymbol{\beta}: \R^N \times \R^N \to \R^N$ by setting 
\be
\label{e:esse}
   \boldsymbol{\beta} ( \mf u, \mf u_b) = 
   \left(
   \begin{array}{ccc}
   (\mf u_1 - \mf u_{1b}) \zeta (\mf u_{2b}) \\
   \mf u_2 - \mf u_{2b} \\
   \end{array} 
   \right).
\eq
We then assign the boundary condition on~\eqref{e:symmetric} by imposing $ \boldsymbol{\beta} (\mf u(x=0), \mf u_b)= \mf 0_N$. The rationale underpinning the above formula is the following. By using Hypotheses~\ref{h:normal} and~\ref{h:sing1d}, we write~\eqref{e:symmetric} as
\be
\label{e:mixed}
\left\{
\begin{array}{ll}
        \mf E_{11} \mf u_{1t} +     \alpha \mf E_{11}  \mf u_{1x} = 
        -  \mf A^t_{21}
         \mf u_{2x}{,} \\
        \mf E_{22}  \mf u_{2t} +      \mf A_{22} \mf u_{2x} = 
        \mf B_{22}  \mf u_{2xx} -  \mf A_{21}
        \mf u_{1x} + \mf G_1  \mf u_{1x} + \mf G_2  \mf u_{2x}. \\ 
\end{array}
\right.
\eq 
Loosely speaking, the components $\mf u_1$ represent the \emph{hyperbolic} part, and $\mf u_2$
represent the \emph{parabolic part}. 
As a matter of fact, by imposing $ \boldsymbol{\beta} (\mf u(x=0), \mf u_b)= \mf 0_N$ we always assign a boundary condition on the parabolic components $\mf u_2 \in \R^{N-h}$. Next, we consider the sign of $\alpha(\mf u_b)$, which only depends on $\mf u_{2b}$ owing to Lemma~\ref{l:sign}. We recall that $\mf E_{11}$ is a positive definite matrix. If $\alpha>0$ at the boundary, then the characteristic lines of the hyperbolic part are entering the domain $x>0$ and hence we also assign the boundary condition on the hyperbolic components $\mf u_1$. If $\alpha\leq 0$ then we do not assign any boundary condition on $\mf u_1$. 

We can now state the following corollary of the analysis done to establish Theorem~\ref{t:main} and Proposition~\ref{p:nc}, and we refer to~\S~\ref{s:brie} for the proof. 
\begin{corol}
\label{c:traccia} The function we exhibit in the proof of Proposition~\ref{p:nc} satisfies, besides properties properties i), ii) and iii), also the following property: let $\bar{\mf u}$ denotes its trace at $x=0$. If $\alpha (\mf u_b)> 0$, then $\alpha (\bar{\mf u})\ge 0$. If $\alpha(\mf u_b)< 0$, then $\alpha (\bar{\mf u})\leq 0$. 
\end{corol}
\subsection{Application to the Navier-Stokes equations}
\label{ss:h:ns}
We consider the compressible Navier-Stokes equations in Eulerian coordinates, i.e. 
\be
\label{e:ns}
\left\{
\begin{array}{lll}
\rho_t + (\rho u)_x =0 {,} \phantom{\displaystyle{\int}} \\
  (\rho u)_t +  (\rho u^2 + p )_x =  (\nu  u_x )_x {,} \phantom{\displaystyle{\int}} \\
\displaystyle{ \left( \rho \left( e + \frac{u^2}{2} \right)\right)_t + \left( \rho u \left( e + \frac{u^2}{2} \right) + p u \right)_x
=  (\kappa \theta_x + \nu u  u_x )_x } {.} 
\end{array}
\right.
\eq
Here $\rho$ represents the fluid density, $u$ the fluid velocity and $\theta$ the absolute temperature. The internal energy $e$ depends on $\theta$ and $e_\theta>0$.
Also, $\nu (\rho)>0$ and $\kappa(\rho)>0$ are the viscosity and heat conductivity coefficients, respectively.   
We consider the case of polytropic gases, where the pressure $p$ satisfies $p(\rho, \theta) = R \rho \theta$. 
We want to write~\eqref{e:ns} in the form~\eqref{e:symmetric}. First, we set $\mf u: =(\rho, u, \theta)^t$ and  we write~\eqref{e:ns} in the quasilinear form 
\be
\label{e:ql}
     \tilde{\mf E} (\mf u)  \mathbf{u}_t  +
   \tilde{\mf A}(\mf u)  \mf u_x=  \tilde{\mf B}(\mf u)  \mf u_{xx}+
    \tilde{\mf G} (\mf u,  \mf u_x )  \mf u_x.
\eq 
In the previous expression, 
$$
    \tilde{\mf E} (\mf u)\! = \! \!
    \left( 
    \begin{array}{ccc}
    1 & 0 & 0 \\
    u & \rho & 0 \\
    \psi(\theta, u) & \rho u & \rho e_\theta \\ 
    \end{array}
    \right)\! \!, \quad 
     \tilde{\mf A} (\mf u)=
    \left( 
    \begin{array}{ccc}
    u & \rho & 0 \\
    u^2+  p_\rho  & 2 \rho u &  p_\theta  \\
    u\psi+ p_\rho u & \rho\psi + \rho u^2+ p & \rho u e_\theta+ u p_\theta  \\ 
    \end{array}
    \right)  {,}
$$
where we have used the shorthand notation $\psi(\theta, u)= e+ u^2/2$.  
The expressions of $\tilde{\mf B}$ and $\tilde{\mf G}$ can be also explicitly computed. By left multiplying~\eqref{e:ql} times 
$$
    \mf S (\mf u)\! = 
    \left( 
    \begin{array}{ccc}
    p_\rho/\rho^2 & 0 & 0 \\
    -u/\rho & 1/\rho & 0 \\
    (u^2-\psi)/\rho \theta & -  u / \rho \theta & 1 /\rho \theta  \\ 
    \end{array}
    \right)
$$
and recalling that $p(\rho, \theta)= R \rho \theta${,} we arrive at~\eqref{e:symmetric} provided 
$$
   \mf E (\mf u)\! = \! \!
    \left( 
    \begin{array}{ccc}
    R \theta/\rho^2 & 0 & 0 \\
   0 & 1 & 0 \\
   0 & 0 & e_\theta/\theta  \\ 
    \end{array}
    \right) \!\!, \quad 
    \mf A (\mf u)\! = \! \!
    \left( 
    \begin{array}{ccc}
    R \theta u / \rho^2 & R \theta /\rho & 0 \\
    R \theta /\rho & u &  R \\
   0 & R  & e_\theta u / \theta   \\ 
    \end{array}
    \right) \!\!, \quad 
    \mf B (\mf u)\! = \! \!
    \left( 
    \begin{array}{ccc}
    0 & 0 & 0 \\
    0 & \nu /\rho & 0 \\
    0 & 0 & \kappa/\rho \theta    \\ 
    \end{array}
    \right)
$$
and
$$
   \mf G (\mf u, \mf u_x)\! = \! \!
    \left( 
    \begin{array}{ccc}
    0 & 0 & 0 \\
    0 & \nu'\rho_x /\rho & 0 \\
    0 & \nu u_x/\rho \theta & \kappa'\rho_x/\rho \theta    \\ 
    \end{array}
    \right).
$$
The eigenvalues of $\mf E^{-1} \mf A$ are $\lambda_1 = u -c$, $\lambda_2 =u$ and $\lambda_3 = u+c$, where $c = \sqrt{\theta R + \theta R^2 /e_\theta}$. 
By direct check one can verify that Hypotheses~\ref{h:normal},$\dots$,\ref{h:jde} are all satisfied provided that $N=3$, $h=1$, $u_1 = \rho$, $\mf u_2= (u, \theta)^t$, $k=2$, $\alpha (u) = u$ and $\mf u^\ast= (\rho^\ast, 0, \theta^\ast)$ for some $\rho^\ast, \theta^\ast >0$. 
The boundary condition $\boldsymbol{\beta} (\mf u(x=0), \mf u_b)= \mf 0_N$ translates as follows: given $\mf u_b=(\rho_b, u_b, \theta_b)$ with $\rho_b, \theta_b>0$ we first assign the boundary conditions $\mf u_2 = (u_b, \theta_b)^t$ at $x=0$. If $u_b > 0$, then we also assign the boundary condition $u_1=\rho_b$ at $x=0$, if $u_b \leq 0$ we do not.  
\subsection{Application to the MHD equations} \label{ss:h:mhd} We consider the equations describing the propagation of plane waves in magnetohydrodynamics, i.e. 
\be
\label{e:mhd}
\left\{
\begin{array}{lll}
\rho_t + (\rho u)_x =0 {,} \phantom{\displaystyle{\int}} \\
\mf b_t + (u \mf b - \beta \mf w)_x = ( \eta \mf b_x )_x {,} \\
  (\rho u)_t +  \displaystyle{\Big(\rho u^2 + p + \frac{1}{2} |\mf b|^2 \Big)_x} =  (\nu  u_x )_x {,} \\
(\rho \mf w)_t + (\rho u \mf w- \beta \mf b)_x =(\nu \mf w_x)_x {,} \phantom{\displaystyle{\int}}\\
\displaystyle{\Big( \rho \Big( \frac{1}{2} u^2 + \frac{1}{2} |\mf w|^2 + e  \Big) + \frac{1}{2} |\mf b|^2 \Big)_t
+ \Big( \rho u \Big( \frac{1}{2} u^2 + \frac{1}{2} |\mf w|^2 + e  \Big) + u \Big( p + |\mf b|^2 \Big) - \beta \mf b^t \mf w  
\Big)_x} \\
\displaystyle{\qquad \qquad = \big( \nu (u u_x + \mf w^t \mf w_x ) + \kappa \theta_x + \eta \mf b^t \mf b_x  \big)_x} {.} \\
\end{array}
\right.
\eq
The quantities $\rho, \theta, e, p, \nu, k$ have the same physical meaning as at the previous paragraph. The fluid velocity is $(u, \mf w)^t$, with $u \in \R$ and $\mf w \in \R^2$ (we assume that waves are propagating in the direction $(1, 0, 0)$) and the magnetic field is $(\beta, \mf b)$
 with $\beta \in \R$, $\beta \neq 0$ constant, and $\mf b \in \R^2$. Finally, $\eta \ge 0$ is the conductivity and in the following we separately consider the cases $\eta>0$ and $\eta=0$.
 
We now want to to write~\eqref{e:mhd} in the form~\eqref{e:symmetric}. First, we set $\mf u=(\rho, \mf b, u, \mf w, \theta)$ and we write~\eqref{e:mhd} in the form~\eqref{e:ql} for suitable matrices $\tilde{\mf E}$, $\tilde{\mf A}$, $\tilde{\mf B}$ and $\tilde{\mf G}$. 
Next, we left multiply~\eqref{e:ql} times    
$$
   \mf S(\mf u) = 
   \left( 
    \begin{array}{ccccc}
    R\theta/\rho^2     & \mf 0_2^t  & 0          & \mf 0_2^t & 0   \\
     \mf 0_2     & \mf I_2/\rho    & \mf 0_2          & \mf 0_{2 \times 2} & \mf 0_2  \\
   -u/\rho    & \mf 0_2^t    & 1/ \rho      & \mf 0_2^t & 0   \\
  -\mf w/\rho & \mf 0_{2 \times 2} & \mf 0_2 &  \mf I_2/\rho & \mf 0_{2}   \\
   (u^2+|\mf w|^2 - \psi)/ \rho \theta  & - \mf b^t/\rho \theta 
    & -u/\rho \theta & - \mf w^t / \rho \theta & 1/ \rho \theta    \\
    \end{array}
    \right) \! \!, 
$$ 
and we arrive at~\eqref{e:symmetric} provided 
$$
      \mf E(\mf u) 
   = \! \!
    \left( 
    \begin{array}{ccccc}
    R \theta/\rho^2    & \mf 0_2^t   & 0          & \mf 0_2^t & 0   \\
     \mf 0_2   & \mf I_2/\rho     & \mf 0_2          & \mf 0_{2 \times 2} & \mf 0_2   \\
   0        & \mf 0_2^t & 1       & \mf 0_2^t & 0  \\
   \mf 0_2  & \mf 0_{2 \times 2}  & \mf 0_2  & \mf I_{2}  & \mf 0_2\\
    0  &\mf 0_2^t & 0 &\mf 0_2^t & e_\theta/\theta   \\
    \end{array}
    \right)\! \!, \quad
    \mf A(\mf u) 
   = \! \!
    \left( 
    \begin{array}{ccccc}
    R \theta u /\rho^2   & \mf 0_2^t     & R \theta/\rho          & \mf 0_2^t & 0  \\
        \mf 0_2     & u \mf I_2/\rho   & \mf b /\rho         & - \beta \mf I_{2} & \mf 0_2   \\
   R \theta /\rho    & \mf b^t/\rho   & u        & \mf 0_2^t & R   \\
   \mf 0_2 & - \beta \mf I_{2} & \mf 0_2  & u \mf I_{2}  & \mf 0_2  \\
    0  &\mf 0_2^t  & R &\mf 0_2^t & u e_\theta /\theta \\
    \end{array}
    \right)  
$$
and
$$
   \mf B(\mf u) 
   = \! \!
    \left( 
    \begin{array}{ccccc}
   0      & \mf 0_2^t & 0          & \mf 0_2^t & 0   \\
   \mf 0_2  & \eta \mf I_2/\rho      & \mf 0_2          & \mf 0_{2 \times 2} & \mf 0_2   \\
   0       & \mf 0_2^t  & \nu/\rho       & \mf 0_2^t & 0  \\
   \mf 0_2 & \mf 0_{2 \times 2} & \mf 0_2  & \nu \mf I_{2}/\rho  & \mf 0_2  \\
    0 &\mf 0_2^t  & 0 &\mf 0_2^t & \kappa/\rho \theta   \\
    \end{array}
    \right)\! \!, \quad 
    \mf G(\mf u, \mf u_x) 
   = \! \!
    \left( 
    \begin{array}{ccccc}
   0      & \mf 0_2^t & 0          & \mf 0_2^t & 0   \\
   \mf 0_2  &\mf 0_{2 \times 2}      & \mf 0_2          & \mf 0_{2 \times 2} & \mf 0_2   \\
   0       & \mf 0_2^t  & \nu' \rho_x/ \rho       & \mf 0_2^t & 0  \\
   \mf 0_2 & \mf 0_{2 \times 2} & \mf 0_2  & \nu' \rho_x \mf I_{2}/\rho  & \mf 0_2  \\
    0 & \eta \mf b^t_x/\rho \theta  & \nu u_x /\rho \theta & \nu \mf w^t_x/\rho \theta & \kappa' \rho_x/\rho \theta   \\
    \end{array}
    \right)\! \!. \quad 
$$
If $\mf b \neq \mf 0_2$, then the matrix $\mf E^{-1} \mf A$ has 7 real distinct eigenvalues that in a neighborhood  of $u=0$ satisfy $\lambda_1(\mf u) < \lambda_2 (\mf u), \lambda_3 (\mf u)<0$, $\lambda_4(\mf u) =u$, 
  $0< \lambda_5(\mf u) < \lambda_6 (\mf u)< \lambda_7 (\mf u)$, see~\cite[Volume 1, p.16]{Serre1} for a related discussion. This implies that Hypothesis~\ref{h:sh} is satisfied and $N=7$, $k=4$ in~\eqref{e:sh}. To verify the other hypotheses, we separately consider two distinct cases. 
\subsubsection{Case $\eta>0$}  
In this case $h=1$ and Hypotheses~\ref{h:normal} and~\ref{h:ks} are satisfied. Note that 
\be
\label{e:a21mhd}
   \mf a_{21}= \big(  \mf 0_2^t     , R \theta/\rho          , \mf 0_2^t , 0 \big)^t {.}
\eq
We set $\alpha (\mf u)=u$ and we obtain~\eqref{e:sing1d} and that Hypothesis~\ref{h:jde} is satisfied. 
Note that any point $\mf u^\ast=(\rho^\ast, \mf b^\ast, 0, \mf w^\ast, \theta^\ast)$, $\rho^\ast, \theta^\ast>0$, $\mf b^\ast \neq \mf 0_2$ satisfies~\eqref{e:uast}. Note that the boundary condition $\boldsymbol{\beta} (\mf u(x=0), \mf u_b)= \mf 0_N$ translates as follows: given $\mf u_b=(\rho_b, \mf b_b, u_b, \mf w_b, \theta_b)^t$ with $\rho_b, \theta_b>0$, $\mf b_b \neq 0$ we first assign the boundary conditions $\mf u_2 = (\mf b_b, u_b, \mf w_b, \theta_b)^t$ at $x=0$. If $u_b > 0$, then we also assign the boundary condition $u_1=\rho_b$ at $x=0$, if $u_b \leq 0$ we do not.  
\subsubsection{Case $\eta=0$}  \label{sss:h:mhdeta0}
In this case $h=3$ and Hypotheses~\ref{h:normal} and~\ref{h:ks} are satisfied. Note that 
\be
\label{e:blocchimhd}
    \mf E_{11}= 
    \left( 
    \begin{array}{ccccc}
    R \theta/\rho^2    & \mf 0_2^t      \\
     \mf 0_2   & \mf I_2/\rho  \\
     \end{array}\right) \! \!, \quad 
   \mf A_{21}= \left( \begin{array}{ccc}
     R \theta /\rho    & \mf b^t/\rho   \\
   \mf 0_2 & - \beta \mf I_{2}  \\
    0  &\mf 0_2^t   \\
    \end{array}
    \right)\! \!, \quad 
   \mf A_{22}= \left( \begin{array}{ccc}
      u        & \mf 0_2^t & R   \\
    \mf 0_2  & u \mf I_{2}  & \mf 0_2  \\
   R &\mf 0_2^t & u e_\theta /\theta \\
   \end{array}    \right)
\eq 
and that Hypothesis~\ref{h:sing1d} is satisfied with $\alpha (\mf u) =u$. This implies that Hypothesis~\ref{h:jde} holds and that any point $\mf u^\ast=(\rho^\ast, \mf b^\ast, 0, \mf w^\ast, \theta^\ast)$, with $\rho^\ast, \theta^\ast>0$, $\mf b^\ast \neq \mf 0_2$ satisfies~\eqref{e:uast}. 
Finally, the boundary condition $\boldsymbol{\beta} (\mf u(x=0), \mf u_b)= \mf 0_N$ translates as follows: given $\mf u_b=(\rho_b, \mf b_b, u_b, \mf w_b, \theta_b)^t$ with $\rho_b, \theta_b>0$, $\mf b_b \neq 0$ we first assign the boundary conditions $\mf u_2 = (u_b, \mf w_b, \theta_b)^t$ at $x=0$. If $u_b > 0$, then we also assign the boundary condition $\mf u_1=(\rho_b, \mf b_b)$ at $x=0$, if $u_b \leq 0$ we do not.  
\section{The nonsingular manifold $\mathcal M^0$}
\label{ss:emme0}
In this section we first derive the ODE satisfied by the boundary layers and the traveling waves of system~\eqref{e:symmetric}, which in general is singular. 
This is done in~\S~\ref{ss:twbl}. Next, in~\S~\ref{ss:m0} we construct the manifold $\mathcal M^0$ and we show that by restricting the traveling waves and boundary layers equations on this manifold we obtain a nonsingular system. Finally, in~\S~\ref{ss:3:nsmhd} we explicitly discuss how the analysis in this section applies to the Navier-Stokes and MHD equations.  
Note that in this section we focus on the case where $h=1$, i.e. we assume that the kernel of $\mf B$ is 
one-dimensional, and we refer to~\S~\ref{s:bigk} for the extension to the case $h>1$. 
\subsection{Traveling waves and boundary layers of~\eqref{e:symmetric}}
\label{ss:twbl}
Given $\sigma \in \R$, we consider a \emph{traveling wave} solution of~\eqref{e:symmetric}, which is a function of one variable only and satisfies $\mf u(t, x) = \mf u(x - \sigma t)$. If $\sigma =0$, $\mf u$ is actually a steady solution of~\eqref{e:symmetric} and we term it \emph{boundary layer}. 
Note that traveling waves and boundary layers satisfy the ODE
\be
\label{e:bltw}
    - \sigma \mf E(\mf u) \mf u' + \mf A(\mf u) \mf u' =  \mf B(\mf u) \mf u'' +
    \mf G (\mf u, \mf u' )  \mf u'. 
\eq 
We now recall Remark~\ref{r:acca1} and we point out that if $h=1$ the hyperbolic component $\mf u_1$ in~\eqref{e:blocku} is actually scalar. In the following we denote it by $u_1$, in such a way that that the ODE~\eqref{e:bltw} can be rewritten as  
$$
\left\{
\begin{array}{ll}
\big[ \alpha  - \sigma  \big]e_{11} u_1'+ \mf a_{21}^t \mf u_2' =0 {,} \\
\mf a_{21} u_1'+ \big[ \mf A_{22} - \sigma \mf E_{22} \big]
 \mf u_2' - 
\mf g_1 u_1' - \mf G_{2} \mf u_2'  = \mf B_{22} \mf u_2''.   
\end{array}
\right.
$$
We couple the above equation with the condition $\sigma'=0$, we recall that $e_{11} > 0$, we assume $\alpha  - \sigma \neq 0$  
and by solving the first equation for $u'_1$ and setting $\mf u_2':= \mf z_2$ we arrive at 
\begin{equation}
\label{e:odeslow}
    \mf v' = \frac{1}{\alpha (\mf u)- \sigma} \mf h (\mf v),  
    \end{equation}
provided that
\be
\mf v : =
\left(
\begin{array}{cc}
u_1 \\
 \mf u_2 \\
\mf z_2 \\
\sigma \\
\end{array}
\right)
\eq
and
\be
\label{e:cosasono1}
\mf h( \mf v) : =
\left(
\begin{array}{ccc}
- e_{11}^{-1} \mf a^t_{21} \mf z_2 \\
\big[  \alpha - \sigma    \big] \mf z_2 \\
\mf B_{22}^{-1} \Big( [\alpha-\sigma ]
\big[ \mf A_{22} - \sigma \mf E_{22} - \mf G_{2} \big] - \, e^{-1}_{11} \big[ \mf a_{21} \mf a_{21} 
^t - \mf g_1  \mf a_{21}^t \big]
\Big) \mf z_2  \\
0 \\
\end{array}
\right)
\eq
Note that the above equation is singular when $\alpha (\mf u) = \sigma$, in particular the equation of the boundary layers is singular when $\alpha (\mf u)=0$. This happens at $\mf u^\ast$ owing to~\eqref{e:uast}. 
\subsection{The manifold $\mathcal M^0$} \label{ss:m0}
We fix $\mf u^\ast$ satisfying~\eqref{e:uast}. We consider the ODE~\eqref{e:odefast}, which is \emph{formally} obtained from~\eqref{e:odeslow} through the change of variable $x = (\alpha(\mf u) - \sigma) y$, and we linearize it at the {equilibrium} point $\mf v^\ast= (\mf u^\ast, \mf 0_{N-1}, 0)^t$. We apply Lemma~\ref{l:signature} in the case $h=1$ and we conclude that the center space (i.e., the eigenspace associated to the eigenvalues with $0$ real part) of the Jacobian matrix $\mf D \mf h (\mf v^\ast)$ is given by 
$$
    M^0 : = \big\{ (u_1, \mf u_2, \mf z_2, \sigma) \in \R^{2N}: \ \mf a^t_{21}(\mf u^\ast) \mf z_2 =0 \big\}. 
$$
Owing to~\eqref{l:kaw}, $\mf a_{21}(\mf u^\ast) \neq \mf 0_{N-1}$ and hence the dimension of $M^0$ is $2N-1$. We now apply the Center Manifold Theorem and we refer to~\cite{Bressan:cm} for the statement and the proof. We recall that the center manifold is not unique in general: we arbitrarily fix one and we term it $\mathcal M^0$. We recall that $\mathcal M^0$ is defined in a neighbourhood of $\mf v^\ast$. The following result describes the structure of $\mathcal M^0$. 
\begin{lemma}
\label{l:slow} There are a sufficiently small constant $\delta>0$ and 
a smooth function
\begin{equation*}
\mf R_0: \R^N \times \R^{N-2} \times \R \to \mathbb{M}^{(N-1) \times (N-2)}
\end{equation*}
such that   
\begin{equation}
\label{e:errezero}
     (u_1, \mf u_2, \mf z_2, \sigma ) \in \mathcal M^0 \cap
     \mathrm{B}_\delta^{2N} (\mf v^\ast)  \iff  \mf z_2 = \mf R_0 (\mf u, \mf z_0, \sigma)\mf z_0 \quad \text{for some $\mf z_0  \in \R^{N-2}$},
\end{equation}
the columns of $\mf R_0$ are linearly independent vectors and  
\begin{equation}
\label{e:identity}
  \mf R_0^t \mf B_{22} \mf R_0 (\mf u^\ast, \mf 0_{N-2}, 0)= \mf I_{N-2}. 
\end{equation}
\end{lemma}
\begin{proof}
The proof extends the argument in~\cite[pp. 240-241]{BianchiniBressan} and is based on the statement of the Center Manifold Theorem given in~\cite{Bressan:cm}. First, we term $\tilde M_0$ the subspace of $\R^{N-1}$ containing all the vectors orthogonal to $\mf a_{21}(\mf u^\ast)$: this means that
$M_0 = \R^N \times \tilde M_0 \times \R$. We recall that the center manifold $\mathcal M^0$ is parametrized by a suitable map $\mf m_0: M^0 \to \mathcal M^0$. Also, we can construct the map $\mf m_0$ in such a way that the composition of $\mf m_0$ with the orthogonal projection onto $M^0$ is the identity. This implies that  
$$
\mf m_0 (\mf u, \tilde{\mf z}_0, \sigma) = (\mf u, \tilde{\mf m}_0 (\mf u, \tilde{\mf z}_0, \sigma), \sigma) 
$$ 
for some suitable function $\tilde{\mf m}_0:  \R^N \times \tilde M_0 \times \R \to \R^{N-1}$. 
Next, by applying the Gram-Schmidt orthonormalization we choose a basis of $\tilde M_0$ that is orthonormal with respect to the scalar product defined by the symmetric and positive definite matrix $\mf B_{22}$. By a slight abuse of notation, in the following we identify a vector 
$\tilde{\mf z}_0 \in \tilde M_0 \subseteq \R^{N-1}$ with the vector $\mf z_0 \in \R^{N-2}$ of its coordinates with respect to the orthonormal basis and hence we regard $\tilde{\mf m}_0$ as a map depending on ${(\mf u, {\mf z}_0, \sigma) \in \R^N \times \R^{N-2} \times \R}$. 

We now recall the explicit expression~\eqref{e:cosasono1} of $\mf h$ and we conclude that every point $(\mf u, \mf 0_{N-1}, \sigma)$ is an equilibrium for~\eqref{e:odefast}. By definition of center manifold, this implies that $(\mf u, \mf 0_{N-1}, \sigma) \in \mathcal M^0$ provided that it is sufficiently close to $\mf v^\ast$. This implies that $\tilde{\mf m}_0 (\mf u, \mf 0_{N-2}, \sigma)= \mf 0_{N-1}$ for every $(\mf u, \sigma)$ sufficiently close to $(\mf u^\ast, 0)$. 
By applying Corollary~\ref{c:dividiamo} with $\mf f: = \tilde{\mf m}_0$ and $\mf  y: = {\mf z}_0$  we conclude that $ \tilde{\mf m}_0 = \mf \mf \mf R_0 (\mf u, {\mf z}_0, \sigma){\mf z}_0$ for a suitable function $\mf R_0$ attaining values in the space $\mathbb M^{(N-1) \times (N-2)}$. We recall that the manifold $\mathcal M^0$ is tangent to $M^0$ at $\mathcal M^0$ and we conclude that the columns of $\mf R_0 (\mf u^\ast, \mf 0_{N-2}, 0)$ are $N-2$ linearly independent vectors. To conclude the proof, we are left to show that the columns of $\mf R_0 (\mf u, \mf z_0, \sigma)$ are linearly independent vectors: since this it true at $(\mf u^\ast, \mf 0_{N-2}, 0)$, then it is also true in a sufficiently small neighborhood.
\end{proof}
We now investigate the solution of~\eqref{e:odefast} lying on $\mathcal M^0$. We first establish two preliminary lemmas. 
\begin{lemma}
\label{l:fazero}
If $\alpha (\mf u)= \sigma$, then 
\begin{equation}
\label{e:a21r0}
\mf a_{12}^t(\mf u) \mf R_0 (\mf u, \mf z_0, \sigma) = \mf 0_{N-2} \quad \text{for every $\mf z_0$}.
\end{equation}
\end{lemma}
\begin{proof}
Assume that {$\alpha (\mf u) =\sigma$} and consider the subsets of $\R^{N-1}$ defined by setting  
$$
   Z : = \big\{ \mf z_2: \mf a_{12}^t(\mf u) \mf z_2 =
   0 \big\} 
   \quad \text{and} \quad 
   \tilde Z : = \big\{ \mf z_2:  \mf z_2 = \mf R_0(\mf u, \mf z_0, \sigma) \mf z_0 \; \text{for some $\mf z_0$} \big\}. 
$$
We first show that $Z \subseteq \tilde Z.$ Fix $\mf z_2 \in Z$, then $\mf v: =(\mf u, \mf z_2, \sigma)$ satisfies $\mf h(\mf v)= \mf 0_{2N}$ provided $\mf h$ is the same as in~\eqref{e:cosasono1}. In other words, $\mf v$ is an equilibrium for the ODE~\eqref{e:odefast}, which implies that $\mf v \in \mathcal M^0$, provided that $\mf v$ is sufficiently close to $(\mf u^\ast, \mf 0_{N-2}, 0)$. This implies that 
$\mf z_2 =   \mf R_0(\mf u, \mf z_0, \sigma) \mf z_0$ for some $\mf z_0 \in \R^{N-2}$, namely that $\mf z_2 \in \tilde Z$. 

Owing to Lemma~\ref{l:kaw}, $\mf a_{12} (\mf u)  \neq \mf 0_{N-1}$ and hence the subspace $Z$ has dimension $N-2$. On the other hand, the columns of $\mf R_0(\mf u, \mf z_0, \sigma)$ are also linearly independent vectors by Lemma~\ref{l:slow} and hence the dimension of $\tilde Z$ is also $(N-2)$ . Since $Z \subseteq \tilde Z$, we conclude that $Z = \tilde Z$.   This yields~\eqref{e:a21r0}. 
\end{proof}
\begin{corol}
\label{c:definiamod}
There is a smooth function $\mf d$ attaining values in $\R^{N-2}$ such that
\be
\label{e:definiamod}
   \mf a_{12}^t  \mf R_0 (\mf u, \mf z_0, \sigma) 
   \mf z_0 = \big[
   \alpha  (\mf u)  - \sigma   \big] 
    \mf d^t( \mf u, \mf z_0, \sigma) \ \mf z_0, \quad \text{for every $\mf u$, $\mf z_0$, $\sigma$.}
\eq 
\end{corol}
\begin{proof}
Owing to Hypothesis~\ref{h:jde}, Lemma~\ref{l:kaw} and Lemma~\ref{l:fazero}, the functions $f:= \mf a_{12}^t  \mf R_0 \mf z_0$ and $a: = \alpha - \sigma $
satisfy the hypotheses of Lemma~\ref{l:dividiamo}. This implies that 
$$
    \mf a_{12}^t (\mf u) \mf R_0 \mf z_0 = 
    \big[ \alpha (\mf u) - \sigma \big] h ( \mf u, \mf z_0, \sigma)     
$$
for a suitable scalar function $h$. Since $h ( \mf u, \mf 0, \sigma)=0$, then owing again to Lemma~\ref{l:dividiamo} we get $h ( \mf u, \mf z_0, \sigma)= \mf d^t  ( \mf u, \mf z_0, \sigma) \mf z_0$ for a suitable  function $\mf d$.
\end{proof}
We are eventually able to provide the equation satisfied by the solutions of~\eqref{e:odefast} lying on $\mathcal M^0$. We recall that~\eqref{e:odefast} is \emph{formally} obtained from~\eqref{e:odeslow} through the change of variable $x = (\alpha(\mf u) - \sigma) y$, that $\mf v'$ denotes the derivative with respect to $x$ and that $\dot{\mf v}$ denotes the derivative with respect to $y$.
\begin{lemma}
\label{l:suemme0}
There is a sufficiently small constant $\delta>0$ such that by restricting system~\eqref{e:odefast} to 
$\mathcal M^0 \cap \mathrm{B}^{2N}_\delta (\mf v^\ast)$ we get 
\be
\label{e:suemme0}
\left\{
\begin{array}{lll}
\dot{u_1} =   - e^{-1}_{11} (\mf u)[\alpha (\mf u)- \sigma ]  \mf d^t (\mf u, \mf z_0, \sigma) \mf z_0 { ,} \\
\dot{\mf u}_2 =[\alpha (\mf u)- \sigma] \mf R_0  (\mf u, \mf z_0, \sigma) \mf z_0 {,} \\
\dot{\mf z}_0 = [\alpha (\mf u)- \sigma ]  \mf \Theta_0  (\mf u, \mf z_0, \sigma) \mf z_0, \\ 
\dot{\sigma} = 0 {,}\\
\end{array}
\right.
\eq
for a suitable smooth function $\mf \Theta_0$ that attains values in $\mathbb{M}^{(N-2)\times (N-2)}$ 
and satisfies 
\begin{equation}
\label{e:Lambda0}
      \mf  \Theta_0  (\mf u^\ast, \mf 0_{N-2}, 0) = 
       \mf R_0^t 
       \mf A_{22} \mf R_0 (\mf u^\ast, \mf 0_{N-2},0).
\end{equation}
\end{lemma}
\begin{proof}
By plugging the relation $\mf z_2 = \mf R_0 \mf z_0$ into the first two lines of~\eqref{e:odefast} and using~\eqref{e:cosasono1} and~\eqref{e:definiamod} we arrive at the first two lines of~\eqref{e:suemme0}. To get the {third} line, we plug the relation $\mf z_2 = \mf R_0 \mf z_0$ into the last line of~\eqref{e:odefast}, we use~\eqref{e:definiamod} and we arrive at  
\be 
\label{e:biduedue}
\begin{split}
\mf B_{22} \dot{\mf z}_2 & = 
\mf B_{22}\Big(  \dot{\mf R}_0 \mf z_0 + \mf R_0 \dot{\mf z}_0 \Big)  =
[\alpha-\sigma ] \Big(  
[ \mf A_{22} - \sigma \mf E_{22} - \mf G_{2}] \mf R_0  - e_{11}^{-1}\mf a_{21}
 \mf d^t  + e_{11}^{-1} 
\mf g_1  \mf d^t \Big) \mf z_0
\end{split}
\eq 
Next, we recall that $\mf R_0$ depends on $\mf u$, $\mf z_0$ and $\sigma$, that $\dot{\mf u}$ is given by the first two lines of~\eqref{e:suemme0} (i.e. it is proportional to both $\alpha - \sigma$ and $\mf z_0$) and that $\dot{\sigma}=0$. We conclude that 
\be
\label{e:errezerodot}
   \mf R_0^t \mf B_{22} \dot{\mf R}_0 \mf z_0 = 
   [\alpha - \sigma  ] \mf H_1 (\mf u, \mf z_0, \sigma) \mf z_0 + \mf H_2 
   (\mf u, \mf z_0, \sigma) \dot{\mf z}_0
\eq  
for some functions $\mf H_1$, $\mf H_2$ attaining values in $\mathbb M^{(N-2) \times (N-2)}$ and satisfying 
\be
\label{e:accaunodue}
     \mf H_1 (\mf u,  \mf 0_{N-2}, \sigma)= \mf 0_{(N-2) \times (N-2)}, \quad  
    \mf H_2  
    (\mf u, \mf 0_{N-2}, \sigma) = \mf 0_{(N-2) \times (N-2)} \quad 
    \text{for every $ \mf u$ and $\sigma$.}
\eq
By left multiplying~\eqref{e:biduedue} times $\mf R_0^t$ and using~\eqref{e:definiamod} and~\eqref{e:errezerodot} we arrive at 
\be \label{e:raccogliamo}
 \begin{split}
 \big( \mf R_0^t & \mf B_{22} \mf R_0 + \mf H_2 \big) \dot{\mf z}_0
 \\ & = 
 [\alpha-\sigma] \big(  \mf R_0^t 
[ \mf A_{22} - \sigma \mf E_{22} - \mf G_{2}] \mf R_0  - e_{11}^{-1} [\alpha-\sigma]
\mf d \mf d^t  + e_{11}^{-1}
\mf R_0^t \mf g_1 \mf d^t - \mf H_1\big) \mf z_0
 \end{split}
 \eq 
Next, we point out that, owing to~\eqref{e:identity} and~\eqref{e:accaunodue}, we have 
$
   \big( \mf R_0^t \mf B_{22} \mf R_0 + \mf H_2 \big) (\mf u^\ast, \mf 0_{N-2}, 0) = \mf I_{N-2}. 
$
By continuity, the matrix-valued function $ \mf R_0^t \mf B_{22} \mf R_0 + \mf H_2$ is nonsingular in a sufficiently small neighbourhood  of $(\mf u^\ast, \mf 0_{N-2}, 0)$. This implies that we can solve~\eqref{e:raccogliamo} for $\dot{\mf z}_0$ and arrive at the third line of~\eqref{e:suemme0} provided that 
\be \label{e:cosaeLambda0}
   \mf \Theta_0 (\mf u, \mf z_0, \sigma) : =
    \big( \mf R_0^t \mf B_{22} \mf R_0 + \mf H_2 \big)^{-1}
   \big(  \mf R_0^t 
[ \mf A_{22} - \sigma \mf E_{22} - \mf G_{2}] \mf R_0  -  [\alpha-\sigma] e_{11}^{-1}
\mf d \mf d^t  + e_{11}^{-1}
\mf R_0^t \mf g_1 \mf d^t - \mf H_1\big)
\eq 
To establish~\eqref{e:Lambda0} we recall the expression of $\mf g_1$ and $\mf G_2$ (see property iv) in Hypothesis~\ref{h:normal}) and we use~\eqref{e:uast},~\eqref{e:identity} and~\eqref{e:accaunodue}.
\end{proof}
Lemma~\ref{l:suemme0} implies that if we restrict the singular system to $\mathcal M^0$ we obtain a nonsingular system. 
\begin{lemma}
\label{l:invariant}
The manifold $\mathcal M^0$ is locally invariant for~\eqref{e:odeslow} and by restricting~\eqref{e:odeslow} to $\mathcal M^0$ we get 
\be
\label{e:emme0slow}
\left\{
\begin{array}{lll}
{u_1}' =  - e_{11}^{-1} \mf d^t (\mf u, \mf z_0, \sigma) \mf z_0 \\
\mf u'_2 =  \mf R_0  (\mf u, \mf z_0, \sigma) \mf z_0 \\
\mf z'_0=   \mf \Theta_0  (\mf u, \mf z_0, \sigma) \mf z_0 \\ 
\sigma'=0. \\
\end{array}
\right.
\eq
Also, if $(\mf u, \mf z_0, \sigma)$ is a solution of~\eqref{e:emme0slow} and $\alpha (\mf u) \neq \sigma$ at $x=0$, then $\alpha (\mf u) \neq \sigma$ for every $x$. 
\end{lemma}
\begin{proof}
Consider a solution $(\mf u, \mf z_0, \sigma)$ of system~\eqref{e:emme0slow}. We introduce the function $a$ by setting $a(x): = \alpha (\mf u(x))$. By combing Hypothesis~\ref{h:jde} with~\eqref{e:definiamod} we conclude that $a'(x)=0$ if $a(x)=\sigma$. This implies that $a(x)=\sigma$ is an equilibrium for the ODE satisfied by $a$.  By the uniqueness part of the Cauchy Lipschitz Picard Lindel\"of Theorem, either $a$ is identically $\sigma$, or it is always different from $\sigma$. This establishes the last statement of the lemma. 

Next, we fix an initial datum for~\eqref{e:odeslow} lying on $\mathcal M^0$, we recall~\eqref{e:errezero} and conclude that the initial datum must be of the form $(\tilde{\mf u}, \mf R_0 \tilde{\mf z}_0, \tilde \sigma)$.  Consider the solution of $(\mf u, \mf z_0, \sigma)$ of system~\eqref{e:emme0slow} with initial datum $(\tilde{\mf u}, \tilde{\mf z}_0, \tilde \sigma)$. 	By the proof of Lemma~\ref{l:suemme0}, the function $(\mf u, \mf R_0 \mf z_0, \sigma)$ is a solution of~\eqref{e:odeslow}. This establishes the first part of the lemma.  
\end{proof}
To conclude this section, we state a result concerning the signature of  $\mf R_0^t \mf A_{22} \mf R_0$. Note that $\mf R_0^t \mf A_{22} \mf R_0$ is a symmetric matrix and therefore has $N-2$ real (non necessarily distinct) eigenvalues. 
\begin{lemma}
\label{l:key}
Assume that $\mf u^\ast$ satisfies~\eqref{e:uast}.  Then $2 \leq k \leq N-1$ and the signature of the matrix $\mf R_0^t \mf A_{22} \mf R_0 (\mf u^\ast, \mf 0_{N-2}, 0) $ is as follows: 
\begin{itemize}
\item $1$ eigenvalue is $0$;
\item $k-2$ eigenvalues are strictly negative;
\item $N-k-1$ eigenvalues are strictly positive, 
\end{itemize}
provided each eigenvalue is counted according to its multiplicity. 
\end{lemma}
The proof of Lemma~\ref{l:key} is quite long and technical and is given in~\S~\ref{s:proof}. 
\subsection{Application to Navier-Stokes and MHD equations} \label{ss:3:nsmhd}
We now discuss the application of the analysis in~\S~\ref{ss:emme0} to the Navier-Stokes and MHD equations with $\eta>0$.  For the applications to the MHD equations with $\eta=0$ we refer to~\S~\ref{s:bigk}.
\subsubsection{Navier-Stokes equations}\label{sss:3:ns}
We recall the discussion in~\S~\ref{ss:h:ns} and we point out that  $\mf u=(\rho, u, \theta)$, $u_1= \rho$ and $\mf u_2 =(u, \theta)^t$. Since $N=3$, 
the dimension of the manifold $\mathcal M^0$ is $5$, $\mf z_0$ is a real valued function, $\mf R_0$ attains value in $\R^2$ and it is perpendicular to $\mf a_{21}$ at $(\mf u^\ast, 0, 0)$ owing to~\eqref{e:definiamod}. By recalling~\eqref{e:identity} we conclude that 
$\mf R_0(\mf u^\ast, 0, 0)~=~(0, \sqrt{\rho^\ast \theta^\ast/k(\rho^\ast)})^t$. The function $\mf \Theta_0$ is real valued and attains the value $0$ at $(\mf u^\ast, 0, 0)$, which is consistent with Lemma~\ref{l:key} since in this case $k-2=0$ and $N-k-1=0$.  
\subsubsection{MHD equations with $\eta>0$}\label{sss:3:mhd}
We recall the discussion in~\S~\ref{ss:h:mhd}. Note that  $\mf u=(\rho, \mf b,  u, \mf w, \theta)$, $u_1= \rho$ and $\mf u_2 =(\mf b,  u, \mf w, \theta)^t$.
Since $N=7$, 
the dimension of the manifold $\mathcal M^0$ is $13$, $\mf z_0$ attains values in $\R^5$ and $\mf R_0$ attains value in $\mathbb{M}^{6\times5}$ and its columns are all perpendicular to $\mf a_{21}$ at $(\mf u^\ast, 0, 0)$ owing to~\eqref{e:definiamod}. By recalling~\eqref{e:identity} we conclude that 
$$
    \mf R_0 (\mf u^\ast, 0, 0) = \left(
    \begin{array}{cccccc}
     \sqrt{\rho^\ast/\eta} \mf I_{2}  & \mf 0_{2 \times 2} & \mf 0_2 \\
     \mf 0^t_{2} & \mf 0_2^t &  0 \\
       \mf 0_{2} & \sqrt{\rho^\ast / \nu} \mf I_2 & \mf 0_2 \\  
     \mf 0_2^t & \mf 0_2^t & \sqrt{\rho^\ast \theta^\ast/\kappa)} \\ 
     \end{array}
    \right)
$$
The function $\mf \Theta_0$ attains the value in $\mathbb{M}^{5 \times 5}$ and by using~\eqref{e:Lambda0} we get 
\be 
\label{e:mhd:Theta}
     \mf \Theta_0 (\mf u^\ast, 0, 0)= 
     \left(
     \begin{array}{ccc}
      \mf 0_{2 \times 2}           & - \beta \rho^\ast \mf I_{2}/ \sqrt{\eta \nu}  & \mf 0_2   \\
   - \beta \rho^\ast \mf I_{2}/\sqrt{\eta \nu}    &\mf 0_{2 \times 2}  & \mf 0_2  \\
  \mf 0_2^t  &\mf 0_2^t & 0
     \end{array}
     \right) 
 \eq
Note that the eigenvalues of $ \mf \Theta_0 (\mf u^\ast, 0, 0)$ are: $0$ (with multiplicity $1$), $- \beta \rho^\ast/ \sqrt{\eta \nu}$ (with multiplicity $2$) and $\beta \rho^\ast/ \sqrt{\eta \nu}$ (with multiplicity $2$). This is consistent with Lemma~\ref{l:key} since in this case  $k=4$. 
\section{Characteristic boundary layers}
\label{s:center}
In this section, we study the characteristic boundary layers, i.e. the boundary layers that decay very slowly to their limit. Note that we have to take them into account because the boundary is characteristic for the hyperbolic system~\eqref{e:hyperbolic}, i.e. one eigenvalue of $\mf E^{-1}\mf A$ can attain the value $0$, see~\eqref{e:uast}. Note that when we handle characteristic boundary layers we have to simultaneously handle travelling waves. We proceed as follows: in~\S~\ref{ss:cma} we construct the center manifold of system~\eqref{e:odeslow} restricted on $\mathcal M^0$. As we point out in~\S~\ref{sss:4:ns}, this construction is trivial in the case of the Navier-Stokes equations, but it is in general non trivial. For instance, it is not trivial in the case of the MHD equations, see~\S~\ref{sss:4:mhd}. In~\S~\ref{ss:ld} we discuss the characteristic boundary layers analysis by assuming that the characteristic vector field is linearly degenerate. This assumption considerably simplifies the analysis and it is satisfied by the Navier-Stokes and MHD equations. In~\S~\ref{ss:general} we discuss the general case and in~\S~\ref{ss:4:nsmhd} we describe the applications of the analysis to the Navier-Stokes and MHD equations.  As in the previous section, here we focus on the case $h=1$ and we refer to~\S~\ref{s:bigk} for the case $h>1$. 
\subsection{Center manifold analysis}
\label{ss:cma}
In this paragraph we construct a manifold containing the boundary layers with characteristic speed. We linearize system~\eqref{e:emme0slow} at the point $(\mf u^\ast, \mf 0_{N-2}, 0)$. Owing to Lemma~\ref{l:key}, the center space has dimension $N+2$. We arbitrarily select a center manifold and we term it $\mathcal M^{00}$. Note the difference between the manifold $\mathcal M^0$ and the manifold $\mathcal M^{00}$: $\mathcal M^{0}$ is a center manifold for system~\eqref{e:odefast}, $\mathcal M^{00}$ is a center manifold for system~\eqref{e:emme0slow} and henceforth for system~\eqref{e:odeslow}. The proof of the following result is similar to, but easier than, the proof of Lemma~\ref{l:slow} and it is therefore omitted. See also~\cite[\S4]{BianchiniBressan}. 
\begin{lemma}
\label{l:center} 
There are a sufficiently small constant $\delta>0$ and 
a smooth function $\mf r_{00}: \R^N \times \R \times \R \to \R^{N-2}$ such that   
\begin{equation}
\label{e:erre00}
     ( \mf u, \mf z_0, \sigma ) \in \mathcal M^{00} \cap \mathrm{B}^{2N-1}_{\delta}(\mf u^\ast, \mf 0_{N-1}, 0) 
     \iff  \mf z_0 = \mf r_{00} (\mf u, z_{00}, \sigma) z_{00}, \quad 
     \text{for a suitable $z_{00} \in \R$} .  
\end{equation}
Also, 
\begin{equation}
\label{e:eigen}
  \mf R_0^t \mf A_{22} \mf R_0 (\mf u^\ast, \mf 0, 0)
  \mf r_{00} (\mf u^\ast,0, 0) = \mf 0_{N-2} \quad 
  \text{and} 
  \quad 
  |\mf r_{00} (\mf u^\ast,0, 0)| =1. 
\end{equation}
\end{lemma}
 We now restrict system~\eqref{e:emme0slow} to $\mathcal M^{00}$. 
\begin{lemma}
\label{l:suemme00}
By restricting system~\eqref{e:emme0slow} (and henceforth system~\eqref{e:odeslow}) to~the manifold $\mathcal M^{00}$ we get 
\be
\label{e:suemme00}
\left\{
\begin{array}{lll}
{u_1}' =  -   e_{11}^{-1}  \mf d^t  \mf r_{00} (\mf u, z_{00}, \sigma) z_{00} { ,} \\
\mf u'_2 =  \mf R_0  \mf r_{00} (\mf u, z_{00}, \sigma) z_{00} { ,} \\
z'_{00}=    \theta_{00}  (\mf u, z_{00}, \sigma) z_{00} { ,} \\ 
\sigma'=0 {.} \\
\end{array}
\right.
\eq
In the previous expression, $\theta_{00}: \R^N \times \R \times \R \to \R$ is a suitable smooth function satisfying 
\be
\label{e:lambda00}
      \theta_{00} (\mf u^\ast, 0, 0) =0. 
\eq 
\end{lemma}
\begin{proof}
We can argue as in the proof of Lemma~\ref{l:suemme0} and therefore we omit most of the details. By plugging the relation $\mf z_0 = \mf r_{00} z_{00}$ into the third line of~\eqref{e:emme0slow} and using~\eqref{e:cosaeLambda0} we arrive at~\eqref{e:suemme00} provided that 
\be 
\label{e:cosaelambda00}
\begin{split}
  \theta_{00}& (\mf u, z_{00}, \sigma):  =
  \Big[ \mf r_{00}^t \{ \mf r_{00} + z_{00} \partial_{z_{00}} \mf r_{00}\} \Big]^{-1}
  \mf r_{00}^t \Big[ \mf \Theta_0 \mf r_{00} + e_{11}^{-1} z_{00} \mf d^t \mf r_{00} \partial_{u_1} \mf r_{00} -
  z_{00}  \big( \mf D_{\mf u_2}\mf r_{00} \big) \mf R_0 \mf r_{00} \Big]{.} 
       \end{split}    
\eq 
To establish~\eqref{e:lambda00}, it suffices to combine~\eqref{e:Lambda0} and~\eqref{e:eigen}. 
\end{proof}
We now collect some properties of the functions $\mf r_{00}$ and $\theta_{00}$ that we need in the following. 
\begin{lemma}
\label{l:negativo}
We have 
\be 
\label{e:negativo}
    \frac{\partial \theta_{00}}{\partial \sigma}(\mf u^\ast, 0, 0)
   < 0. 
\eq
\end{lemma}
\begin{proof}
We proceed according to the following steps. \\
{\sc Step 1:} we establish the equality 
\be 
\label{e:eigenatuast}
   - e^{-1}_{11} \mf a_{21} \mf d^t \mf r_{00} + \mf A_{22} \mf R_0 \mf r_{00} = \mf 0_{N-1} 
   \qquad \text{at $(\mf u^\ast, \mf 0_{N-2}, 0)$.}
\eq 
To establish~\eqref{e:eigenatuast} we recall that any solution of~\eqref{e:suemme0} is a solution of~\eqref{e:odeslow} and hence we plug the relation $\mf z_2 = \mf R_0 \mf r_{00} z_{00}$ into the third line of~\eqref{e:cosasono1} and we divide by $(\alpha - \sigma)$. By using~\eqref{e:definiamod} and dividing by $z_{00}$ we arrive at
\be 
\label{e:fordimezzo}
\mf B_{22}\Big[ (\mf R_0 \mf r_{00})'  + \mf R_0 \mf r_{00} \theta_{00} \Big]= 
[ \mf A_{22} - \sigma \mf E_{22} - \mf G_{2}] \mf R_0 \mf r_{00}  - \, e^{-1}_{11}\mf a_{21}  \mf d^t \mf r_{00}  
 + e^{-1}_{11} \mf g_1  \mf d^t \mf r_{00}  {\comment .} 
\eq 
By evaluating the above expression at the point $(\mf u^\ast, \mf 0_{N-2}, 0)$ and using~\eqref{e:lambda00} we eventually arrive at~\eqref{e:eigenatuast}. \\
{\sc Step 2:} by using the explicit expressions~\eqref{e:cosaeLambda0} and~\eqref{e:cosaelambda00} of $\mf \Theta_0$ and $\theta_{00}$, respectively, and by recalling~\eqref{e:eigen} we arrive at 
\begin{equation}
\label{e:derivata}
    \frac{\partial \theta_{00}}{\partial \sigma}(\mf u^\ast, 0, 0) =  \mf r_{00}^t  \frac{\partial \mf \Theta_0}{\partial \sigma} \mf r_{00} =
     2  \mf r_{00}^t  \frac{\partial \mf R_0^t}{\partial \sigma} \mf A_{22} \mf R_0 \mf r_{00} - 
     \mf r_{00}^t  \mf R_0^t 
 \mf E_{22} \mf R_0  \mf r_{00}^t
    +  \mf r_{00}^t \mf d \mf d^t \mf r_{00} e^{-1}_{11} .  
\end{equation}
By using~\eqref{e:eigenatuast} and~\eqref{e:definiamod} we get 
$$
  \mf r_{00}^t \frac{\partial \mf R_0^t}{\partial \sigma} \mf A_{22} \mf R_0 \mf r_{00} =
  \mf r_{00}^t \frac{\partial ( \mf R_0^t \mf a_{21} e^{-1}_{11})}{\partial \sigma} \mf d^t \mf r_{00} = -e^{-1}_{11} \mf r_{00}^t \mf d   \mf d^t \mf r_{00}{\comment .} 
$$
By plugging the above relation into~\eqref{e:derivata} we conclude that 
$$
   \frac{\partial \theta_{00}}{\partial \sigma}(\mf u^\ast, 0, 0) =  -
     \mf r_{00}^t  \mf R_0^t 
 \mf E_{22} \mf R_0  \mf r_{00}^t
    -  \mf r_{00}^t \mf d \mf d^t \mf r_{00} e^{-1}_{11} {\comment ,} 
$$
and this implies~\eqref{e:negativo} because $e_{11}>0$, $\mf E_{22}$ is positive definite, $\mf r_{00} \neq \mf 0_{N-2}$ owing to~\eqref{e:eigen} and $\mf R_0$ has rank $N-2$. 
\end{proof}
We now recall that $\lambda_k$ is the $k$-th eigenvalue of $\mf E^{-1} \mf A$ and can attain the value $0$. 
\begin{lemma}
\label{l:thetafazero}There is a constant $\delta>0$ such that, if $\mf u \in \mathrm{B}^N_\delta (\mf u^\ast)$, then we have 
\be
\label{e:thetafa0}
    \theta_{00} (\mf u, \mf 0_{N-2}, \lambda_k (\mf u) )=0.  
\eq  
\end{lemma}
\begin{proof}
Owing to~\eqref{e:negativo}, we can apply the Implicit Function Theorem and conclude that there is a function $\sigma (\mf u)$ such that 
$\theta_{00} (\mf u, \mf 0_{N-2}, \sigma (\mf u)) \equiv 0.$ By applying~\eqref{e:fordimezzo} at the point $(\mf u, \mf 0_{N-2}, \sigma (\mf u))$
we arrive at 
\be
\label{e:lousiamo}
    [ \mf A_{22} - \sigma \mf E_{22} ] \mf R_0 \mf r_{00}  - \, e^{-1}_{11}\mf a_{21}  \mf d^t \mf r_{00}  
 =\mf 0_{N-1} \quad \text{at $(\mf u, \mf 0_{N-2}, \sigma (\mf u))$.}
\eq
Next, we introduce the vector $\mf v:= \Big( -e_{11}^{-1} \mf d^t \mf r_{00}, \mf R_0 \mf r_{00} \Big)^t$ and by using the previous formula and~\eqref{e:definiamod} we get that $(\mf A - \sigma (\mf u) \mf E ) \mf v = \mf 0_N$. Since $\mf v \neq \mf 0_N$ (because $\mf r_{00} \neq \mf 0_{N-2}$ and $\mf R_0$ has rank $N-2$), this implies that $\sigma(\mf u)$ is an eigenvalue of $\mf E^{-1} \mf A (\mf u)$. This in turn implies that $\sigma (\mf u) = \lambda_k (\mf u)$ because $\sigma$ is close to $0$ and all the other eigenvalues of 
 $\mf E^{-1} \mf A (\mf u)$ are bounded away from $0$ by strict hyperbolicity.  This concludes the proof of the lemma. 
\end{proof}
By using~\eqref{e:lousiamo}, recalling that $\sigma (\mf u) = \lambda_k(\mf u)$ and using the block decompositions~\eqref{e:blockae2} we establish the following property. 
\begin{lemma}
\label{l:autovettore}
The vector $( - e^{-1}_{11}\mf d^t
    \mf r_{00}, 
    \mf R_0   \mf r_{00})^t$ evaluated at the point $ (\mf u,  0,  \lambda_k ( \mf u))$ is an eigenvector of $\mf E^{-1} \mf A (\mf u)$ corresponding to the eigenvalue $\lambda_k (\mf u)$. 
\end{lemma}
\subsection{The linearly degenerate case (Navier-Stokes and MHD equations)}
\label{ss:ld}
In this paragraph we focus on the linearly degenerate case, which is considerably simpler 
than the general case and it is the case of the Navier-Stokes and MHD equations. More precisely, we make the following assumptions:
\begin{itemize}
\item [I)] Let $\mf r_k (\mf u)$ be an eigenvalue of $\mf E^{-1}\mf A(\mf u)$ corresponding to $\mf \lambda_k (\mf u)$. Then $\nabla \lambda_k \cdot \mf r_k \equiv 0$.  
\item[II)] There is a smooth, invertible diffeomorphism $\mf w: \R^N \to \R^N$ such that $\mf u$ is a smooth solution of~\eqref{e:symmetric} if and only if $\mf w$ satisfies~\eqref{e:viscouscl} for some suitable functions  $\mf f$ and $\mf D$. 
\end{itemize}
Note that the Navier-Stokes and MHD equations satisfy both I) and II). This paragraph aims at establishing the following result.
\begin{lemma}
\label{l:curvekld} Under assumptions I) and II) above, 
there is a sufficiently small constant $\delta>0$ such that the following holds. There is a function $\boldsymbol{\zeta}_k: \mathrm{B}^N_\delta(\mf u^\ast) \times ]-\delta, \delta[ \to \R^N$ satisfying the following properties:
\begin{itemize}
\item[A)] For every $(\tilde{\mf u}, s_k) \in \mathrm{B}^N_\delta(\mf u^\ast) \times ]-\delta, \delta[ $, one (and only one) of the following cases holds true. 
\begin{itemize}
\item[i)] there is a contact discontinuity connecting $\tilde{\mf u}$ and $\boldsymbol{\zeta}_k (\tilde{\mf u}, s_k)$. The speed of the contact discontinuity is nonnegative and close to $0$. 
\item[ii)] there is a steady solution of~\eqref{e:symmetric} (i.e., a boundary layer) such that 
\be
\label{e:blld}
   \mf u(0) = \boldsymbol{\zeta}_k (\tilde{\mf u}, s_k), \quad \lim_{x \to + \infty} \mf u(x) = \tilde{\mf u}. 
\eq
The boundary layer lies on $\mathcal M^{00}$, i.e. it is a solution of~\eqref{e:suemme00}. 
\end{itemize}
\item[B)] The map $\boldsymbol{\zeta}_k$ is Lipschitz continuous with respect to both $\tilde{\mf u}$ and $s_k$. It also is differentiable with respect to $s_k$ at any point $(\tilde{\mf u}, 0)$ and furthermore
\be 
\label{e:dercurvekld}
   \frac{\partial \boldsymbol{\zeta}_k}{\partial s_k} (\tilde{\mf u}, 0) =
     \left(
    \begin{array}{cc}
    -e_{11}^{-1} \mf d^t
    \mf r_{00}  \\
    \mf R_0   \mf r_{00}
    \end{array}
    \right) \quad \text{applied at the point} \quad 
   \left\{
    \begin{array}{ll}
   (\tilde{\mf u},  0,  \lambda_k ( \tilde{\mf u}))
    & \text{if} \;  \lambda_k ( \tilde{\mf u}) \ge 0 {\comment ,} \\
    (\tilde{\mf u},  0,  0)
    & \text{if} \;  \lambda_k ( \tilde{\mf u}) < 0. \\  
    \end{array}
    \right.  
    \eq
    \item[C)] If $\alpha(\tilde{\mf u})>0$, then $\alpha (\boldsymbol{\zeta}_k (\tilde{\mf u}, s_k))>0$
    for every $s_k.$
\end{itemize}
\end{lemma}
Note that, owing to Lemma~\ref{l:autovettore}, if $ \lambda_k ( \tilde{\mf u}) \ge 0$ then $\partial \boldsymbol{\zeta}_k/\partial s_k$ evaluated at $s_k$ is an eigenvector of $\mf E^{-1} \mf A$. The proof of Lemma~\ref{l:curvekld} is given in the next paragraph. 
\subsubsection{Proof of Lemma~\ref{l:curvekld}} \label{sss:proofk}
We fix $\delta>0$ (the exact value will be determined in the following) and  $\tilde{\mf u} \in \mathrm{B}^N_\delta (\mf u^\ast)$. We provide the construction of $\boldsymbol{\zeta}_k (\tilde{\mf u}, \cdot)$ by separately considering two cases. \\
{\sc CASE 1:} $\lambda_k (\tilde{\mf u}) \ge 0$. In this case we can use the classical construction, which was originally provided by Lax in~\cite{Lax}. More precisely, in this case $\boldsymbol{\zeta}_k(\tilde{\mf u}, \cdot)$ is the integral curve of the vector field $\mf r_k $, i.e. it is the solution of the Cauchy problem
$$
    \partial \boldsymbol{\zeta}_k / \partial s_k = \mf r_k (\boldsymbol{\zeta}_k), \quad \boldsymbol{\zeta}_k (\tilde{\mf u}, 0) = \tilde{\mf u}. 
$$
Owing to the analysis in~\cite{Lax}, $\boldsymbol{\zeta}_k$ satisfies properties Ai) and B) in the statement of the lemma.  To establish property C), we recall Hypothesis~\ref{h:jde} and Lemma~\ref{l:autovettore} and we conclude that 
$
   \partial ( \alpha\circ \boldsymbol{\zeta}_k)  / \partial s_k =0$ if $\alpha \circ \boldsymbol{\zeta}_k =0$.
This implies that $0$ is an equilibrium for the ODE satisfied by $\alpha \circ \boldsymbol{\zeta}_k$. By the uniqueness part of the Cauchy Lipschitz Picard Lindel\"of Theorem, either $\alpha\circ \boldsymbol{\zeta}_k$ is identically $0$, or it is always different from $0$. This in particular establishes property C). \\
{\sc Case 2:} $\lambda_k (\tilde{\mf u}) < 0$. If we used the classical solution in this case, we would end up with contact discontinuities with \emph{negative} speed, which do not belong to the domain $x>0$. Instead, we study system~\eqref{e:emme0slow} and we set $\sigma=0$. In this way, we study the boundary layers of~\eqref{e:symmetric} that lie on the manifold $\mathcal M^{00}$. 
More precisely, we consider the ODEs  
\be 
\label{e:lindeg1}
   \frac{d u_1}{d\tau} = - \mf d^t \mf r_{00} (\mf u, z_{00}, 0), \quad 
    \frac{d \mf u_2}{d\tau}  = \mf R_0 \mf r_{00}(\mf u, z_{00}, 0), \quad 
    \frac{d z_{00}}{ d\tau} = \theta_{00} (\mf u, z_{00}, 0).   
\eq
Note that the above equations are obtained from~\eqref{e:emme0slow} by taking $\sigma=0$ and dividing by $z_{00}$. We consider the Cauchy problem obtained by coupling~\eqref{e:lindeg1} with the initial datum $(\mf u, z_{00})(0)=(\tilde{\mf u}, 0)$ and we fix $\delta>0$ sufficiently small in such way that the solution is defined on $]-\delta, \delta[$. We now establish
\begin{lemma}
\label{l:lindeg1} There is a sufficiently small constant $\delta>0$ such that, if $|\tilde{\mf u} - \mf u^\ast | <\delta$, then the solution of the Cauchy problem obtained by coupling~\eqref{e:lindeg1} with the initial datum $(\mf u, z_{00})(0)=(\tilde{\mf u}, 0)$ satisfies $z_{00}(s)<0$ if $0<s<\delta$ and $z_{00}(s)>0$ if $-\delta<s<0$. 
\end{lemma}
\begin{proof}
We fix $s>0$ and we show that $z_{00}(z)<0$. The proof of the other implication is analogous. First, we recall that
$\lambda_k (\tilde{\mf u})<0$ and we infer that
\be
\label{e:vagiu}
   \frac{d z_{00}}{d \tau} (0)= 
   \theta_{00}(\tilde{\mf u}, 0, 0) \stackrel{\eqref{e:thetafa0}}{=} \int_{\lambda_k (\tilde{\mf u})}^0 \frac{\partial \theta_{00}}{\partial \sigma} 
   (\tilde{\mf u}, 0, \sigma) d \sigma <0. 
\eq
To establish the last inequality, we have used~\eqref{e:negativo} and the fact that the constant $\delta$ is small. Since the derivative at $\tau=0$ is negative by~\eqref{e:vagiu}, then $z_{00}(\tau)<0$ for $\tau>0$ sufficiently small. Assume by contradiction that $z_{00}(s)>0$ for some $s>0$. We introduce the value $t$ by setting $
  t : = \mathrm{min} \{ \tau: z_{00}(\tau) = 0 \} 
$
and we point out that  and 
$d z_{00}/d\tau\ge 0$ at $\tau=t$. This means that $\theta_{00} (\mf u(t), 0, 0) \ge 0$ and by arguing as in~\eqref{e:vagiu} we conclude that $\lambda_k (\mf u(t)) \ge 0$ (because otherwise $\theta_{00} (\mf u(t), 0, 0) < 0$).  Next, we recall that $z_{00}<0$ on $]0, t[$ and we introduce the function $y$ by setting 
\be
\label{e:cambiold}
   y(\tau) = \int_{t/2}^\tau \frac{1}{z_{00}(\xi)}d \xi {\comment .}
\eq
Note that $y$ is an invertible map from $]0, t[$ onto $\R$. We term $y^{-1}$ its inverse and we point out that 
  \be
  \label{e:derivoo}
    \frac{d (u_1 \circ y^{-1})}{d x}=  \frac{d u_1 }{d \tau} 
    \left( \frac{ d y}{d \tau} \right)^{-1}=
    - \mf d^t \mf r_{00} z_{00},  \quad 
     \frac{d (\mf u_2 \circ y^{-1})}{d x}=
    \mf R_0 \mf r_{00} z_{00},  \quad   
    \frac{d (z_{00} \circ y^{-1})}{dx} =
    \theta_{00} z_{00}. 
\eq
Also, $\lim_{x \to + \infty} \mf u\circ y^{-1}(x) = \tilde{\mf u}$, $\lim_{x \to - \infty} \mf u\circ y^{-1}(x) = \mf u (t)$. Owing to~\eqref{e:suemme00}, this implies that there is solution of~\eqref{e:bltw} (i.e., a traveling wave) with $\sigma =0$ which 
connects $\tilde{\mf u}$ and $\mf u(t)$. This implies that $\mf u(t)$ lies on the $k$-th admissible wave fan curve starting at $\tilde{\mf u}$, which is unique owing to assumption II) at the beginning of this paragraph (see also~\cite[Corollary~3.3]{Bianchini}). Owing to assumption I), the value of $\lambda_k$ along this curve is constant, which contradicts the fact that $\lambda_k (\tilde{\mf u}))<0$, $\lambda_k (\mf u(t))\ge 0$. This implies that $z_{00}(s)<0$ for $s>0$ and concludes the proof.  
\end{proof}
We are now ready to define $\boldsymbol{\zeta}_k (\tilde{\mf u}, \cdot)$ in the case where $\lambda_k (\tilde{\mf u})<0$: we consider the Cauchy problem obtained by coupling~\eqref{e:lindeg1} with the initial datum $(\mf u, z_{00})(0)= (\tilde{\mf u}, 0)$ and we set $\boldsymbol{\zeta}_k(\tilde{\mf u}, s_k):= \mf u(s_k)$. To 
establish property C) in the statement of Lemma~\ref{l:curvekld} we set $a(t) := \alpha (\mf u(t))$ and we point out that, owing to Hypothesis~\ref{h:jde} and to~\eqref{e:definiamod}, $d a / dt =0$ when $a=0$. By the uniqueness part of the Cauchy Lipschitz Picard Lindel\"of Theorem, $a$ cannot change sign, which yields property C).
Property B) in the statement of Lemma~\ref{l:curvekld} is satisfied owing to classical results on ODEs. To establish property Aii), we set 
\be
\label{e:ipsilon}
    y(\tau) = \int_{s_k}^\tau \frac{1}{z_{00}(\xi)}d \xi {\comment .}
\eq
Owing to Lemma~\ref{l:lindeg1}, $y$ is an invertible map from $]0, s_k]$ onto $[0, + \infty[$. We term $y^{-1}$ its inverse and we point out that~\eqref{e:derivoo} holds. Because of the way we have constructed the manifold $\mathcal M^{00}$, this means that $\mf u \circ y^{-1}$ is a steady solution of~\eqref{e:symmetric} satisfying~\eqref{e:blld}. This concludes the proof of Lemma~\ref{l:curvekld}.
\subsection{The general case}
\label{ss:general}
In this paragraph we extend the analysis in~\S~\ref{ss:ld} to the general case, i.e. we remove the assumptions I) and II) at the beginning of~\S~\ref{ss:ld}. 
\begin{lemma}
\label{l:curvek}
There is a sufficiently small constant $\delta>0$ such that the following holds. There is a function $\boldsymbol{\zeta}_k: \mathrm{B}^N_\delta(\mf u^\ast) \times ]-\delta, \delta[ \to \R^N$ satisfying the following properties:
\begin{itemize}
\item[A)] For every $(\tilde{\mf u}, s_k) \in \mathrm{B}^N_\delta(\mf u^\ast) \times ]-\delta, \delta[ $, 
there is $\mf{\underline u} \in \R^N$ such that 
\begin{itemize}
\item[A1)] the value $\tilde{\mf u}$ is connected to $\underline{\mf u}$ by rarefaction waves an at most countable number of Liu admissible shocks or contact discontinuities. Also, the speed of the shocks, the contact discontinuities and the rarefaction waves is nonnegative and close to $0$. 
     \item[A2)] there is a steady solution (i.e., a \emph{boundary layer}) of~\eqref{e:symmetric} such that 
     \be
     \label{e:2bdproblem}
        \mf u (0) = \boldsymbol{\zeta}_k (\tilde{\mf u}, s_k), \quad 
        \lim_{x \to + \infty} \mf u(x) = \underline{\mf u}. 
     \eq
    \end{itemize}
    The boundary layer lies on the manifold $\mathcal M^{00}$, i.e. it satisfies~\eqref{e:suemme00}. 
\item[B)] The map $\boldsymbol{\zeta}_k$ is Lipschitz continuous with respect to both $\tilde{\mf u}$ and $s_k$. It also is differentiable with respect to $s_k$ at any point $(\tilde{\mf u}, 0)$ and furthermore~\eqref{e:dercurvekld} holds true. 
\end{itemize}
\end{lemma} 
To establish the previous lemma we rely  on a construction introduced in~\cite{AnconaBianchini,BianchiniSpinolo:ARMA}. The proof is provided in~\S~\ref{sss:proof1},\S~\ref{sss:proofA} and~\S~\ref{sss:proofB}. We discuss the basic idea underpinning the construction of $\boldsymbol{\zeta}_k$ in Remark~\ref{r:mon}.
\subsubsection{Construction of the function $\boldsymbol{\zeta}_k$}\label{sss:proof1}
We fix $\delta>0$ (the exact value of $\delta$ will be determined in the following), $\tilde{\mf u} \in \mathrm{B}^N_\delta(\mf u^\ast)$ and $s_k \in ]-\delta, \delta[$, $s_k >0$. 
We consider the fixed point problem  
\begin{equation}
\label{e:mappaT}
\left\{
\begin{array}{ll}
u_1 (\tau) = \displaystyle{
\tilde{u}_1 - \int_0^{\tau} 
 e_{11}^{-1} \mf d^t \mf r_{00} (\mf u (s), z_{00} (s), \sigma (s)) ds } { ,} \\
\mf u_2 (\tau) = \tilde{\mf u}_2 + 
\displaystyle{ \int_0^{\tau} 
\mf R_0 \mf r_{00} (\mf u (s), z_{00} (s), \sigma (s)) ds } {,} \\
 z_{00} (\tau) = \tilde c \Big[ f(\tau) - {\displaystyle \monconc_{[0, s_k ]} } f (\tau) \Big] { ,}
\\
\sigma (\tau) =
\displaystyle{\frac{d}{d \tau}  \monconc_{[0, s_k ]} f  } {.} \\
\end{array}
\right.
\end{equation}
In the previous expression, $\mf r_{00}$ and $\theta_{00}$ are the same as in~\eqref{e:suemme00} and the constant $\tilde c$ is defined by setting 
\be
\label{e:ci}
   \tilde c: = - \frac{\partial \theta_{00}}{\partial \sigma}(\tilde{\mf u}, 0, \lambda_k (\tilde{\mf u}))
   >0. 
\eq
To establish the last inequality, we have used Lemma~\ref{l:negativo}, which implies that $\partial \theta_{00} / \partial \sigma<0$ in a small enough neighbourhood of $(\mf u^\ast, 0, 0)$.  
Also, the function $f$ is defined by setting 
\be
\label{e:cosaeffe}
   f(\tau) : = \int_0^\tau \left[ \frac{\theta_{00}  (\mf u (s), z_{00} (s), \sigma (s))}{\tilde c} + \sigma(s)  \right] ds  {,}  
\eq 
and the monotone concave envelope is given by 
\be 
\label{e:cosaemonconc}
         \monconc_{[0, s_k ]} f (\tau): =
          \inf  \Big\{ 
        h(\tau): \ h:[0, s_k] \to \R \; 
         \text{is concave, non decreasing and $h(s) \ge f(s)
         \, \forall \, s \in [0, s_k]$}
         \Big\},
\eq 
see Figure~\ref{f:mon} for a representation. 
\begin{figure}
\begin{center}
\caption{The function $f$ (black) and its monotone concave envelope (red)} 
\psfrag{u}{$\underline \tau$}
\psfrag{t}{$\bar \tau$} 
\psfrag{f}{$f$}  
\label{f:mon}
\bigskip
\includegraphics[scale=0.4]{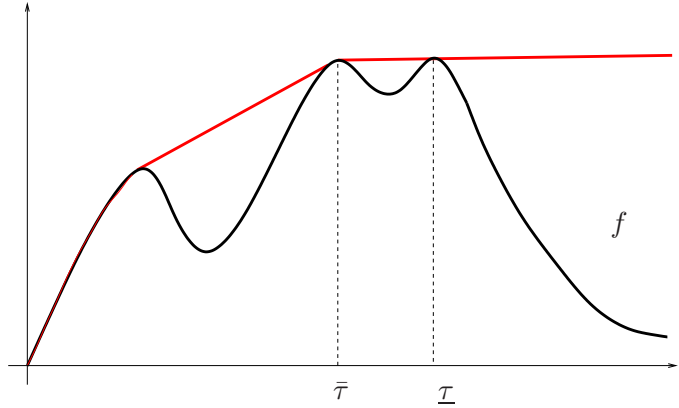}
\end{center}
\end{figure}
If $s_k<0$, one should consider the same fixed problem as in~\eqref{e:mappaT}, where the monotone concave envelope is replace by the monotone convex envelope, i.e. 
\be 
\label{e:cosaemonconv}
      \monconv_{[s_k, 0 ]} f (\tau)=
         \sup \Big\{ 
          h(\tau): \ h:[0, s_k] \to \R \; 
         \text{is convex, non decreasing and $h(s) \leq f(s)
         \, \forall \, s \in [s_k, 0]$}
         \Big\}. 
\eq 
We have the following result. 
\begin{lemma}
\label{l:admissiblestates}
There is a sufficiently small constant $\delta>0$ such that, for every $\tilde{\mf u} \in \mathrm{B}^N_\delta (\mf u^\ast)$, 
$s_k \in [0, \delta[$, there is a unique triplet of continuous functions $(\mf u, z_{00}, \sigma)$ satisfying~\eqref{e:mappaT} which is confined in a sufficiently small neighbourhood of $(\tilde{\mf u}, 0, 0)$. The same holds if $s_k \in ]-\delta, 0]$, provided in~\eqref{e:mappaT} we replace~\eqref{e:cosaemonconc} with~\eqref{e:cosaemonconv}. 
\end{lemma}
The proof of the above lemma is based on the same fixed point argument  as in~~\cite[Lemma 3.4]{BianchiniSpinolo:ARMA} (see also~\cite[\S~3]{Bianchini}) and is therefore omitted. We can now define $\boldsymbol{\zeta}_k$ by setting
\be
\label{e:gikappa}
    \boldsymbol{\zeta}_k (\tilde{\mf u}, s_k)= \mf u(s_k), \quad \text{$(\mf u, z_{00}, \sigma)$ satisfying} \;
    \left\{
    \begin{array}{ll}
    \text{\eqref{e:mappaT}} & \text{if $s_k \ge 0$}, \\
    \text{\eqref{e:mappaT} with~\eqref{e:cosaemonconc} replaced 
    by~\eqref{e:cosaemonconv}} & \text{if $s_k<0$}. \\    
    \end{array}
    \right. 
\eq    
\subsubsection{Proof of Lemma~\ref{l:curvek}, property A)} 
\label{sss:proofA}
We only consider the case $s_k \ge 0$, the case $s_k<0$ is entirely analogous. First, we discuss an elementary result on the monotone concave envelope. 
Given $s_k>0$ and $f \in C^0 ([0, s_k])$ we introduce the values $m$ and $\underline{\tau}$ by setting 
\be
\label{e:massimo}
   m: = \max \big\{ f(\tau): \ \tau \in [0, s_k] \big\}, \quad 
    \underline \tau: = \max \big\{ \tau \in [0, s_k]: \; 
    f(\tau) = m \big\},
\eq
see also Figure~\ref{f:mon}. Note that the constant function $h(\tau)=m$ is a concave, non decreasing function satisfying $h \ge f$, hence by definition~\eqref{e:cosaemonconc} we have 
\be
\label{e:stasopra}
    m \ge \monconc_{[0, s_k]} f (\tau) \ge 
    \mathrm{conc}_{[0, s_k]} f (\tau) \ge f(\tau), \quad \text{for every $\tau \in [0, s_k]$}.
\eq
In the previous expressions, $\mathrm{conc}_{[0, s_k]} f$ denotes the concave envelope of the function $f$ on $[0, s_k]$. Note that the above chain of inequalities implies that $m$ is also the maximum of the function $\monconc_{[0, s_k]} f $. We have the following elementary result. The proof is standard and therefore omitted. 
\begin{lemma}
\label{l:concave}
If $f \in C^0 ([0, s_k])$, then
\be
\label{e:monconvcon} 
   \monconc_{[0, s_k]} f (\tau) =
   \left\{
   \begin{array}{ll}
   \mathrm{conc}_{[0, s]} f (\tau) & \tau \leq \underline{\tau} { ,} \\
  m & 
   \tau \ge \underline \tau {.}
   \end{array}
   \right.
\eq
\end{lemma}
We now recall~\eqref{e:gikappa} and we complete the proof of Lemma~\ref{l:curvek}, property A).  First, we point out that $z_{00} (\tau) \leq 0$ for every $\tau$. Next, we recall the definition~\eqref{e:massimo} of $\underline \tau$ and we separately consider two cases. \\
{\sc Case 1:} $\underline \tau = s_k$. In this case we set $\underline{\mf u}: = \mf u(s_k)$. Property A2) is trivially satisfied, it suffices to take $\mf u(t) \equiv \underline{ \mf u}$. We now establish property A1). Owing to Lemma~\ref{l:concave}, we can apply~\cite[Theorem 3.2]{Bianchini} and conclude that $\tilde{\mf u}$  is connected to $\underline{\mf u}= \mf u(s_k)$ by rarefaction waves and a sequence of  contact discontinuities and shocks satisfying~Liu admissibility condition. Since the shocks and contact discontinuities speed is $\sigma$, then the speed is nonnegative.  \\
{\sc Case 2:} $\underline \tau < s_k$.  In this case we set $\underline{\mf u} : = \mf u(\underline \tau)$. Owing to Lemma~\ref{l:concave} and to~\cite[Theorem 3.2]{Bianchini}, $\underline{\mf u}$ satisfies property A1). We now establish property A2). Note that, by the definition of $\underline \tau$, if $\underline \tau < s_k$ then $z_{00}(\tau)<0$ on $]\underline \tau, s_k]$. Also, owing to~\eqref{e:stasopra}, $z_{00} (\underline \tau) =0$ and, owing to Lemma~\ref{l:concave}, $\sigma (\tau) \equiv 0$ on $]\underline \tau, s_k]$. We now consider the same function $y$ as in~\eqref{e:ipsilon} and we point out that $y$ is an invertible function from $]\underline \tau, s_k]$ onto $[0, + \infty[$. We term $y^{-1}$ its inverse and we point out that~\eqref{e:derivoo} holds. This implies that $(\mf u \circ y^{-1}, z_{00} \circ y^{-1}, 0)$ is a solution of~\eqref{e:emme0slow} and hence a boundary layer for~\eqref{e:symmetric}. Also, it satisfies~\eqref{e:2bdproblem} and this concludes the proof of A2). 
\subsubsection{Proof of Lemma~\ref{l:curvek}, property B)}\label{sss:proofB}
The proof of the Lipschitz continuity property follows from an argument analogous to the argument in the proof of~\cite[Lemma 14.3]{BianchiniBressan} and it is therefore omitted. We now focus on the proof of property~\eqref{e:dercurvekld}. 
We first establish a preliminary result.
\begin{lemma}
\label{l:passlim}
        Fix a sequence $s_k^n \to 0^+$ and term $(\mf u^n(0)$, $z_{00}^n(0)$, $\sigma^n(0))$ the value attained at $\tau=0$ by the solution of the fixed point problem~\eqref{e:mappaT}, which is defined on the interval $[0, s^n_k]$. Then 
        \be
\label{e:ilimiti}
    \lim_{n \to +\infty} \mf u^n(0) = \tilde{\mf u}, \quad 
     \lim_{n \to +\infty} z_{00}^n(0) = 0, \quad 
     \lim_{n \to +\infty} \sigma^n(0) =
      \left\{
      \begin{array}{ll}
      \lambda_k (\tilde{\mf u}) & \mathrm{if} \ \lambda_k (\tilde{\mf u}) \ge 0 {\comment ,} \\
     0 &  \mathrm{if} \  \lambda_k (\tilde{\mf u})< 0 {\comment .} \\
      \end{array}
      \right.  
\eq   
\end{lemma}
\begin{proof}
First, we point out that the first equality in~\eqref{e:ilimiti} is trivial since $\mf u^n (0) = \tilde{\mf u}$ for every $n$. We are left to establish the other two equalities. We point out that it suffices to show that, for every subsequence, there is a further subsequence (which for simplicity we do not re-label) satisfying~\eqref{e:ilimiti}. Also, since $\{ z_{00}^n (0) \}_{n \in \mathbb N}$ and 
 $\{ \sigma^n (0) \}_{n \in \mathbb N}$ are both bounded, then we can assume (up to subsequences, that we do not re-label) that the limits in~\eqref{e:ilimiti} exist. In particular, we term $\ell$ the limit of $\sigma^n(0)$. Also,  we term 
$f^n$ the function defined as in~\eqref{e:cosaeffe} and defined on $[0, s^n_k]$. Note that $f^n(0)=0$ for every $n$. Up to subsequences (that we do not re-label), we can also assume that one of the following cases holds true. \\
{\sc Case 1:} $\max_{[0, s_k^n]} f^n>0$ for every $n$. Since $f^n(0)=0$, owing to Lemma~\ref{l:concave}, this implies that 
$\monconc_{[0, s_k^n]} f\equiv \mathrm{conc}_{[0, s^n_k]} f$ in a right neighbourhood of $0$ for every $n$. We now apply estimate (5.8) in~\cite[p. 314]{AnconaMarson} with $g=0$ and we conclude that $\lim_{n \uparrow +\infty} z_{00}^n(0) = 0$. Next, we apply estimate (5.7)  in~\cite[p. 314]{AnconaMarson} with $g(\tau) : = (f^n)'(0) \tau$ and we conclude that 
$
    \lim_{n  \uparrow +\infty} \sigma^n(0) =  \lim_{n  \uparrow +\infty} (f^n)'(0). 
$
Recall that this limit is termed $\ell$. Owing to~\eqref{e:cosaeffe}, this implies that 
$\theta_{00} (\tilde{\mf u}, 0, \ell) / \tilde c + \ell = \ell$, i.e. that  
$\theta_{00} (\tilde{\mf u}, 0, \ell) =0$. By recalling~\eqref{e:thetafa0} and that $\partial \theta_{00}/\partial \sigma<0$, we conclude that $\ell = \lambda_k (\tilde{\mf u})$. Since $\sigma^n (\tau) \ge 0$ for every $\tau$, then $\ell \ge 0$ and hence $\lambda_k(\tilde{\mf u}) \ge 0$. \\
{\sc Case 2:} $\max_{[0, s_k^n]} f^n=0$ for every $n$. Owing to~\eqref{e:stasopra}, this implies that $\monconc_{[0, s_k^n]} f(0)=0$ and hence that $z_{00}^n(0)=0$, which gives the second limit in~\eqref{e:ilimiti}. Also, since $\monconc_{[0, s_k^n]} f$ is a nondecreasing function, then in this case we have $\monconc_{[0, s_k^n]} f \equiv 0$, which implies $\sigma^n(0) =0$ and hence $\ell=0$.   \\
To conclude the proof of~\eqref{e:ilimiti}, we are left to show that, if $\lambda_k (\tilde{\mf u}) >0$, then we 
are necessarily in {\sc Case 1}. Assume by contradiction that $\lambda_k (\tilde{\mf u}) >0$ and that we are in {\sc Case 2}, which implies $\sigma^n(0) \equiv 0$. Owing to~\eqref{e:cosaeffe} we have 
$$
   (f^n)'(0)\stackrel{\eqref{e:cosaeffe}}{=} \frac{\theta_{00} (\tilde{\mf u}, 0,0))}{\tilde c}
   \stackrel{\eqref{e:thetafa0}}{=}
   \frac{1}{\tilde c} \int_{\lambda_k (\tilde{\mf u})}^{0} \frac{\partial \theta_{00}}{\partial \sigma} 
   (\tilde{\mf u}, 0, \xi)
   d \xi  \stackrel{\eqref{e:ci}}{=}
   \lambda_k (\tilde{\mf u}) + \mathcal O(1) |  \lambda_k (\tilde{\mf u})|^2. 
$$
Since $\lambda_k (\tilde{\mf u})>0$ is close to $0$, the above chain of equalities implies that $(f^n)'(0) >0$, which contradicts the assumption that the maximum of $f^n$ is attained at $\tau=0$. This concludes the proof of~\eqref{e:ilimiti}. 
\end{proof}
{We} can now establish~\eqref{e:dercurvekld}: we only show that~\eqref{e:dercurvekld} is the right derivative, to show that it also the left derivative the argument is analogous. Owing to~\eqref{e:mappaT}, 
$$
    u_1(s_k) = \tilde u_1 - s_k \mf d^t \mf r_{00}(\mf u(0), z_{00} (0), \sigma(0)) + \mathcal O(1) s_k^2, 
   \qquad 
   \mf u_2 (s_k) = \tilde{\mf u}_2 + s_k \mf R_0 \mf r_{00}(\mf u(0), z_{00} (0), \sigma(0)) + \mathcal O(1) s_k^2. 
$$
We now pass to the limit $s_k \to 0^+$ and by using~\eqref{e:ilimiti} we arrive at~\eqref{e:dercurvekld}. 
\begin{remark}
\label{r:mon}
The basic idea underpinning the construction of $\boldsymbol{\zeta}_k$ is the following. We recall the construction of the $i$-th admissible wave fan curve given in~\cite[\S~14]{BianchiniBressan} and~\cite{Bianchini}: one considers the same fixed point problem as in~\eqref{e:mappaT}, the only difference is that one takes the concave envelope instead of the monotone concave envelope. Very loosely speaking, the basic idea in~\cite{BianchiniBressan,Bianchini} is that the intervals  where $z_{00}<0$ correspond to shocks or contact discontinuities with speed $\sigma$, the sets where $z_{00}=0$ and $\sigma$ is strictly increasing correspond to rarefactions (see the proof of~\cite[Lemma 14.1]{BianchiniBressan} for a more detailed explanation) and finally the sets where $z_{00}=0$ and $\sigma$ is constant correspond again to contact discontinuities. To understand why in the case of the initial-boundary value problem we replace the concave envelope with the monotone concave envelope we {first of all} recall Lemma~\ref{l:concave}.  
The basic idea underpinning the construction of $\boldsymbol{\zeta}_k$ is that we use the same construction as in the Cauchy problem ``as long as possible", i.e. as long as derivative of the concave envelope, i.e. $\sigma$, is non-negative. After that, instead of using waves with negative speed, which are not admissible, we use boundary layers, i.e. we take 
$\sigma =0$. 
\end{remark}
\subsection{Application to the Navier-Stokes and MHD equations}  \label{ss:4:nsmhd}
\subsubsection{Navier-Stokes equations} \label{sss:4:ns} We recall the discussion in~\S~\ref{ss:h:ns} and~\S~\ref{sss:3:ns} and we point out that the analysis in~\S~\ref{ss:cma} is actually redundant in this case because the manifold $\mathcal M^{00}$ is the whole manifold $\mathcal M^0$.  Indeed, by linearizing~\eqref{e:suemme0} at $(\mf u^\ast, 0, 0)$ we obtain a nilpotent matrix and hence the center space is the whole $\R^5$. 
Since an eigenvector of $\mf E^{-1} \mf A(\mf u)$ associated to $\lambda_2 (\mf u) = u$ is $\mf r_2 (\mf u)=(\rho, 0, -\theta)^t$, then the second vector field is linearly degenerate and hence we can apply the analysis in~\S~\ref{ss:ld} and we do not need the analysis in~\S~\ref{ss:general}. 
\subsubsection{MHD equations with $\eta>0$} \label{sss:4:mhd} We recall the discussion in~\S~\ref{ss:h:mhd} and~\S~\ref{sss:3:mhd} and, in particular, that the dimension of $\mathcal M^0$ is $13$. In this case we need the analysis~\S~\ref{ss:cma} since the manifold $\mathcal M^{00}$ has dimension $9$ and it is strictly contained in $\mathcal M^0$. We recall~\eqref{e:mhd:Theta} and~\eqref{e:eigen} and we conclude that $\mf r_{00}(\mf u^\ast, 0, 0)= (\mf 0_2^t, \mf 0_2^t, 1)^t$.  
Since an eigenvector of $\mf E^{-1} \mf A(\mf u)$ associated to $\lambda_4 (\mf u) = u$ is $\mf r_4 (\mf u)=(\rho, \mf 0_2^t, 0, \mf 0_2^t, -\theta)^t$, then the fourth vector field is linearly degenerate and hence we can apply the analysis in~\S~\ref{ss:ld} and we do not need the analysis in~\S~\ref{ss:general}. 
\section{Boundary layers lying on $\mathcal M^0$}
\label{s:blemme0}
In this section we complete the analysis of the boundary layers lying on $\mathcal M^0$. Note that in~\S~\ref{s:center} we have considered the characteristic boundary layers, which lie on $\mathcal M^0$. As we point out in~\S~\ref{sss:5:ns}, in the case of the Navier-Stokes the manifold $\mathcal M^0$ only contains the characteristic boundary layers and hence the analysis in this section is basically useless. However, in general there might be non characteristic boundary layers lying on $\mathcal M^0$, for instance this happens in the case of the MHD equations, see~\S~\ref{sss:5:mhd}. The main result of this section is Theorem~\ref{t:slowmanifold}. In~\S~\ref{ss:proofsmld} we provide its proof in the linearly degenerate case, in~\S~\ref{ss:pgeneral} we deal with the general case. In~\S~\ref{ss:5:nsmhd} we describe how the analysis applies to the Navier-Stokes and MHD equations. As in the previous sections, we only consider the case $h=1$ and we refer to~\S~\ref{s:bigk} for the case $h>1$. We first establish a corollary of Lemma~\ref{l:key}.  
\begin{lemma}
\label{l:thetaper}
There is a sufficiently small constant $\delta>0$ such that the following holds. If $|\mf u - \mf u^\ast|<\delta$, then all the eigenvalues of the matrix $\Theta_0(\mf u, \mf 0_{N-2}, 0)$ are real numbers. Also, the signature is as follows: $k-2$ eigenvalues are strictly negative, $N-k-1$ are strictly positive and one eigenvalue is $\theta_{00}
(\mf u,0, 0)$ and has exactly the same sign as $\lambda_k (\mf u)$. 
\end{lemma}
\begin{proof}
First, we recall~\eqref{e:cosaeLambda0} and by applying Lemma~\ref{l:sym} we conclude that all the eigenvalues of  $\Theta_0(\mf u, \mf 0_{N-2}, 0)$  are real numbers. We recall Lemma~\ref{l:key} and by the continuity of the eigenvalues we infer that, if $\mf u$ belongs to a sufficiently small neighbourhood of $\mf u^\ast$, there are $k-2$ strictly negative eigenvalues, $N-k-1$ strictly positive eigenvalues, and one eigenvalue close to $0$. We now show that the eigenvalue close to $0$ is $\theta_{00}
(\mf u, \mf 0_{N-2}, 0)$: it suffices to recall~\eqref{e:fordimezzo}, evaluate it at the point $(\mf u, \mf 0_{N-2}, 0)$, recall that $\mf G_2$ and $\mf g_1$ vanish, left multiply~\eqref{e:fordimezzo} times $\mf R_0^t$ and use~\eqref{e:definiamod}. To show that $\theta_{00}$ has exactly the same sign as $\lambda_k(\mf u)$, we recall~\eqref{e:negativo},~\eqref{e:thetafa0} and argue as in~\eqref{e:vagiu}.  
\end{proof}
We fix the eigenvectors $\mf p_2({\mf u}) , \dots, \mf p_{k-1}(\mf u)  \in \R^{N-2}$ 
corresponding to the $k-2$ strictly negative eigenvalues of $\Theta_0(\mf u, \mf 0_{N-2}, 0)$ (the ones different from $\theta_{00}(\mf u, 0, 0)$, if $\lambda_k (\mf u)<0$). We can choose them in such a way that $\mf p_i^t(\mf u^\ast)  \mf p _j (\mf u^\ast)=0$ if $i\neq j$ and $|\mf p_i(\mf u^\ast) |=1 $ for every $i$.  
We also define the vectors $\mf q_2, \dots, \mf q_{k-1} \in \R^{N}$ by setting 
\begin{equation}
\label{e:definiscot}
      \mf q  (\mf u) : = 
      \left(
      \begin{array}{ccc}
      - e_{11}^{-1} \mf d^t ({\mf u}, \mf 0_{N-2},0) \mf p_i (\mf u) \\
      \mf R_0 ( \mf u, \mf 0_{N-2},0) \mf p_i (\mf u)  \\
      \end{array}
      \right), \quad i=2, \dots, k-1. 
\end{equation}
\begin{theorem}
\label{t:slowmanifold}
There is a sufficiently small constant $\delta>0$ such that the following holds. There is a function $\boldsymbol{\psi}_{sl}: \mathrm{B}^N_\delta(\mf u^\ast) \times \mathrm{B}^{k-1}_\delta(\mf 0_{k-1}) \to \R^N$ satisfying the following properties:
\begin{itemize}
\item[A)] For every $(\tilde{\mf u}, s_2, \dots, s_k) \in \mathrm{B}^N_\delta(\mf u^\ast) \times  \mathrm{B}^{k-1}_\delta(\mf 0_{k-1})$, 
there is $\mf{\underline u} \in \R^N$ such that property A1) in the statement of Lemma~\ref{l:curvek} holds true and furthermore
\begin{itemize}
     \item[A2)] there is a steady solution (i.e., a \emph{boundary layer}) of~\eqref{e:symmetric} such that 
     \be
     \label{e:2bdproblem2}
        \mf u (0) = \boldsymbol{\psi}_{sl} (\tilde{\mf u}, s_2, \dots, s_k), \quad 
        \lim_{x \to + \infty} \mf u(x) = \underline{\mf u}. 
     \eq
    \end{itemize}
    The boundary layer lies on the manifold $\mathcal M^{0}$, i.e. it satisfies~\eqref{e:emme0slow}. 
\item[B)] The map $\boldsymbol{\psi}_{sl}$ is Lipschitz continuous with respect to both $\tilde{\mf u}$ and $s_2, \dots, s_k$. It also is differentiable with respect to $s_2, \dots, s_k$ at any point $(\tilde{\mf u}, \mf 0_{k-1})$ and the columns of the Jacobian matrix are the vector in~\eqref{e:dercurvekld} and the vectors $\mf q_2, \dots, \mf q_{k-1}$ in~\eqref{e:definiscot}.
\end{itemize}
\end{theorem} 
The proof of Theorem~\ref{t:slowmanifold} relies on the construction in~\S~\ref{s:center}. For this reason, we first provide the proof in the linearly degenerate case (see conditions I) and II) at the beginning of~\S~\ref{ss:ld}). Next, we consider the general case.
\subsection{Proof of Theorem~\ref{t:slowmanifold} in the linearly degenerate case} 
\label{ss:proofsmld}
We assume conditions I) and II) at the beginning of~\ref{ss:ld}, we fix $(\tilde{\mf u}, s_2, \dots, s_k)$ as in the statement of the theorem and we recall case i) and ii) in the statement of Lemma~\ref{l:curvekld}. Next, 
we proceed according to the following steps. \\  
{\sc Step 1:} we assume $\lambda_k (\tilde{\mf u}) \ge 0$ and define the value $\boldsymbol{\psi}_{sl} (\tilde{\mf u}, s_2, \dots, s_k)$. First, we recall that $\underline{\mf u} = \boldsymbol{\zeta}_k (\tilde{\mf u}, s_k)$ satisfies property A1) in the statement of Lemma~\ref{l:curvekld}.  Next, we linearize system~\eqref{e:emme0slow} at the point $(\underline{\mf u}, \mf 0_{N-2}, 0)$, we recall that $\lambda_k(\underline{\mf u})= \lambda_k (\tilde{\mf u}) \ge 0$ by linear degeneracy  and owing to Lemma~\ref{l:thetaper} we conclude that the stable space (i.e., the space generated by the eigenvalues corresponding to eigenvectors with strictly negative real part) is 
$$
    M^-: = \big\{ \underline{\mf u}, \mf z_0, 0): \, \mf z_0 \in \mathrm{span} 
    < \mf p_2(\underline{\mf u}), \dots, \mf p_{k-1}(\underline{\mf u})> \big\}. 
$$
We apply the Stable Manifold Theorem and we determine a map $\tilde{\boldsymbol{\psi}}_{sl}$, attaining values in $\R^N \times \R^{N-2} \times \R$, which parameterizes the stable manifold. We term $\boldsymbol{\psi}_{sl}$  the projection of $\tilde{\boldsymbol{\psi}}_{sl}$ onto $\R^N$ (i.e. $\boldsymbol{\psi}_{sl}$ are the first $N$ components of $\tilde{\boldsymbol{\psi}}_{sl}$). By construction, property A2) in the statement of the theorem is satisfied.  Note that $\boldsymbol{\psi}_{sl}$ depends on $s_2, \dots, s_{k-1}$ and on $\underline{\mf u} = \boldsymbol{\zeta}_k (\tilde{\mf u}, s_k)$, so as a matter of fact it depends on $s_2, \dots, s_k, \tilde{\mf u}$. By relying on property B) in the statement of Lemma~\ref{l:curvekld} and on the Stable Manifold Theorem we can establish property B) in the statement of Theorem~\ref{t:slowmanifold}. \\
{\sc Step 2:} we assume $\lambda_k (\tilde{\mf u}) < 0$ and define the value $\boldsymbol{\psi}_{sl} (\tilde{\mf u}, s_2, \dots, s_k)$. We set $\underline {\mf u}:= \tilde{\mf u}$ and point out that property A1) in the statement of Lemma~\ref{l:curvekld} is trivially satisfied. Next, we linearize system~\eqref{e:emme0slow} at the point $(\tilde{\mf u}, \mf 0_{N-2}, 0)$ and owing to Lemma~\ref{l:thetaper} we conclude that, since $\lambda_k (\tilde{\mf u}) < 0$, then the stable space has dimension $k-1$. We apply the Stable Manifold Theorem, we determine a map $\tilde{\boldsymbol{\psi}}_{sl}$ parametrizing the stable manifold and we term $\boldsymbol{\psi}_{sl}$ its projection onto $\R^N$ (i.e. $\boldsymbol{\psi}_{sl}$ are the first $N$ components of $\tilde{\boldsymbol{\psi}}_{sl}$). This implies that properties A2) and B) in the statement of the theorem are satisfied. 
\subsection{Proof of Theorem~\ref{t:slowmanifold} in the general case}
\label{ss:pgeneral}
We fix $(\tilde{\mf u}, s_2, \dots, s_k)$ as in the statement of the theorem and we proceed according to the following steps.\\
{\sc Step 1:} we recall the statement of Lemma~\ref{l:curvek} and we conclude that there is $\underline{\mf u}$ satisfying A1) and A2) in there. In particular, there is a boundary layer lying on $\mathcal M^{00}$ and satisfying~\eqref{e:2bdproblem}. Since the boundary layer lies on $\mathcal M^{00}$, then it has the form $(\mf u, z_{00}, 0)$. By the analysis in \S~\ref{ss:cma}, this implies that $(\mf u, z_{00} \mf r_{00}, 0)$ is an orbit lying on $\mathcal M^0$. We term it $\mf v_0$ and we point out that $\mf v_0(0) = (\boldsymbol{\zeta}_k (\tilde{\mf u}, s_k), z_{00} \mf r_{00}, 0)$. \\
{\sc Step 2:} we apply Lemma~\ref{l:uniformdecay} in the case where system~\eqref{e:ode} is given by~\eqref{e:emme0slow}, $\mf v^\ast: = (\mf u^\ast, \mf 0_{N-2}, 0)$ and $\check{\mf v}:~=~(\boldsymbol{\zeta}_k(\tilde{\mf u}, s_k), \mf 0_{N-2}, 0).$ Owing to Lemma~\ref{l:key}, $n_- = k-2$. By applying Lemma~\ref{l:uniformdecay} we define a map $\mf m_-$, which depends on $\boldsymbol{\zeta}_k(\tilde{\mf u}, s_k), s_2, \dots, s_{k-1}$ and hence on $\tilde{\mf u}$ and $s_1, \dots, s_k$. \\
{\sc Step 3:}  we apply Lemma~\ref{l:slaving} with~\eqref{e:ode}, $\mf v^\ast$ and $\check{\mf v}$ as in {\sc Step 2} and $\mf v_0 (0)$ as in {\sc Step 1}. 
We define a map $\mf m_p$ which depends on $\mf v_0(0)$ and $s_2, \dots, s_{k-1}$, and hence on $\tilde{\mf u}$ and $s_1, \dots, s_k$. Note that $\mf m_-$ and $\mf m_{p}$ both attain values in $\R^N \times \R^{N-2} \times  \R$. We set 
\be 
\label{e:effeslow}
   \tilde{\boldsymbol{\psi}}_{sl} (\tilde{\mf u}, s_2, \dots, s_k)= \mf v_0(0) + 
    \mf m_- (\tilde{\mf u}, s_2, \dots, s_k) - \check{\mf v} +
    \mf m_p (\tilde{\mf u}, s_2, \dots, s_k).  
\eq 
We term $\boldsymbol{\psi}_{sl} $ the first $N$ components of $ \tilde{\boldsymbol{\psi}}_{sl}$, i.e. the projection of  $\tilde{\boldsymbol{\psi}}_{sl}$
onto $\R^N$. 
Note that owing to~\eqref{e:expdecay2}  the solution of~\eqref{e:emme0slow} with initial datum $  \tilde{\boldsymbol{\psi}}_{sl} $ approaches the boundary layer $\mf v_0$ at exponential rate as $x \to + \infty$. In particular, the $\mf u$ component converges to $\underline{\mf u}$ and hence~\eqref{e:2bdproblem2} is satisfied. We remark in passing that, since the last component of $\mf v_0$ is identically $0$, so is the last component of  $\tilde{\boldsymbol{\psi}}_{sl}$ because $\sigma$ is constant on the orbits of~\eqref{e:emme0slow}.   \\
{\sc Step 4:} to establish property B) in the statement of Theorem~\ref{t:slowmanifold} we combine property B) in Lemma~\ref{l:curvek} with the regularity part of Lemma~\ref{l:uniformdecay}, with~\eqref{e:diffat0} and with Lemma~\ref{l:dipendenzav0}. 
\vspace{0.5cm}

To conclude this paragraph we make some heuristic comment on the proof of Theorem~\ref{t:slowmanifold}. In particular, we recall {\sc Case 1} and {\sc Case 2} at the end of~\S~\ref{sss:proofA} and we describe the structure of the boundary layer satisfying~\eqref{e:2bdproblem2} in these cases. \\
{\sc Case 1:} $\underline \tau= s_k$. In this case $\underline{\mf u} = \boldsymbol{\zeta}_k (\tilde{\mf u}, s_k)$ and the boundary layer lying on $\mathcal M^{00}$ is trivial, i.e. $\mf u(x) \equiv \underline{\mf u}$. Also, $z_{00}(s_k)=0$ and hence $\mf v_0 (0)= \check{\mf v}$, which owing to~\eqref{e:diffat0} implies that  $\mf m_p (\tilde{\mf u}, s_2, \dots, s_k)=\mf 0_{2N-1}$. This implies that  $\tilde{\boldsymbol{\psi}}_{sl} (\tilde{\mf u}, s_2, \dots, s_k)=
    \mf m_- (\tilde{\mf u}, s_2, \dots, s_k)$. Hence, in this case the boundary layer satisfying~\eqref{e:2bdproblem2} does not have any component lying on~$\mathcal M^{00}$ and the $\mf u$ component decays to $\underline{\mf u}$ at exponential rate. \\
{\sc Case 2:} $\underline \tau< s_k$. In this case there is a nontrivial boundary layer lying on $\mathcal M^{00}$, which satisfies~\eqref{e:2bdproblem}. By applying Lemma~\ref{l:slaving} we construct a slaving manifold of solutions of~\eqref{e:emme0slow} that approaches the boundary layer lying on $\mathcal M^{00}$ at exponential rate. In this case the boundary layer satisfying~\eqref{e:2bdproblem2}  will in general converge to $\underline{\mf u}$ at a slower rate than in {\sc Case 1}. 
\subsection{Application to the Navier-Stokes and MHD equations}\label{ss:5:nsmhd}
\subsubsection{Navier-Stokes equation} \label{sss:5:ns} We recall the discussion in~\S~\ref{sss:4:ns}, that the manifold $\mathcal M^0$ coincides with the manifold $\mathcal M^{00}$ and that $k-2=0$. This implies that the analysis in~\S~\ref{s:blemme0} is actually redundant in this case.  

\subsubsection{MHD equations with $\eta>0$} \label{sss:5:mhd} We recall the discussion in~\S~\ref{sss:3:mhd} 
and that $k-2=2$. This implies that we need the analysis in~\S~\ref{s:blemme0} and the function $\boldsymbol{\psi}_{sl}$ depends on $\tilde{\mf u}$ and on $3$ other variables $s_2, s_3, s_4$.
\section{Complete boundary layers analysis} 
\label{s:complete}
Very loosely speaking, in this section we combine the ``slow" boundary layers lying on $\mathcal M^0$ with the ``fast" boundary layers lying on a stable manifold for~\eqref{e:odefast}. See~\S~\ref{ss:roadmap} for a more detailed discussion. Owing to the nonlinearity, we cannot simply add the ``slow" and the ``fast" boundary layers, but we have to take into account possible interactions. From the technical viewpoint, this issue is tackled by relying on the notion of slaving manifold, which is overviewed in~\S~\ref{s:slaving}. The exposition is organized as follows. In~\S~\ref{ss:fastvaria} we work on system~\eqref{e:odefast}, where $\mf h$ is given by~\eqref{e:cosasono1}. In~\S~\ref{ss:back} we show that (in some cases) we can actually go back to the original system~\eqref{e:odeslow}. This requires a quite careful analysis which uses Hypothesis~\ref{h:jde}. As in the previous sections, we focus on the case $h=1$ and we refer to~\S~\ref{s:bigk} for the case $h>1$.  
\subsection{Fast variable analysis}\label{ss:fastvaria} First, we have to introduce some notation. We recall the proof of Theorem~\ref{t:slowmanifold} and the fact that the map $\boldsymbol{\psi}_{sl} (\tilde{\mf u}, s_2, \dots, s_k)$ is the projection onto $\R^N$ of a map $\tilde{\boldsymbol{\psi}}_{sl} $, attaining values in $\R^N \times \R^{N-2} \times \R$ such that i) the last component of $\tilde{\boldsymbol{\psi}}_{sl}(\tilde{\mf u}, s_2, \dots, s_k)$ is identically $0$;  ii) for every $(\tilde{\mf u}, s_2, \dots, s_k)$ the solution of the Cauchy problem obtained by coupling~\eqref{e:emme0slow} with the initial datum $\tilde{\boldsymbol{\psi}}_{sl}(\tilde{\mf u}, s_2, \dots, s_k)$ satisfies $\lim_{x \to + \infty} \mf u(x) = \underline{\mf u}$, where $\underline{\mf u}$ is a certain state depending on $\tilde{\mf u}$ and $s_k$. Next, we recall that system~\eqref{e:emme0slow} is obtained from~\eqref{e:suemme0} through the change of variables $x = \alpha y$, and 
that~\eqref{e:suemme0} is system~\eqref{e:odefast} restricted on $\mathcal M^0$, which is a center manifold for system~\eqref{e:odefast}.  We now consider the solution of the Cauchy problem obtained by coupling the initial datum $\tilde{\boldsymbol{\psi}}_{sl}(\tilde{\mf u}, s_2, \dots, s_k)$
     with system~\eqref{e:emme0slow}: this provides a solution of~\eqref{e:odefast} lying on the center manifold $\mathcal M^0$. We term it 
     $\mf v_0 [\tilde{\mf u}, s_2, \dots, s_k]$. Note that this is not the same $\mf v_0$ as in~\S~\ref{ss:pgeneral} and that, for any given  $\tilde{\mf u}, s_2, \dots, s_k$,   $\mf v_0 [\tilde{\mf u}, s_2, \dots, s_k]$ is a function of $y$. The next lemma states that~\eqref{e:odefast} has an invariant manifold of orbits approaching   $\mf v_0 [\tilde{\mf u}, s_2, \dots, s_k]$.   
\begin{lemma}
\label{l:fastvaria} 
There is a sufficiently small constant $\delta>0$ such that the following holds. There is a function $\tilde{\boldsymbol{\psi}}_b: \mathrm{B}^N_\delta(\mf u^\ast) \times \mathrm{B}^{k}_\delta(\mf 0_{k}) \to \R^N \times \R^{N-1} \times \R$ such that 
\begin{itemize}
\item[A1)] The last component of $\tilde{\boldsymbol{\psi}}_b$ is identically $0$. 
\item[A2)] For every $(\tilde{\mf u}, s_1, \dots, s_k)$ the following holds: if $\mf v$ is the solution of the Cauchy problem obtained by coupling the initial datum $\mf v(0)= \tilde{\boldsymbol{\psi}}_b (\tilde{\mf u}, s_1, \dots, s_k)$ with  system~\eqref{e:odefast} (where $\mf h$ is given by~\eqref{e:cosasono1}), then 
     \be
     \label{e:completedecay}
         \lim_{y \to + \infty} \Big| \mf v(y) - \mf v_0 [\tilde{\mf u}, s_2, \dots, s_k](y) \Big| = 0. 
     \eq
\item[B)] The map $\boldsymbol{\psi}_b$ is Lipschitz continuous with respect to both $\tilde{\mf u}$ and $s_1, \dots, s_k$. It also is differentiable with respect to $s_1, \dots, s_k$ at any point $(\tilde{\mf u}, \mf 0_{k})$ and the columns of the Jacobian matrix are the vector in~\eqref{e:dercurvekld}, the vectors $\mf q_2, \dots, \mf q_{k-1}$ in~\eqref{e:definiscot} and the vector $(1, \mf 0_{N-1})^t$.
\end{itemize}
\end{lemma}
\begin{proof}
We use the notion of slaving manifold and in particular Lemma~\ref{l:slaving}.  We fix $(\tilde{\mf u}, s_2, \dots, s_k)$, we set $\mf v^\ast: =(\mf u^\ast, \mf 0_{N-1}, 0)$ and $\check{\mf v}: = (\boldsymbol{\psi}_{sl}, \mf 0_{N-1}, 0)$. The function $\boldsymbol{\psi}_{sl}$ is the same as in the statement of Theorem~\ref{t:slowmanifold} and it is evaluated at the point $(\tilde{\mf u}, s_2, \dots, s_k)$. We apply Lemma~\ref{l:uniformdecay} with~\eqref{e:ode} given by~\eqref{e:odefast} and $\mf v^\ast$ and $\check{\mf v}$ as before. Note that by linearizing~\eqref{e:odefast}, applying Lemma~\ref{l:signature} and recalling that in we are considering the case $h=1$ we conclude that the number $n_-$ in the statement of Lemma~\ref{l:uniformdecay} is $1$. Owing to Lemma~\ref{l:uniformdecay}, we can define a function $\mf m_-$, which depends 
on $s_1$ and $\check{\mf v}$ and hence (recalling the expression of $\check{\mf v}$) on $s_2, \dots, s_k$ and $\tilde{\mf u}$. Next, we apply Lemma~\ref{l:slaving} with $\mf v_0(0) = \mf v_0 [\tilde{\mf u}, s_1, \dots, s_k](0)$ and $\mf v^\ast$ and $\check{\mf v}$ as before.
     We set 
     \be 
\label{e:effebi1}
\tilde{\boldsymbol{\psi}}_b (\tilde{\mf u}, s_1, \dots, s_k): =  \mf m_- (\tilde{\mf u}, s_1, \dots, s_k) - \check{\mf v} +
       \mf m_p (\tilde{\mf u}, s_1, \dots, s_k) + \mf v_0 [\tilde{\mf u}, s_1, \dots, s_k](0) . 
\eq  
Owing to~\eqref{e:expdecay2}, property A2) in the statement of Lemma~\ref{l:fastvaria} holds true. To establish property A1) it suffices to recall  
that the last component of $\tilde{\boldsymbol{\psi}_{sl}}(\tilde{\mf u}, s_1, \dots, s_k)$ and hence of $\mf v_0 [\tilde{\mf u}, s_1, \dots, s_k]$
is identically $0$, use~\eqref{e:completedecay} and recall that the last component of every solution of~\eqref{e:odefast} is constant. 

To establish property B), we combine property B) in the statement of Theorem~\ref{t:slowmanifold}, the regularity statements in Lemma~\ref{l:uniformdecay},~\eqref{e:diffat0} and Lemma~\ref{l:dipendenzav0}. 
\end{proof}
\subsection{Back to the original variables} \label{ss:back}
We recall that the function $\mf v$ satisfying~\eqref{e:completedecay} is a solution of~\eqref{e:odefast}. We now want to go back to the original system~\eqref{e:odeslow} and obtain a boundary layer of~\eqref{e:symmetric}. This is possible owing to Lemma~\ref{l:back} below. In the statement of the lemma, $\boldsymbol{\psi}_b$ is the projection onto $\R^N$ of the map $\tilde{\boldsymbol{\psi}}_b$, which attains values in $\R^N \times \R^{N-1} \times \R$. 
\begin{lemma}
\label{l:back}
If $\alpha \circ \boldsymbol{\psi}_b (\tilde{\mf u}, s_1, \dots, s_k) >0$ then there is a solution of~\eqref{e:odeslow} such that $\alpha (\mf u(x)))>0$ for every $x>0$ and furthermore 
\be
     \label{e:2bdproblem4}
        \sigma \equiv 0, \qquad 
        \mf u (0) = \boldsymbol{\psi}_b (\tilde{\mf u}, s_1, \dots, s_k), \qquad 
        \lim_{x \to + \infty} \mf u(x) = \underline{\mf u}. 
     \eq
\end{lemma}
\subsection{Proof of Lemma~\ref{l:back}}
The proof of Lemma~\ref{l:back} relies on some preliminary results. 
\begin{lemma}
\label{l:diverge0}
Let $(\mf u_0, \mf z_{20}, 0)$ be a solution of~\eqref{e:odefast} lying on $\mathcal M^0$. If $\alpha (\mf u_0)>0$ at $y=0$, then  $\alpha (\mf u_0)>0$ for every $y \ge 0$
and furthermore
\be
\label{e:diverge2}
\int_0^{+ \infty} \alpha (\mf u_0(y)) dy = + \infty.  
\eq   
Conversely, if $\alpha (\mf u_0)<0$ at $y=0$, then  $\alpha (\mf u_0)<0$ for every $y \ge 0$
and the integral at the left hand side of~\eqref{e:diverge2} equals $- \infty$.
\end{lemma}
\begin{proof}[Proof of Lemma~\ref{l:diverge0}]
We only consider the case where $\alpha (\mf u_0)>0$ at $y=0$, the other case is analogous.  
We apply Lemma~\ref{l:dividiamo} with $a: = \alpha$ and  $f: = \partial_{u_1} \alpha$ (the partial derivative of $\alpha$ with respect to the first component of $\mf u$) and we point out that, owing to Hypothesis~\ref{h:jde}, the hypotheses of Lemma~\ref{l:dividiamo} are satisfied. By applying Lemma~\ref{l:dividiamo} we conclude that $\partial_{u_1} \alpha = \alpha g$ for some suitable function $g$.  

Next, we set $\alpha_0: = \alpha (\mf u_0)$ and we recall that by restricting~\eqref{e:odefast} on $\mathcal M^0$ we obtain~\eqref{e:suemme0}. By using the equality  $\partial_{u_1} \alpha = \alpha g$, using again Hypothesis~\ref{h:jde} and  recalling~\eqref{e:definiamod} and $\sigma=0$ we conclude that $d \alpha_0/ dy = \alpha_0^2 \tilde g$ for some smooth function $\tilde g$ (its precise expression is not important here). We term $m$ a constant satisfying $|\tilde g| \leq m$ (recall that $\mf u_0$ is confined in a neighbourhood of $\mf u^\ast$ by definition of center manifold) and by the comparison principle for ODEs we arrive at 
\be 
\label{e:bdasotto}
   \frac{\alpha_0(0)}{m y \alpha_0(0)+1} \leq \alpha_0(y), 
\eq
which implies that, if $\alpha_0(0)>0$, then $\alpha_0 (y)>0$ for every $y>0$ and furthermore~\eqref{e:diverge2} holds true.  
\end{proof}
By relying on Lemma~\ref{l:diverge0} we establish the following result. 
\begin{lemma}
\label{l:diverge}
Let  $\tilde{\boldsymbol{\psi}_b}$ and $\mf v$ be the same as in the statement of Lemma~\ref{l:fastvaria} and let $\mf u$ denote the first $N$ components of $\mf v$. Fix $\tilde{\mf u}, s_1, \dots, s_k$.  If $\alpha (\boldsymbol{\psi}_b(\tilde{\mf u}, s_1, \dots, s_k))>0$, then $\alpha(\mf u(y))>0$ for every $y\ge 0$ and furthermore 
\be
\label{e:diverge}
\int_0^{+ \infty} \alpha (\mf u(y)) dy = + \infty.  
\eq  
\end{lemma}
\begin{proof}[Proof of Lemma~\ref{l:diverge}]
We assume $\alpha (\boldsymbol{\psi}_{b})>0$: by construction, this means that $\alpha (\mf u) >0$ at $y=0$. By combining Hypothesis~\ref{h:jde} and~\eqref{e:odefast} and recalling~\eqref{e:cosasono1} and that $\sigma=0$ we conclude that $d \mf \alpha (u(y))/dy$ is $0$ when $\alpha (\mf u(y))=0$. This implies that, if $\alpha (\mf u(y))>0$ at $y=0$, then $\alpha (\mf u(y))>0$ for every $y$.   We are left to establish~\eqref{e:diverge}{\comment.}

We recall the definition of $\mf v_0$ given before the statement of Lemma~\ref{l:fastvaria} and we term $\mf u_0$ the first $N$ components of $\mf v_0$. We recall that $\mf v_0$ lies on $\mathcal M^0$, and hence satisfies Lemma~\ref{l:diverge0}. Owing to~\eqref{e:expdecay}, 
\be 
\label{e:convgamma}
   |\alpha (\mf u(y)) - \alpha (\mf u_0(y))| \leq \unpo e^{-2\gamma y },
\eq
for every $y>0$ and for a suitable constant $\gamma>0$. We separately consider the following cases. \\
{\sc Case 1:} $\alpha (\mf u_0 )>0$ at $y=0$. We write $\alpha (\mf u)= \alpha (\mf u_0) + \alpha (\mf u)- \alpha (\mf u_0)$ and by combining~\eqref{e:diverge2} and~\eqref{e:convgamma} we arrive at~\eqref{e:diverge}. \\
{\sc Case 2:} $\alpha (\mf u_0) < 0$ at $y=0$. By combining the decomposition $\alpha (\mf u)= \alpha (\mf u_0) + \alpha (\mf u)- \alpha (\mf u_0)$ with the second part of Lemma~\ref{l:diverge} we conclude that the right hand side of~\eqref{e:diverge} is $- \infty$, which contradicts the fact that $\alpha (\mf u(y))>0$ for every $y \ge 0$. This means that {\sc Case 2} cannot occur. \\
{\sc Case 3:} $\alpha (\mf u_0) = 0$ at $y=0$. Owing to Lemma~\ref{l:diverge}, this implies that $\alpha (\mf u_0)=0$ for every $y \ge 0$ and hence that $\alpha (\mf u)= \alpha (\mf u)-\alpha (\mf u_0)$ for every $y$. Owing to~\eqref{e:convgamma} this implies that $|\alpha (\mf u(y))|\leq \unpo e^{-2\gamma y }$ for every $y \ge 0$ and for some $\gamma >0$. On the other hand, by combining Hypothesis~\ref{h:jde},~\eqref{e:cosasono1} and~\eqref{e:odefast} and recalling Lemma~\ref{l:dividiamo} and that $\mf z_2$ is confined in $\mathrm{B}_\delta^{N-1} (\mf 0_{N-1})$ we conclude that $d \alpha (\mf u (y)) / dy = \unpo \delta \alpha (\mf u(y)).$ By the comparison principle for ODEs, this implies that $\alpha (\mf u (y)) \ge \alpha (\mf u(0)) \exp(\unpo \delta y)$. If $\delta$ is sufficiently small, this contradicts the estimate $|\alpha (\mf u(y))|\leq \unpo e^{-2\gamma y }$, provided that $\alpha (\mf u (0))\neq 0$. Hence, we actually have $\alpha (\mf u (0))=\alpha (\boldsymbol{\psi}_b)=0$, which contradicts the assumption $\alpha (\boldsymbol{\psi}_b)>0$. This means that {\sc Case 3} cannot occur and concludes the proof of Lemma~\ref{l:diverge}. \end{proof}
We can now provide the 
\begin{proof}[Proof of Lemma~\ref{l:back}] We fix $\tilde{\mf u}, s_1, \dots, s_k$ and, as in the statement of  Lemma~\ref{l:fastvaria}, we term $\mf v$ the solution of the Cauchy problem obtained by coupling~\eqref{e:odefast} with the initial condition 
     $\mf v(0) = \tilde{\boldsymbol{\psi}}_b (\tilde{\mf u}, s_1, \dots, s_k)$.  We term $\mf u$ the first $N$ components of $\mf v$.  We define the function $x : [0, + \infty[ \to \R$ by setting  
\be \label{e:changev}
   x (y) : = \int_0^y \alpha (\mf u(z)) dz. 
\eq
Assume that $\alpha(\boldsymbol{\psi}_b(\tilde{\mf u}, s_1, \dots, s_k))>0$: owing to Lemma~\ref{l:diverge}, this implies that $x$ is an invertible  
change of variables from $[0, + \infty[$ onto $[0, + \infty[$. We term $\eta$ its inverse. The function $\mf v \circ \eta$ is a solution of~\eqref{e:odeslow} and $\alpha\circ \mf u \circ \eta (x)>0$ for every $x>0$ owing to Lemma~\ref{l:diverge}. The see that $ \mf u \circ \eta$ satisfies~\eqref{e:2bdproblem4} we recall~\eqref{e:completedecay} and we point out that by definition (see the discussion before the statement of Lemma~\ref{l:fastvaria}) the first $N$ components of the function $\mf v_0$ are a function $\mf u$ satisfying~\eqref{e:2bdproblem4}. This concludes the proof of Lemma~\ref{l:back}. 
\end{proof}
\subsection{Application to the Navier-Stokes and MHD equations} \label{ss:6:nsmhd}
\subsubsection{Navier-Stokes equations} \label{sss:6:ns} We recall the discussion in~\S~\ref{ss:h:ns} and~\S~\ref{sss:4:ns} and that $k=2$ and we conclude that that $\boldsymbol{\psi}_b$ depends on $(\tilde{\mf u}, s_1, s_2)$. Concerning the analysis in~\S~\ref{ss:back}, recall that the function $\alpha$ is $\alpha (\rho, u, \theta)=u$.   
\subsubsection{MHD equations with $\eta>0$} \label{sss:6:mhd} We recall the discussion in~\S~\ref{sss:5:mhd}, that $k=4$ and that $\alpha (\mf u) =u$. We conclude that that $\boldsymbol{\psi}_b$ depends on $(\tilde{\mf u}, s_1, s_2, s_3, s_4)$.
\section{Proof of the main results}
\label{s:brie} In this section we establish the proof of Theorem~\ref{t:main}, Proposition~\ref{p:nc} and Corollary~\ref{c:traccia} in the case where $h=1$, and we refer to~\S~\ref{s:bigk} for the case $h>1$. 
\subsection{Proof overview}
We overview the proof  of Proposition~\ref{p:nc} and Corollary~\ref{c:traccia}. Theorem~\ref{t:main} is actually a corollary of 
Proposition~\ref{p:nc} and we discuss its proof in~\S~\ref{ss:ptm}.  Also, note that we mostly focus on the case where the data $\mf u_i$ and $\mf u_b$ are close to a state $\mf u^\ast$ satisfying~\eqref{e:uast}, which is the case we termed \emph{doubly characteristic}. If~\eqref{e:uast} does not hold,  the analysis is actually simpler, and we discuss it in~\S~\ref{ss:nonchar}. 

The very basic idea of the proof of Theorem~\ref{t:main} and Proposition~\ref{p:nc} is the same as in the paper by Lax~\cite{Lax}. We recall that the key point in~\cite{Lax} is the construction of the $i$-th admissible wave fan curve $\boldsymbol{\phi}_i (\tilde{\mf u})$ through a given state $\tilde{\mf u}$. This curve contains all the states that can be connected to $\tilde{\mf u}$ by either a rarefaction wave or an admissible shock or a contact discontinuity with speed close to the $i$-th eigenvector of $\mf E^{-1} \mf A$. One then composes the curves $\boldsymbol{\phi}_1, \dots, \boldsymbol{\phi}_N$, takes the inverse function and determines a solution of the Cauchy problem obtained by patching together finitely many rarefaction waves, admissible shocks and contact discontinuities. In the present paper we focus on the initial-boundary value problem and we work in the domain $x \in [0, + \infty[$. We use the admissible wave-fan curves $\boldsymbol{\phi}_{k+1}, \dots, \boldsymbol{\phi}_N$ (a more general version than those in~\cite{Lax}), to connect the initial datum $\mf u_i$ to some state $\tilde{\mf u}$. Next, we have to describe the states that can be connected to $\tilde{\mf u}$ by rarefaction waves, admissible shocks and contact discontinuities with speed bigger than, but close to, $0$, and by boundary layers. This is done by using the analysis discussed in the previous section, and in particular Theorem~\ref{t:slowmanifold} and Lemma~\ref{l:fastvaria}. To conclude we have to show that the composite map is invertible. 

The proof is organized as follows. In~\S~\ref{ss:awfc} we recall the construction of the admissible wave fan curve under more general hypotheses than those in~\cite{Lax}. In~\S~\ref{ss:fullbc} we discuss the proof of the main results in the case where we assign $N$ boundary conditions on~\eqref{e:symmetric},  in~\S~\ref{ss:enneuno} we consider the case where we assign $N-1$ boundary conditions. In~\S~\ref{ss:ptm} we provide the proof of Theorem~\ref{t:main}, which follows from the proof of Proposition~\ref{p:nc}. In~\S~\ref{ss:brie:proof} we establish the proof of two technical lemmas. 
In~\S~\ref{ss:nonchar} we discuss the case where~\eqref{e:uast} is violated. 
\subsection{The admissible wave fan curve} \label{ss:awfc}
The admissible wave fan curve $\boldsymbol{\phi}_{i}$ was first constructed in~\cite{Lax} under the assumptions that the system is in conservation form and that the $i$-th characteristic field is either linearly degenerate or genuinely nonlinear. These hypotheses were later relaxed in a series of papers by Liu~\cite{Liu1,Liu2}, Tzavaras~\cite{Tzavaras} and Bianchini~\cite{Bianchini}.  We now recall a result from~\cite{Bianchini}. 
\begin{lemma}
\label{l:Taireg} Under Hypotheses~\ref{h:normal},$\dots$,~\ref{h:jde},
for every $i=(k+1), \dots, N$, there is a sufficiently small constant $\delta>0$ such that the following holds. There is a function $\boldsymbol{\phi}_i: \mathrm{B}^N_\delta(\mf u^\ast) \times ]-\delta, \delta[ \to \R^N$ satisfying the following properties:
\begin{itemize}
\item[i)] For every $(\tilde{\mf u}, s_i) \in \mathrm{B}^N_\delta(\mf u^\ast) \times ]-\delta, \delta[ $, the value $\tilde{\mf u}$ (on the right) is connected to $\boldsymbol{\phi}_i^{s_i} (\tilde{\mf u})$ (on the left) by rarefaction waves and an at most countable number of shocks and contact discontinuities satisfying Liu admissibility condition. 
\item[ii)] The map $\boldsymbol{\phi}_i$ is Lipschitz continuous with respect to both $\tilde{\mf u}$ and $s_i$. It also is differentiable with respect to $s_i$ at any point $(\tilde{\mf u}, 0)$ and furthermore
\be 
\label{e:dercurvei}
   \left. \frac{\partial \boldsymbol{\phi}^{s_i}_i (\tilde{\mf u}) }{\partial s_i} \right|_{s_i=0} = 
   \mf r_i (\tilde{\mf u}), 
   \eq
   where $\mf r_i (\tilde{\mf u})$ is an eigenvector of $\mf E^{-1}\mf A(\tilde{\mf u})$ associated to $\lambda_i (\tilde{\mf u})$.
\end{itemize}
\end{lemma}
Note that, if the system is not in conservation form, the Rankine-Hugoniot conditions are not defined and hence the fact that two given states 
$\mf u^-$ and $\mf u^+ \in \R^N$ can be the left and the right state of a shock or a contact discontinuity \emph{depends} on the underlying viscous mechanism.  The notion of shock curve in the nonconservative case is discussed in~\cite[\S3]{Bianchini}.
\subsection{The case where we assign a full boundary condition} \label{ss:fullbc}
In this paragraph we tackle the case where $\alpha>0$ at the boundary, which implies that $\boldsymbol{\beta}(\mf u, \mf u_b) = \mf u- \mf u_b$ and hence that we assign $N$ boundary conditions at $x=0$. We first establish the proof of Proposition~\ref{p:nc}, next the proof of Corollary~\ref{c:traccia}.  
\subsubsection{Proof of Proposition~\ref{p:nc}}
First, we fix $\mf u_i$ and we consider the map 
\be 
\label{e:brs1}
  \boldsymbol{\zeta}_{tot} (s_1, \dots, s_N): = \boldsymbol{\psi}_b(\boldsymbol{\phi}^{s_{k+1}}_{k+1}\circ\dots \boldsymbol{\phi}^{s_N}_{N} (\mf u_i),  s_1, \dots, s_k).
\eq 
The meaning of the above formula is the following: we are evaluating the function $\boldsymbol{\psi}_b$ (the same as in the statement of Lemma~\ref{l:fastvaria}) at the point 
$(\tilde{\mf u}, s_1, \dots, s_k)$, where $\tilde{\mf u}$ is given by $\boldsymbol{\phi}^{s_{k+1}}_{k+1}\circ\dots \boldsymbol{\phi}^{s_N}_{N} (\mf u_i)$. 
\begin{lemma}
\label{l:brie1} Assume that $\mf u^\ast$ satisfies~\eqref{e:uast}. There is a sufficiently small constant $\delta >0$ such that, if $|\mf u_i - \mf u^\ast| < \delta$, then the map $\boldsymbol{\zeta}_{tot}$ defined as in~\eqref{e:brs1} is locally invertible in a neighborhood of $\mf 0_N$. 
\end{lemma}
The proof of Lemma~\ref{l:brie1} is provided in~\S~\ref{ss:brie:proof}.  Owing to Lemma~\ref{l:brie1}, the equation $\boldsymbol{\zeta}_{tot} (s_1, \dots, s_N) = \mf u_b$ uniquely determines the values of $(s_1, \dots, s_N)$ provided $|\mf u_i - \mf u_b|$ is sufficiently small. We now exhibit $\mf u$ satisfying properties i),$\dots$ iii) in the statement of Proposition~\ref{p:nc}. We set $\tilde{\mf u}: = \boldsymbol{\phi}^{s_{k+1}}_{k+1}\circ\dots \boldsymbol{\phi}^{s_N}_{N} (\mf u_i)$ and we point out that, owing to Lemma~\ref{l:Taireg}, $\tilde{\mf u}$ (on the left) and $\mf u_i$ (on the right) are joined by rarefactions waves and Liu admissible shocks and contact discontinuities.  In particular, if the system is in conservation form and every vector field is either genuinely nonlinear or linearly degenerate one can use Lax's construction~\cite{Lax}. Assume for a moment that we can indeed use the same construction as in~\cite{Lax} and consider the Cauchy problem between $\mf u_i$ (on the right) and $\tilde{\mf u}$ (on the left): there is a value $\tilde{\lambda}>0$ close to $\lambda_{k+1} (\tilde{\mf u})$ such that the solution of the Cauchy problem is identically equal to $\tilde{\mf u}$ for $x < \tilde{\lambda} t$. The solution of the Riemann problem with data $\mf u_i$ (on the right) and $\tilde{\mf u}$ (on the left) in the general case is constructed in~\cite{Bianchini,BianchiniBressan} and it is identically equal to $\tilde{\mf u}$ for $x < \tilde{\lambda} t$.

    We now define the function $\mf u$ satisfying the statement of Proposition~\ref{p:nc} on the set $x > \tilde{\lambda} t$: on this set we define it as the solution of the Riemann problem with data $\mf u_i$ (on the right) and $\tilde{\mf u}$ (on the left) and we refer to~\cite{Bianchini,BianchiniBressan,Lax} for the explicit expression (see in particular~\cite[formula (14.7)]{BianchiniBressan}).  Next, we consider the value $s_k$ and we complete the proof of Proposition~\ref{p:nc}. 

We first assume that the $k$-th vector field is linearly degenerate, namely we assume conditions I) and II) at the beginning of~\S~\ref{ss:ld}. We separately consider the following cases:
\begin{itemize}
\item if $\lambda_k (\tilde{\mf u}) >0$, then we recall (see~\S~\ref{ss:proofsmld}) that $\underline{\mf u}= \boldsymbol{\zeta}_k (\tilde{\mf u}, s_k)$ and we set $\bar{\mf u}:= \underline{\mf u} = \boldsymbol{\zeta}_k (\tilde{\mf u}, s_k)$. 
We set 
\be 
\label{e:brsld0}
     \mf u(t, x) = \left\{
     \begin{array}{ll}
     \text{see~\cite{Bianchini,BianchiniBressan,Lax}} & x > \tilde{\lambda} t{,} \\
     \tilde{\mf u} & \lambda_k (\tilde{\mf u}) t < x < \tilde{\lambda} t{,} \\ 
     \bar{\mf u}  & 0 <  x < \lambda_k (\tilde{\mf u}) t{,}  \\
     \end{array}
     \right.
\eq 
and we now show that the above function satisfies  properties i), ii) and iii) in the statement of Proposition~\ref{p:nc}. We recall the proof of Lemma~\ref{l:curvekld} and we conclude that property Ai) in the statement of Lemma~\ref{l:curvekld} is satisfied, which in turn implies (recalling the definition of $\tilde{\mf u})$ that properties i) and ii)  in the statement of Proposition~\ref{p:nc} are satisfied. Property iii)$_1$ is trivial since $\underline{\mf u}= \bar{\mf u}$. To establish property iii)$_2$ we use~\eqref{e:2bdproblem4} and we recall that we have imposed 
$\boldsymbol{\psi}_b (\tilde{\mf u}, s_1, \dots, s_k)=\mf u_b$. 
\item if $\lambda_k (\tilde{\mf u} )= 0$, then we recall (see~\S~\ref{ss:proofsmld}) that $\underline{\mf u}= \boldsymbol{\zeta}_k (\tilde{\mf u}, s_k)$ and in this case we set $\bar{\mf u}: = \tilde{\mf u}$. We set 
\be 
\label{e:brsld1}
     \mf u(t, x) = \left\{
     \begin{array}{ll}
     \text{see~\cite{Bianchini,BianchiniBressan,Lax}} & x > \tilde{\lambda} t{,} \\
     \tilde{\mf u} & 0 < x < \tilde{\lambda} t{.} \\ 
     \end{array}
     \right.
\eq 
Note that property iii)$_1$ in the statement of Theorem~\ref{t:main} is satisfied because property Ai) in the statement of Lemma~\ref{l:curvekld} is satisfied, which implies (since $\lambda_k (\tilde{\mf u}) = 0$ and the field is linearly degenerate) that there a Liu admissible contact discontinuity joining $\bar{\mf u} = \tilde{\mf u}$
(on the right) with $\underline{\mf u}= \boldsymbol{\zeta}_k (\tilde{\mf u}, s_k)$ (on the left). We can establish the other properties of Proposition~\ref{p:nc} by arguing as in the previous case. 
\item if $\lambda_k (\tilde{\mf u} )< 0$, then we recall (see~\S~\ref{ss:proofsmld}) that $\underline{\mf u}= \tilde{\mf u}$ and we set $\bar{\mf u}: = \underline{\mf u}= \tilde{\mf u}$ and define $\mf u$ as in~\eqref{e:brsld1}. Recall that property Aii) in the statement of Lemma~\ref{l:curvekld} is satisfied. By the definition of $\tilde{\mf u}$, properties i) and ii) in the statement of Proposition~\ref{p:nc} are satisfied. Property iii)$_1$ is trivial since $\underline{\mf u}= \bar{\mf u}$. To establish property iii)$_2$ we use~\eqref{e:2bdproblem4} and we recall that we have imposed 
$\boldsymbol{\psi}_b (\tilde{\mf u}, s_1, \dots, s_k)=\mf u_b$. 
\end{itemize} 
We now consider the general case, when assumptions I) and II) at the beginning of~\S~\ref{ss:ld} are not necessarily satisfied.  We assume for simplicity $s_k>0$ (the case $s_k <0$ is analogous) and we recall that the function $f$ is defined as in~\eqref{e:cosaeffe} and that $\underline{\mf u}= \mf u_k(\underline \tau)$, where $\mf u_k$ satisfies~\eqref{e:mappaT} and $\underline \tau$ is as in~\eqref{e:massimo}. We let $\sigma$ be the same as in~\eqref{e:mappaT} and we recall that $\sigma$ is non-negative and non-increasing. We define the value $\bar \tau$ by setting 
\be 
\label{e:taubarra}
   \bar \tau : =\left\{ \begin{array}{ll}
    \min\{ \tau \in [0, s_k]: \;  \sigma (\tau)= 0 \} & \sigma (s_k) =0{,} \\
    s_k & \sigma(s_k) > 0{.} \\
    \end{array}
    \right.
\eq 
See also Figure~\ref{f:mon}. 
We now define $\mf u$ by setting 
\be 
\label{e:brs2}
     \mf u(t, x) = \left\{
     \begin{array}{ll}
     \text{see~\cite{Bianchini,BianchiniBressan,Lax}} & x > \tilde{\lambda} t{,} \\
     \tilde{\mf u} & \sigma(0) t < x < \tilde{\lambda} t{,} \\ 
     \mf u_k (\tau) & \sigma (\tau) t = x, \; 0 \leq \tau \leq \bar \tau, \\
     \end{array}
     \right.
\eq 
where as before $\mf u_k$ satisfies~\eqref{e:mappaT}. Note that the trace of the function $\mf u$ in~\eqref{e:brs2} on the $t$ axis is $\bar{\mf u}: = \mf u_k (\bar \tau)$. We now verify that the function $\mf u$ in~\eqref{e:brs2} satisfies the properties in the statement of Theorem~\ref{t:main}. We need the following result (where we denote by $f$ the same function as in~\eqref{e:cosaeffe} and by $\underline \tau$ the same value as in~\eqref{e:massimo})
\begin{lemma}
\label{l:taubarra} We have $\bar \tau \leq \underline{\tau}$ and $\monconc_{[0, s_k]}f (\bar \tau) = f (\bar \tau) = f (\underline \tau)=m$, where $m$ is the maximum of $f$ as in~\eqref{e:massimo}.
\end{lemma}
We postpone the proof of Lemma~\ref{l:taubarra} and we point out that, owing to the inequality  
$\bar \tau \leq \underline{\tau}$ and to Lemma~\ref{l:concave}, $\monconc_{[0, s_k]}f = \mathrm{conc}_{[0, s_k]}f$ on $[0, \bar \tau]$. This implies that we can apply the analysis in~\cite{Bianchini} and conclude that $\bar{\mf u} = \mf u_k(\bar \tau)$ (on the left) is connected to $\tilde{\mf u}$ (on the right) by  rarefaction waves and a sequence of Liu admissible shocks or contact discontinuities  with strictly positive speed and that the function $\mf u$ in~\eqref{e:brs2} satisfies properties i) and ii) in the statement of Proposition~\ref{p:nc}. 

To establish property iii), we separately consider the the cases $\bar \tau =s_k$ and $\bar \tau < s_k$. Assume $\bar \tau=s_k$: since $\bar \tau \leq \underline{\tau}$ by Lemma~\ref{l:taubarra}, we can infer from $\bar \tau=s_k$ that $\bar \tau = \underline \tau$ and hence $\bar{\mf u} = \underline{\mf u}$, which implies that property iii)$_1$ is trivially satisfied. To establish property iii)$_2$ we use~\eqref{e:2bdproblem4}. 

If $\bar \tau<s_k$, we recall that $f(\bar \tau) = f(\underline \tau) =m$ by Lemma~\ref{l:taubarra} and that $\monconc_{[0, s_k]} f \equiv 
\mathrm{conc}_{[0, s_k]} f$ on  $[0, \underline \tau]$ by Lemma~\ref{l:concave}: by applying again the analysis in~\cite{Bianchini} we conclude that $\bar{\mf u}$ (on the right) and $\underline{\mf u}$ (on the left) are connected by a  Liu admissible shock or contact discontinuity. To establish property iii)$_2$ in the statement of Proposition~\ref{p:nc} we use~\eqref{e:2bdproblem4}. This concludes the proof of Proposition~\ref{p:nc}. \begin{proof}[Proof of Lemma~\ref{l:taubarra}]
To establish the inequality $\bar \tau \leq \underline \tau$, we recall that by Lemma~\ref{l:concave} $\mathrm{moncon}_{[0, s_k]}f \equiv m$ on $[\underline \tau, s_k]$, which in turn implies that 
$\sigma \equiv 0$  on $[\underline \tau, s_k]$ because by definition $\sigma$ is the derivative of $\mathrm{moncon}_{[0, s_k]}f \equiv m$. By the definition of $\bar \tau$, this implies that $\bar \tau \leq \underline \tau$. 

We now establish the equalities $\monconc_{[0, s_k]}f (\bar \tau) = f (\bar \tau) = f (\underline \tau)=m$ by separately considering the following cases. \\
{\sc Case 1:} $\bar \tau=s_k$. This implies $\bar \tau=\underline \tau$. By~\eqref{l:concave} and the definition of $\underline \tau$ (see~\eqref{e:massimo}) we have  $\monconc_{[0, s_k]}f (\underline \tau) = f (\underline \tau)=m$ and hence we conclude that 
$\monconc_{[0, s_k]}f (\bar \tau) = f (\bar \tau) = f (\underline \tau)=m$, which concludes the proof of the lemma. \\
{\sc Case 2:} $\bar \tau< s_k$. Since $\sigma$ is a non-increasing function, then $\sigma \equiv0$ on $[\bar \tau, s_k]$ and in particular $\sigma \equiv 0$ on $[\bar \tau, \underline \tau]$, which implies that $ \monconc_{[0, s_k]}f (\bar \tau)= \monconc_{[0, s_k]}f (\underline \tau)$. On the other hand, by~\eqref{l:concave} and the definition of $\underline \tau$ (see~\eqref{e:massimo}) we have
 $ \monconc_{[0, s_k]}f (\underline \tau)=f(\underline \tau)=m$. To conclude we are left to establish the equality $f(\bar \tau) = \monconc_{[0, s_k]}f (\bar \tau)$. We separately consider the cases $\bar \tau=0$ and $\bar \tau >0$. \\
{\sc Case 2A:} $\bar \tau >0$. We recall that by Lemma~\ref{l:concave} and the inequality $\bar \tau \leq \underline \tau$ we have $\monconc_{[0, s_k]}f \equiv \mathrm{conc}_{[0, s_k]}f$ on $[0, \bar \tau]$. We claim that there is a strictly increasing sequence $\{ \tau_n \}$ such that $\tau_n \uparrow \bar \tau$
as $n \to \infty$ and $f(\tau_n) =  \mathrm{conc}_{[0, s_k]}f (\tau_n)$. If this were not the case, then there would be $\tau_0< \bar \tau$ such that $f < \mathrm{conc}_{[0, s_k]}f$ on $]\tau_0, \bar \tau[$. This would imply that $\sigma$ is constant on $]\tau_0, \bar \tau[$. Since $\sigma$ is a continuous function and $\sigma (\bar \tau)=0$, then we would have $\sigma \equiv 0$ on $]\tau_0, \bar \tau[$, and this would contradict the definition of $\bar \tau$. Hence the sequence  $\{ \tau_n \}$ exists and by taking the limit in the equality $f(\tau_n) =  \mathrm{conc}_{[0, s_k]}f (\tau_n) = \mathrm{monconc}_{[0, s_k]}f (\tau_n) $ we establish the desired equality. 
 \\
 {\sc Case 2B:} $\bar \tau =0$. Since $f(0)=0$ by the definition~\eqref{e:cosaeffe} of $f$, then  $\mathrm{conc}_{[0, s_k]}f(0)=0$. 
Since $\monconc_{[0, s_k]}f \equiv \mathrm{conc}_{[0, s_k]}f$ on $[0, \underline \tau]$, then $\monconc_{[0, s_k]}f (0) = f(0) =0$ and this concludes the proof of the lemma. 
\end{proof}
\subsubsection{Proof of Corollary~\ref{c:traccia}}
We recall that in~\S~\ref{ss:fullbc} we are focusing on the case where $\alpha (\mf u_b) > 0$. We first show that $\alpha (\underline{ \mf u}) \ge 0$. To see this 
we recall Lemma~\ref{l:back} and that $\mf u_b = \boldsymbol{\psi}_b (\tilde{\mf u}, s_1, \dots, s_k)$. Owing to~\eqref{e:2bdproblem4} and to the continuity of $\alpha$, we conclude that $\alpha (\underline{\mf u}) \ge 0$.  

We are left to prove that, if $\alpha (\underline{\mf u}) \ge 0$, then $\alpha (\bar{\mf u}) \ge 0$. We first consider the linearly degenerate case. More precisely, we assume that properties I) and II) at the beginning of~\S~\ref{ss:ld} are satisfied. We recall that by the analysis at the previous paragraph if either $\lambda_k(\tilde{\mf u})>0$ or $\lambda_k (\tilde{\mf u})<0$, then $\underline{\mf u}=\bar{\mf u}$ 
and hence the property $\alpha (\bar{\mf u}) \ge 0$ is trivial. If $\lambda_k (\tilde{\mf u})=0$, we recall 
that $\underline{\mf u}$ and $\bar{\mf u}$ both lie on the curve $\boldsymbol{\zeta}_k(\tilde{\mf u}, s_k)$, which is defined as in proof of Lemma~\ref{l:curvekld}, and hence the implication $\alpha (\underline{\mf u}) \ge 0 \implies \alpha (\bar{\mf u}) \ge 0$
follows from property C) in the statement of Lemma~\ref{l:curvekld}.

We now establish the implication $\alpha (\underline{\mf u}) \ge 0 \implies \alpha (\bar{\mf u}) \ge 0$ in the general case. We recall that $\underline{\mf u} = \mf u_k (\underline \tau)$ and that $\bar{\mf u} = \mf u_k (\bar \tau)$ for suitable values $\underline \tau$ and $\bar \tau$ and for $\mf u_k$ satisfying~\eqref{e:mappaT}. To establish the implication it suffices to show that $\alpha \circ \mf u_k$ cannot change sign on $[0, s_k]$, and to prove this it suffices to show that, if $\alpha \circ \mf u_k =0$, then $\partial (\alpha \circ \mf u_k)/ \partial \tau =0$. To see this, we use again~\eqref{e:jde},~\eqref{e:definiamod} and the structure of the derivative $\partial \mf u_k/ \partial \tau$, which comes from~\eqref{e:mappaT}. 
\subsection{The case where we assign $N-1$ boundary conditions}\label{ss:enneuno}
In this paragraph we establish the proof of Proposition~\ref{p:nc} and Corollary~\ref{c:traccia} in the case where $\alpha\leq 0$ at the boundary, which implies that $\boldsymbol{\beta}(\mf u, \mf u_b) = \mf 0_N$ if and only if $\mf u_2- \mf u_{2b}= \mf 0_{N-1}$ and hence that we assign $N-1$ boundary conditions at $x=0$ on~\eqref{e:symmetric}. 

First, we fix $\mf u_i$ and we term $\mf p_{\mf u_2}: \R^N \to \R^{N-1}$ the projection 
$\mf p_{\mf u_2}(u_1, \mf u_2)= \mf u_2$ and we consider the map 
\be 
\label{e:brs3}
   \boldsymbol{\zeta}_{par} (s_2, \dots, s_N): = \mf p_{\mf u_2} \circ \boldsymbol{\psi}_{sl}(\boldsymbol{\phi}^{s_{k+1}}_{k+1}\circ\dots \boldsymbol{\phi}^{s_N}_{N} (\mf u_i),  s_2, \dots, s_k).
\eq 
The meaning of the above formula is the following: we are evaluating the function $\boldsymbol{\psi}_{sl}$ (the same as in the statement of Theorem~\ref{t:slowmanifold}) at the point 
$(\tilde{\mf u}, s_2, \dots, s_k)$, where $\tilde{\mf u}$ is given by $\boldsymbol{\phi}^{s_{k+1}}_{k+1}\circ\dots \boldsymbol{\phi}^{s_N}_{N} (\mf u_i)$. 
\begin{lemma}
\label{l:brie2} Assume that $\mf u^\ast$ satisfies~\eqref{e:uast}. There is a sufficiently small constant $\delta >0$ such that, if $|\mf u_i - \mf u^\ast| < \delta$, then the map $\boldsymbol{\zeta}_{par}$ defined as in~\eqref{e:brs1} is locally invertible in a neighborhood of  $\mf 0_{N-1}$. 
\end{lemma}
The proof of Lemma~\ref{l:brie2} is established in~\S~\ref{ss:brie:proof}. To establish the proof of Proposition~\ref{p:nc} we impose 
\be
\label{e:lbc}
      \mf p_{\mf u_2} (\mf u_b) = \mf p_{\mf u_2} \circ \boldsymbol{\psi}_{sl}(\boldsymbol{\phi}^{s_{k+1}}_{k+1}\circ\dots \boldsymbol{\phi}^{s_N}_{N} (\mf u_i),  s_2, \dots, s_k),  
\eq
which is equivalent to say that $\boldsymbol{\beta} (\boldsymbol{\psi}_{sl}(\boldsymbol{\phi}^{s_{k+1}}_{k+1}\circ\dots \boldsymbol{\phi}^{s_N}_{N} (\mf u_i),  s_2, \dots, s_k), \mf u_b )= \mf 0_N$. 

The rationale underpinning~\eqref{e:lbc} is the following. If $\alpha (\mf u_b) \leq 0$, then we can only impose a boundary condition on the last $N-1$ conditions on the solution of~\eqref{e:symmetric}, i.e. we can impose a boundary condition on 
$\mf u_2$. This loss of boundary condition is consistent with the fact that there is no ``fast component" of the boundary layers, i.e. we do not need the analysis in~\S~\ref{s:complete}. Indeed, recall that in~\S~\ref{s:complete} we have constructed the fast component of the boundary layers and a key
     point in the construction is that, if $\alpha (\mf u(0)) >0$, then~\eqref{e:changev} is an 
     invertible change of variables from $[0, + \infty[$ onto $[0, + \infty[$. If $\alpha (\mf u(0)) <0$, then~\eqref{e:changev} 
    maps $[0, + \infty[$ onto $]- \infty, 0]$ and if $\alpha (\mf u(0)) =0$ then~\eqref{e:changev} does not define a change of variables.  This implies that formula~\eqref{e:2bdproblem4} does not extend to the case $\alpha (\boldsymbol{\psi}_b) \leq 0$ and explains why there is no ``fast component" of the boundary layers in the case $\alpha (\mf u_b) \leq 0$.

Owing to Lemma~\ref{l:brie2}, equation~\eqref{e:lbc} uniquely determines the values of $(s_2, \dots, s_N)$ and the rest of the proof of Proposition~\ref{p:nc} is basically the same as in the case where we assign a full boundary condition, see~\S~\ref{ss:fullbc}. The main difference is that when in~\S~\ref{ss:fullbc} we use~\eqref{e:2bdproblem4} here we have to use~\eqref{e:2bdproblem2}. The proof of Corollary~\ref{c:traccia} is also the same. We omit the details. 

\subsection{Proof of Theorem~\ref{t:main}}
\label{ss:ptm}
The proof of Theorem~\ref{t:main} follows from the proof of Proposition~\ref{p:nc}. Indeed, assume that the hypotheses of Theorem~\ref{t:main} are satisfied: in particular, there is an invertible diffeomorphism $\mf u \leftrightarrow \mf w$ such that, if $\mf w$ satisfies~\eqref{e:viscouscl}, then $\mf u$ satisfies~\eqref{e:symmetricee}, and viceversa.  Apply Proposition~\ref{p:nc} to $\mf u$ and consider the function $\mf w(t, x): = \mf w(\mf u(t, x))$. 
We claim that $\mf w$ satisfies properties a), b), c) and d) in the statement of Theorem~\ref{t:main}. Properties a) and d2) are a direct consequence of properties i) and iii)$_2$ in the statement of Proposition~\ref{p:nc}. To establish the other properties, we point out  that, by the proof of Proposition~\ref{p:nc}, $\bar{\mf u}$ is connected to $\mf u_i$ by rarefaction waves, shocks and contact discontinuities constructed by relying on the admissible wave fan curves defined in~\cite{Bianchini} (see also Lemma~\ref{l:Taireg} in here). By the analysis in~\cite{Bianchini}, this implies that the state $\bar{\mf w} = \mf w(\bar{\mf u})$ is connected to $\bar{\mf w} = \mf w(\mf u_i)$  by rarefaction waves, shocks and contact discontinuities  and that $\mf w(t, x)$ is a distributional solution of~\eqref{e:cl}. Very loosely speaking, the reason why this is true is because the admissible wave fan curve $\boldsymbol{\phi_i}$ is constructed by relying on smooth solutions (more precisely, traveling waves solutions) of~\eqref{e:symmetricee}, and smooth solutions of~\eqref{e:symmetricee} are in a one-to-one correspondence with smooth solutions of~\eqref{e:viscouscl}.

\subsection{Proof of Lemmas~\ref{l:brie1} and~\ref{l:brie2}} \label{ss:brie:proof}
The proof is organized as follows. First, we recall some notation and we establish a preliminary result, i.e. Lemma~\ref{l:invertiamo}. Next, we establish Lemma~\ref{l:brie1} and Lemma~\ref{l:brie2}. 

We recall that $\mf q_2 (\mf u), \dots, \mf q_k(\mf u)$ are the vectors defined as in~\eqref{e:definiscot} and $\mf r_{k}(\mf u), \dots, \mf r_N(\mf u)$ are the eigenvectors of $\mf E^{-1} \mf A(\mf u)$ associated to $\lambda_k(\mf u), \dots, \lambda_N (\mf u)$, respectively.  Note that, owing to~\eqref{e:uast} and Lemma~\ref{l:autovettore}, $\mf q_k(\mf u^\ast)= \mf r_k(\mf u^\ast)$. 
\begin{lemma}
\label{l:invertiamo}
Let $\mf u^\ast$ satisfy~\eqref{e:uast}.  The vectors 
\be 
\label{e:linindip}
  \left(
  \begin{array}{cc}
  1 \\
  \mf 0_{N-1} \\
  \end{array}
  \right), 
  \mf q_2 (\mf u^\ast), \dots, \mf q_{k-1} (\mf u^\ast), 
  \mf r_k (\mf u^\ast), \dots, \mf r_{N} (\mf u^\ast).   
\eq
are linearly independent. 
\end{lemma}
\begin{proof}[Proof of Lemma~\ref{l:invertiamo}]
In the proof we assume that the vectors $\mf q_2, \dots, \mf q_{k-1}, 
  \mf r_k, \dots, \mf r_{N}$ are alway evaluated at the point $\mf u^\ast$. First, we make the following remarks.
  \begin{itemize}
  \item The vectors $\mf r_k, \dots, \mf r_N$ are linearly independent because they are eigenvectors associated to different eigenvalues.   We term $W_+$ the space generated by $\mf r_k, \dots, \mf r_N$, which has dimension $N+1-k$. 
  \item The vectors $\mf q_2, \dots, \mf q_{k-1}$  are also linearly independent. To see this, we recall definition~\eqref{e:definiscot}, the fact that $\mf p_2, \dots \mf p_{k-1}$ are linearly independent and the fact that the matrix $\mf R_0$ has maximal rank. We term $W_-$ the space generated by $\mf q_2, \dots, \mf q_{k-1}$, which has dimension $k-2$.
  \item We also term $W_b$ the space generated by $(1, \mf 0_{N-1})^t$
  \end{itemize}
We now proceed according to the following steps. \\
{\sc Step 1:}  we show that $W_b \cap W_- = \big\{ \mf 0_N \big\}$. It suffices to combine~\eqref{e:definiscot}, the fact that $\mf p_2, \dots \mf p_{k-1}$ are linearly independent and the fact that the matrix $\mf R_0$ has maximal rank. \\
{\sc Step 2:} we show that $W_+ \cap \big( W_b \oplus W_-\big) = \big\{ \mf 0_N \big\}. $ We fix $\boldsymbol{\xi}$ in the intersection, namely 
\be 
\label{e:cosaeq}
    \boldsymbol{\xi} = \sum_{i=2}^{k-1} a_i \mf q_i + a_1 
    \left(
        \begin{array}{cc}
        1 \\
        \mf 0_{N-1} \\
        \end{array}
        \right) = \sum_{j=k}^N b_i \mf r_j
\eq 
for suitable coordinates $a_1, \dots, a_{k-1}$ and $b_k, \dots, b_N$. We want to show that $\boldsymbol{\xi}= \mf 0_N$. We use the second equality in~\eqref{e:cosaeq}, we recall that $\mf A \mf r_j= \lambda_j \mf E \mf r_j$ and finally we use~\eqref{e:ortho} with $\mf T= \mf E$. We conclude that   
\be
\label{e:>0}
  \boldsymbol{\xi}^t \mf A \boldsymbol{\xi} = \sum_{j=k}^N b_j^2 \lambda_j  |\mf r_i|^2{\comment.}
\eq 
Next, we point out that, owing to the relation $\mf a_{21}^t \mf R_0(\mf u^\ast, \mf 0_{N-2}, 0) = \mf 0^t_{N-2}$ (see~\eqref{e:definiamod}) , 
$$
    \mf q_i^t \mf A \mf q_i = 
    \left(
    \begin{array}{cc}
          - e_{11}^{-1} \mf d^t \mf p_i  \\
           \mf R_0\mf p_i 
    \end{array}
    \right)^t   
    \left(
    \begin{array}{cc}
          0 & \mf a_{21}^t  \\
          \mf a_{21} & \mf A_{22} \\
    \end{array}
    \right)
    \left(
    \begin{array}{cc}
             - e_{11}^{-1} \mf d^t \mf p_i  \\
           \mf R_0\mf p_i 
    \end{array}
    \right) =
    \varsigma_i |\mf p_i|^2, 
$$
where $\varsigma_i<0$ is the eigenvector of $\mf R_0^t \mf A_{22} \mf R_0(\mf u^\ast, \mf 0_{N-2}, 0)$ associated to $\mf p_i$. 
By using the first equality in~\eqref{e:cosaeq} and the relation $\mf p_i^t\mf p_j = 0$ if $i \neq j$ (see the discussion before~\eqref{e:definiscot}) we arrive at 
\be
\label{e:<0}
  \boldsymbol{\xi}^t \mf A \boldsymbol{\xi} = \sum_{i=2}^{k-1} a_i^2 \varsigma_i |\mf p_i|^2 .
\eq 
We now compare~\eqref{e:>0} and~\eqref{e:<0} and we recall that $\lambda_k (\mf u^\ast)=0$, $0< \lambda_{k+1} (\mf u^\ast)< \dots < \lambda_N(\mf u^\ast)$ and $\varsigma_2, \dots, \varsigma_{k-1} <0$. We conclude that $a_2=\dots =a_{k-1}=0=b_{k+1} = \dots = b_N$. 
This implies that $a_1 (1, \mf 0_{N-1})^t= b_k \mf r_k$, but owing to the Kawashima-Shizuta condition~\eqref{e:ks} this implies that $a_1=b_k=0$, namely that $\boldsymbol{\xi} = \mf 0_N$.  This concludes the proof of the lemma.  
\end{proof}
We can now provide the 
\begin{proof}[Proof of Lemma~\ref{l:brie1}]
First, we combine property B) in the statement of Lemma~\ref{l:fastvaria} with property B) in the statement of Lemma~\ref{l:Taireg} and we conclude that the map under consideration is Lipschitz continuous and differentiable at $(s_1, \dots, s_N) = \mf 0_{N}$. 
Next, we recall a version of the {I}mplicit {F}unction Theorem valid for Lipschitz continuous maps (see~\cite[p.253]{Clarke}) and we conclude that to show that the map is invertible it suffices to show that the columns of the Jacobian matrix evaluated at $\mf 0_N$ are linearly independent. By continuity, it suffices to prove that they are linearly independent in the case where $\mf u_i= \mf u^\ast$. This is true by Lemma~\ref{l:invertiamo}. The proof of the lemma is complete.  
\end{proof}
Finally, we provide the 
\begin{proof}[Proof of Lemma~\ref{l:brie2}]
We argue as in the proof of Lemma~\ref{l:brie1} and we conclude that it suffices to show that 
 the vectors 
$
    \mf p_{\mf u_2} \circ \mf q_2, \dots, \mf p_{\mf u_2} \circ \mf q_{k-1} (\mf u^\ast),  
    \mf p_{\mf u_2} \circ \mf r_k, \dots,  \mf p_{\mf u_2} \circ \mf r_N(\mf u^\ast)
$
are linearly independent. Assume that 
$$
    \sum_{i=2}^{k-1} a_i \mf p_{\mf u_2} \circ \mf q_i (\mf u^\ast) +
    \sum_{j=k}^{N} b_j \mf p_{\mf u_2} \circ \mf r_j (\mf u^\ast) = \mf 0_{N-1}, 
$$
for some real numbers $a_2, \dots, a_{k-1}, b_k, \dots, b_N$. This implies that the vector $\sum_{i=2}^{k-1} a_i\mf q_i +
    \sum_{j=k}^{N} b_j \mf r_j$ belongs to the space generated by $(1, \mf 0_{N-2}^t)$. Owing to Lemma~\ref{l:invertiamo}, this implies that $a_2= \dots= a_{k-1}= b_k= \dots= b_N=0$ and shows that the vectors  $\mf p_{\mf u_2} \circ \mf q_2 (\mf u^\ast), \dots, \mf p_{\mf u_2} \circ \mf q_{k-1} (\mf u^\ast),  
    \mf p_{\mf u_2} \circ \mf r_k(\mf u^\ast), \dots,  \mf p_{\mf u_2} \circ \mf r_N (\mf u^\ast)$ are linearly independent. 
\end{proof}
 
\subsection{The case where~\eqref{e:uast} does not hold} \label{ss:nonchar}
We recall that the data  $\mf u_b$ and $\mf u_i$ in the statement of Theorem~\ref{t:main} and Proposition~\ref{p:nc} are sufficiently close. We can assume that they belong to a sufficiently small neighborhood of a given state $\mf u^\ast$. We have so far discussed considered the case where $\mf u^\ast$ satisfies~\eqref{e:uast}. We now discuss the proof of Theorem~\ref{t:main}, Proposition~\ref{p:nc} and Corollary~\ref{c:traccia} in the case where~\eqref{e:uast} does not hold. We separately consider the following cases. \\
{\sc Case 1:} if $\alpha (\mf u^\ast) =0$, but all the eigenvalues of $\mf E^{-1} \mf A$ are bounded away from $0$, then the boundary is not characteristic, and the analysis is the same as in the present paper, but simpler.  In particular, Lemma~\ref{l:key} modifies as follows: let $n$ be the number of strictly negative eigenvalues of $\mf E^{-1} \mf A$, then the matrix $\mf R_0^t \mf A_{22} \mf R_0 (\mf u^\ast, \mf 0_{N-2}, 0)$ is nonsingular and has $n-1$ eigenvalues with strictly negative real part, and $N-n-1$ eigenvalues with strictly positive real part (each eigenvalue is counted according to its multiplicity). Since the matrix $\mf R_0^t \mf A_{22} \mf R_0 (\mf u^\ast, \mf 0_{N-2}, 0)$ is nonsingular, then we do not need the analysis in~\S~\ref{s:center}. The rest of the analysis is basically the same. \\
{\sc Case 2:} if $\alpha (\mf u^\ast) \neq 0$, then  $\alpha$ is bounded away from $0$ in a sufficiently small neighborhood  
of $\mf u^\ast$. Assume that $\alpha- \lambda_i$ is also bounded away from $0$, for every $\lambda_i$ positive eigenvalues of $\mf E^{-1} \mf A$. Then we can apply the analysis in~\cite{BianchiniSpinolo:ARMA}. \\
{\sc Case 3:}  assume that $\alpha$ is bounded away from $0$, but $\alpha (\mf u^\ast)= \lambda_j (\mf u^\ast)>0$ for some $\lambda_j$ positive eigenvalue of $\mf E^{-1} \mf A$.  We can then apply the analysis in~\cite{BianchiniSpinolo:ARMA} to study the boundary layers, but we need some of the analysis of the present paper to construct the $j$-th admissible wave fan curve. Indeed, to construct this curve we have to study the equation of the traveling waves with speed $\sigma$ close to $\lambda_j$. This means that we have to study~\eqref{e:odeslow}, which is singular at $\mf u^\ast$ when $\sigma = \lambda_j(\mf u^\ast)$. To tackle this issue we consider~\eqref{e:odefast}, we linearize at $\mf v^\ast=(\mf u^\ast, \mf 0_{N-2}, \lambda_j (\mf u^\ast)$ and we construct a center manifold. By arguing as in~\S~\ref{ss:emme0}, we can show that the restriction of~\eqref{e:odeslow} is~\eqref{e:emme0slow} for a suitable function $\mf \Theta_0$ attaining the value $\mf R_0^t [\mf A_{22} - \lambda_j \mf E_{22}]\mf R_0$ at the point 
$(\mf u^\ast, \mf 0_{N-2}, \lambda_j (\mf u^\ast))$. By arguing as in the proof of Lemma~\ref{l:key} we get that the matrix 
$\mf R_0^t [\mf A_{22} - \lambda_j \mf E_{22}]\mf R_0$ evaluated at $(\mf u^\ast, \mf 0_{N-2}, \lambda_j (\mf u^\ast))$ is singular.  We can then argue as in~\S~\ref{ss:cma} and construct a center manifold. To construct the $j$-th admissible wave fan curve we rely on the analysis in~\cite{Bianchini} and argue as in~\S~\ref{ss:general}, the main difference is that in~\eqref{e:mappaT} we have to take the concave envelope instead of the monotone concave envelope.  

\section{The case $h>1$ (the dimension of the kernel of $\mf B$ is larger than 1)}
\label{s:bigk}
In this section we discuss how to extend the analysis at the previous section to the case where $h>1$. Note that, owing to the discussion in~\S~\ref{ss:h:mhd}, this is the case of the MHD equations with $\eta =0$, where $h=3$. 
We now separately discuss the extension of the analysis of each of the previous sections. 
\subsection{The analysis in~\S~\ref{ss:emme0}}
By arguing as in~\S~\ref{ss:emme0}, we can write~\eqref{e:bltw} in the form~\eqref{e:odeslow} provided
\be
\label{e:sonoloro}
\mf h( \mf v) : =
\left(
\begin{array}{ccc}
- \mf E_{11}^{-1} \mf A^t_{21}\mf z_2 \\
\big[  \alpha - \sigma    \big] \mf z_2 \\
\mf B_{22}^{-1}\big([\alpha-\sigma ]
[ \mf A_{22} - \sigma \mf E_{22} - \mf G_{2}] - \,  \mf A_{21} \mf E^{-1}_{11}  \mf A_{21} 
^t  + \mf G_1 \mf E^{-1}_{11} \mf A_{21}^t 
\big) \mf z_2  \\
0 \\
\end{array}
\right){,}
\qquad \mf v : =
\left(
\begin{array}{cc}
\mf u_1 \\
 \mf u_2 \\
\mf z_2 \\
\sigma \\
\end{array}
\right){.}
\eq
Note that now $\mf u_1 \in \R^h$, $\mf u_2, \mf z_2 \in \R^{N-h}$. By linearizing the above equation at the point $(\mf u^\ast, \mf 0_{N-h}, 0)$ we get that the center space is given by 
$$
    M^0 : = { \Big\{} (\mf u_1, \mf u_2, \mf z_2, \sigma) \in \R^{2N}: \ \mf A^t_{21}(\mf u^\ast) \mf z_2 =\mf 0_h {\Big\}}. 
$$ 
Note that the relation $\mf A^t_{21}(\mf u^\ast) \mf z_2 =\mf 0_h$ means that $\mf z_2$ is perpendicular to the $h$ columns of $\mf A_{21}$, which are linearly independent by Lemma~\ref{l:kaw}. This implies that the dimension of (any) center manifold is $2N-2h+1$. We can then repeat the analysis in~\S~\ref{ss:emme0} with no relevant change. Note that now $\mf z_0 \in \R^{N-2h}$ and $\mf R_0 \in \mathbb{M}^{(N-h) \times (N-2h)}$ and recall that $N \ge 2h$ by Lemma~\ref{l:dimensioni}. Also, note that~\eqref{e:definiamod} becomes 
\be
\label{e:definiamod2}
  \mf A_{21}^t \mf R_0 (\mf u, \mf z_0, \sigma) \mf z_0 = [\alpha (\mf u)-\sigma ] \mf D^t (\mf u, \mf z_0, \sigma) \mf z_0
\eq
for a suitable function $\mf D$ attaining values in $\mathbb M^{(N-2h) \times h}$. We eventually arrive at~Lemma~\ref{l:invariant} with~\eqref{e:emme0slow} replaced by  
\be
\label{e:emme0slow2}
\left\{
\begin{array}{lll}
{\mf u_1}' =  -\mf E_{11}^{-1} \mf D^t (\mf u, \mf z_0, \sigma) \mf z_0{,} \\
\mf u'_2 =  \mf R_0  (\mf u, \mf z_0, \sigma) \mf z_0{,} \\
\mf z'_0 =   \mf \Theta_0  (\mf u, \mf z_0, \sigma) \mf z_0{,} \\ 
\sigma' = 0. \\
\end{array}
\right.
\eq
Note that the function $\mf \Theta_0$ attains values in $\mathbb M^{(N-2h) \times (N-2h)}$ and satisfies~\eqref{e:Lambda0}. By arguing as in the proof of Lemma~\ref{l:key} (see also~\cite[Lemma~4.7]{BianchiniSpinolo:ARMA} and apply it with $q=h$, $n_{11}=0$) one can establish the following result. 
\begin{lemma}
\label{l:key2}
Assume that $\mf u^\ast$ satisfies~\eqref{e:uast}.  Then $h+1 \leq k \leq N-h$ and the signature of the matrix $\mf R_0^t \mf A_{22} \mf R_0 (\mf u^\ast, \mf 0_{N-2}, 0) $ is as follows: 
\begin{itemize}
\item $1$ eigenvalue is $0$;
\item $k-h-1$ eigenvalues are strictly negative;
\item $N-k-h$ eigenvalues are strictly positive. 
\end{itemize}
As usual, each eigenvalue is counted according to its multiplicity. 
\end{lemma}
\subsubsection{Applications to the MHD equations with $\eta=0$}
We recall the discussion in~\S~\ref{sss:h:mhdeta0} and that $\mf u= (\rho, \mf b, u, \mf w, \theta)^t$. Note that $\mf u_1 =(\rho, \mf b)^t$ and $\mf u_2 = (u, \mf w, \theta)^t$. Since $N=7$ and $h=3$, then 
the dimension of the manifold $\mathcal M^0$ is $9$, $\mf z_0$ is a real valued function, $\mf R_0$ attains value in $\R^4$ and it is perpendicular to each column of  $\mf A_{21}$ at $(\mf u^\ast, 0, 0)$ owing to~\eqref{e:definiamod2}. By recalling~\eqref{e:identity} we conclude that 
$\mf R_0(\mf u^\ast, 0, 0)~=~(0,  \mf 0_2^t, \sqrt{\rho^\ast \theta^\ast/k(\rho^\ast)})^t$. The function $\mf \Theta_0$ is real valued and attains the value $0$ at $(\mf u^\ast, 0, 0)$, which is consistent with Lemma~\ref{l:key} since in this case $k-h-1=0$ and $N-k-h=0$ (recall that $k=4$).  
\subsection{The analysis in~\S~\ref{s:center}}
The analysis in~\S~\ref{s:center} extends with only minor changes to the case $h>1$. The dimension 
of the manifold $\mathcal M^{00}$ is $N+2$ and one can construct the curve $\boldsymbol{\zeta}_k$ by arguing as in~\S~\ref{ss:ld} (linearly degenerate case) and~\S~\ref{ss:general} (general case). Note that~\eqref{e:dercurvekld} is replaced by  
\be
\label{e:dercurvekld2}
 \frac{\partial \boldsymbol{\zeta}_k}{\partial s_k} (\tilde{\mf u}, 0) =
     \left(
    \begin{array}{cc}
    -\mf E_{11}^{-1} \mf D^t
    \mf r_{00}  \\
    \mf R_0   \mf r_{00}
    \end{array}
    \right) \quad \text{applied at the point} \quad 
   \left\{
    \begin{array}{ll}
   (\tilde{\mf u},  0,  \lambda_k ( \tilde{\mf u}))
    & \text{if} \;  \lambda_k ( \tilde{\mf u}) \ge 0 \\
    (\tilde{\mf u},  0,  0)
    & \text{if} \;  \lambda_k ( \tilde{\mf u}) < 0. \\  
    \end{array}
    \right.  
\eq 
Note that Lemma~\ref{l:autovettore} extends and hence, if $\lambda (\tilde{\mf u})\ge0$, then 
$\partial \boldsymbol{\zeta}_k/\partial s_k$ evaluated at $ (\tilde{\mf u}, 0)$ is an eigenvector of $\mf E^{-1} \mf A (\tilde{\mf u}).$

\subsubsection{Applications to the MHD equations with $\eta=0$}\label{sss:mhd3}
We recall the discussion in~\S~\ref{ss:h:mhd} and and we point out that the analysis in~\S~\ref{ss:cma} is actually redundant in this case because the manifold $\mathcal M^{00}$ is actually the whole manifold $\mathcal M^0$.  Indeed, by linearizing~\eqref{e:emme0slow2} at $(\mf u^\ast, 0, 0)$ we obtain a nilpotent matrix and hence the center space is the whole $\R^9$. 

Since an eigenvector of $\mf E^{-1} \mf A(\mf u)$ associated to $\lambda_4 (\mf u) = u$ is $\mf r_4 (\mf u)=(\rho, \mf 0_2^t, 0, \mf 0_2^t, -\theta)^t$, then the fourth vector field is linearly degenerate and hence we can apply the analysis in~\S~\ref{ss:ld} and we do not need the analysis in~\S~\ref{ss:general}. 
\subsection{The analysis in~\S~\ref{s:blemme0}} The analysis in~\S~\ref{s:center} extends with only minor changes to the case $h>1$. In particular, we can extend Theorem~\ref{t:slowmanifold} and by applying Lemma~\ref{l:key2} we get that $\boldsymbol{\psi}_{sl}$ depends on $\tilde{\mf u}$, on $s_k$ and on other $k-h-1$ scalar variables $s_{h+1}, \dots, s_{k-1}$. The function $\boldsymbol{\psi}_{sl}$ has the same regularity as in property B) in the statement of Theorem~\ref{t:slowmanifold}. The columns of the Jacobian matrix evaluated at the point $(\tilde{\mf u}, \mf 0_{k-h})$ are the vector in~\eqref{e:dercurvekld2} and the vectors 
\be 
\label{e:definiamoq}
    \mf q_i(\mf u) : = 
    \left(
    \begin{array}{cc}
    -\mf E_{11}^{-1}\mf D^t \mf p_i \\
     \mf R_0  \mf p_i \\
    \end{array}
    \right), \quad i=h+1, \dots, k-1, 
\eq 
where $\mf p_{h+1}, \dots, \mf p_{k-1} \in \R^{N-2h}$ are eigenvectors of $\mf R_0^t \mf A_{22} \mf R_0 (\mf u, \mf 0_{N-2h}, 0)$ associated to strictly negative eigenvalues (the ones different from $\theta_{00} (\mf u,0, 0)$ if $\lambda_k (\mf u)<0$). 
\subsubsection{Applications to the MHD equations with $\eta=0$} We recall the discussion in~\S~\ref{sss:mhd3}, that the manifold $\mathcal M^0$ coincides with the manifold $\mathcal M^{00}$ and that $k-h-1=0$. This implies that the analysis in~\S~\ref{s:blemme0} is actually redundant in this case and the function $\boldsymbol{\psi}_{sl}$ in Theorem~\ref{t:slowmanifold} is the same as the function $\boldsymbol{\zeta}_k$ in Lemma~\ref{l:curvekld}.  
\subsection{The analysis in~\S~\ref{s:complete}}
We first focus on Lemma~\ref{l:fastvaria}: part A) in the statement of the lemma extends with no change to the case $h>1$. Part B) should be modified as follows: the regularity is the same, but the columns of the Jacobian marix at $(\mf u^\ast, \mf 0_k)$ are the vector in~\eqref{e:dercurvekld2},  the vectors $\mf q_{h+1}, \dots, \mf q_{k}$ (defined as in~\eqref{e:definiamoq}) and the vectors 
\be 
\label{e:definiamos}
     \mf s_i : =  
     \left(
    \begin{array}{cc}
    -\mf E_{11}^{-1}\mf A_{21}^t \mf t_j\\
    \mf 0_{N-h} \\
    \end{array}
    \right), \quad j=1, \dots, h, 
\eq 
where $\mf t_1, \dots, \mf t_h$ are linearly independent eigenvectors of $- \mf B_{22} \mf A_{21} \mf E_{11}^{-1}\mf A_{21}^t (\mf u^\ast)$ associated the eigenvalues with negative real part (recall Lemma~\ref{l:signature}). 

The main differences in the proof of Lemma~\ref{l:fastvaria} are: i) $\boldsymbol{\psi}_{sl}$ depends on $(\tilde{\mf u}, s_{h+1}, \dots, s_k)$; ii) when we use Lemma~\ref{l:uniformdecay} we apply Lemma~\ref{l:signature} to determine $n_-$ and, if $h>1$, then $n_-=h$. The rest of the proof is basically the same. 
The other results in~\S~\ref{s:complete} extend with no relevant change to the case $h>1$. 
\subsubsection{Applications to the MHD equations with $\eta=0$}
Since $h=3$ and $k=4$, then the function $\boldsymbol{\psi}_b$ depends on $(s_1, s_2, s_3, s_4)$ and on $\tilde{\mf u}$. The function $\alpha (\mf u)$ in~\S~\ref{ss:back} is $\alpha (\mf u) =u.$
\subsection{The analysis in \S~\ref{s:brie}}
To establish Theorem~\ref{t:main} in the case $h>1$ we basically argue as in the case $h=1$. The main difference is that when $\alpha \leq 0$ at the boundary we have to define the function $\boldsymbol{\zeta}_{par}$ by using the projection $\mf p_{\mf u_2}: \R^N \to \R^{N-h}$ by setting $\mf p_{\mf u_2}(\mf u_1, \mf u_2) = \mf u_2.$ The statement of Lemmas~\ref{l:brie1} and~\ref{l:brie2} extends to the case $h>1$. 
To establish their proof we rely on the following extension of Lemma~\ref{l:invertiamo}. 
\begin{lemma}
\label{l:invertiamo2}
Assume that $\mf u^\ast$ satisfies~\eqref{e:uast} and let $\mf s_1, \dots, \mf s_h$, $\mf q_{h+1}, \dots \mf q_{k-1}$  be as in~\eqref{e:definiamos} and~\eqref{e:definiamoq}, respectively.  Then the vectors $\mf s_1 (\mf u^\ast), \dots, \mf s_h (\mf u^\ast)$, $\mf q_{h+1} (\mf u^\ast), \dots \mf q_{k-1} (\mf u^\ast)$, $\mf r_k  (\mf u^\ast), \dots, \mf r_N  (\mf u^\ast)$ are linearly independent. 
\end{lemma}
\begin{proof}[Proof of Lemma~\ref{l:invertiamo}]
In the proof we always assume that the vectors are evaluated at the point $\mf u^\ast$. 
By arguing as in the proof of Lemma~\ref{l:invertiamo} we obtain that the vectors $\mf q_{h+1}, \dots \mf q_{k-1} $ are linearly independent and we term $W^-$ the generated subspace, which has dimension $k-h-1$. Also, we term $W^+$ the space generated by 
$\mf r_k  , \dots, \mf r_N  $ and point out that $W^+$ has dimension $N-k+1$.

We now show that the vectors $\mf s_1, \dots, \mf s_h$ are linearly independent. Assume that there are ${a_1, \dots, a_h \in \R}$ such that $$
   \sum_{i=1}^h a_i \mf s_i= \mf 0_N,
$$ 
then this implies that 
$$
    \sum_{i=1}^h a_i \mf E_{11}^{-1}\mf A_{21}^t \mf t_j= \mf 0_h
$$ 
and hence that 
$$ 
    - \sum_{i=1}^h  a_i \mf B_{22}^{-1} \mf A_{21}\mf E_{11}^{-1}\mf A_{21}^t \mf t_i= \mf 0_h.
$$
On the other hand, since $\mf t_1, \dots, \mf t_h$ are eigenvectors of the matrix $ - \mf B_{22}^{-1} \mf A_{21}\mf E_{11}^{-1}\mf A_{21}^t$ associated to strictly negative eigenvalues, this implies that there is a null linear combination of $\mf t_1, \dots, \mf t_h$. Since by assumption 
$\mf t_1, \dots, \mf t_h$ are linearly independent, this implies that all the coefficients of the linear combination are $0$. Since the $i$-th coefficient is the product between $a_i$ and a strictly negative eigenvalue of $\mf B_{22}^{-1} \mf A_{21}\mf E_{11}^{-1}\mf A_{21}^t$, this implies that $a_1=\dots=a_h=0$ and shows that $\mf s_1, \dots, \mf s_h$ are linearly independent. We term $W^b$ the subspace generated by $\mf s_1, \dots, \mf s_h$, which has dimension $h$, and by recalling~\eqref{e:definiamos} we conclude that $W^b = \{ (\mf u_1, \mf 0_{N-h}): \; \mf u_1 \in \R^h\}$. 

By repeating {\sc Step 1, 2} and {\sc 3} in the proof of Lemma~\ref{l:invertiamo} we then conclude the proof of Lemma~\ref{l:invertiamo2}. 
\end{proof}

\section{Factorization  results} \label{s:fattorizzo}
The following lemma is elementary, we provide the proof for the sake of completeness.
\begin{lemma}
\label{l:dividiamo}
Fix a natural number $m > 2$. 
Assume $f: \R^d \to \R$ and $a: \R^d \to \R$ are two $C^m$ function that satisfy the following properties: 
\begin{itemize}
\item[i.] If $a (\mf x) = 0$, then $\nabla a (\mf x) \neq \mf 0_d$. 
\item[ii.] If $a (\mf x) =0$, then $f (\mf x) = 0$. 
\end{itemize}
Then there is a unique $C^{m-1}$ function $g: \R^d \to \R$ such that 
\begin{equation}
\label{e:aperg}
     f (\mf x ) = a (\mf x) g (\mf x), \quad \text{for every $\mf x \in \R^d$}.
\end{equation}
\end{lemma}
\begin{proof}
We proceed according to the following steps. \\
{\sc Step 1:} we write $\mf x \in \R^d$ as $\mf x = (x_1, \mf x_2)^t$, $x_1\in \R$, $\mf x_2 \in \R^{d-1}$ and 
we exhibit a function $g$ satisfying~\eqref{e:aperg} in the case where $a(x_1, \mf x_2)=x_1$. We set 
$$
    g (\mf x) := \frac{f (x_1, \mf x_2)}{x_1} \; \text{if} \; x_1 \neq 0, 
    \qquad 
    \frac{\partial f}{\partial x_1} (0, \mf x_2) \; \text{if} \; x_1 =0. 
$$
By using the representation
$$
g(x_1, {\mf x}_2) = \frac{1}{x_1} 
     \int_0^{x_1} 
   \frac{\partial f}{\partial x_1} (\xi, {\mf x}_2)d \xi \quad \text{if $x_1 \neq 0${\comment,}}
 $$  
we can show that the above function is continuous and by using the Taylor expansion formula we can find its derivatives, which are continuous up to the order $m-1$.   \\
{\sc Step 2:} we consider the general case.  We fix a point $\mf x^\ast $ such that $a (\mf x^\ast) =0$. Owing to assumption i), $\nabla a (\mf x^\ast) \neq \mf 0_d$. Just to fix the ideas, we assume that $\partial a / \partial x_1 \neq 0$. We consider the map $\mf w: \R^d \to \R^d$ defined by setting
\begin{equation}
\label{e:cambiocoordinate}
    \mf w (\mf x) :=
    \left(
    \begin{array}{cc}
    a (\mf x) \\
    {\mf x}_2 \\
    \end{array}
    \right).
\end{equation}
Since by assumption $\partial a / \partial x_1 \neq 0$ at $\mf x^\ast$, then the Jacobian matrix $\mf D \mf w (\mf x^\ast)$ is non singular. Owing to the Local Invertibility Theorem, there is a radius $r (\mf x^\ast)>0$ such that $ \mf w$ is invertible in $\mathrm{B}^d_{r (\mf x^\ast)}(\mf x^\ast).$ We term $\mf w^{-1}$ its inverse and we point out that, by construction, 
$a (\mf w^{-1}(\mf y)) = y_1$. We can then apply {\sc Step 1} to the functions $a \circ \mf w^{-1}$ and $f \circ \mf w^{-1}$ and infer that there is $g_{\mf x^\ast}$ such that 
$
   f( \mf w^{-1} (\mf y) )= a ( \mf w^{-1} (\mf y) )g_{\mf x^\ast} (\mf y).
$
We can then consider the function $g_{\mf x^\ast} \circ \mf w$, which satisfies~\eqref{e:aperg}, but is only defined in $\mathrm{B}^d_{r (\mf x^\ast)}(\mf x^\ast).$ To obtain a globally defined function and 
conclude the proof of the lemma, we consider the closed set 
$$
   \mathcal A= 
   \{ \mf x \in \R^d: \, a (\mf x) =0 \} \subseteq \bigcup_{\mf x^\ast \in \mathcal A}
  \mathrm{B}^d_{r (\mf x^\ast)}(\mf x^\ast). 
$$
We can find a countable, locally finite set of points $\{ \mf  x_n^\ast \}_{n \ge1}$ such that 
$$
   \mathcal A:= 
   \{ \mf x \in \R^d: \, a (\mf x) =0 \} \subseteq \bigcup_{n=1}^\infty
   \mathrm{B}^d_{r (\mf x^\ast_n)/2}(\mf x^\ast_n). 
$$
We can then construct a partition of unity associated to the above covering. In particular, we can fix a sequence of smooth functions $\{ \theta_n \}_{n \ge 0}$ such that 1) $\sum_{n=0}^\infty \theta_n (\mf x) \equiv 1$; 2) for every $\mf x \in \R^d$, there are at most finitely many $n$'s such that 
$\theta_n (\mf x) \neq 0$; 3) for every $n\ge 1$, the support of $\theta_n$ is contained in $\mathrm{B}^d_{ r (\mf x^\ast_n)} (\mf x^\ast_n)$; 4) the support of $\theta_0$ is contained in $\R^d \setminus \bigcup_{n\ge 1} \mathrm{B}^d_{ r (\mf x^\ast_n)/2} (\mf x^\ast_n)$. 
We can now define the function $g$ by setting 
$$
    g (\mf x) : =
    \sum_{n=1}^\infty \theta_n (\mf x) g_{\mf x_n} \big(\mf w (\mf x) \big) + \frac{f(\mf x)}{a (\mf x)} \theta_0 (\mf x). 
$$
Note that the above series pointwise converges owing to property 2) above. Also, the function $g$ is of class $C^{m-1}$ because both the $\theta$-s and the $g_{\mf x_n}$-s are $C^{m-1}$. To establish~\eqref{e:aperg} it suffices to point out that 
$$
    a (\mf x) g (\mf x)  =
    \sum_{n=1}^\infty \theta_n (\mf x)  a (\mf x)
    g_{\mf x_n} \big(\mf w (\mf x) \big) + f(\mf x) \theta_0 (\mf x)=
    \underbrace{ \sum_{n=0}^\infty \theta_n (\mf x)}_{=1}  f (\mf x)
 = f(\mf x). 
$$
{\sc Step 3:} we establish uniqueness of the function $g$. 
By combining condition i) in the statement of the lemma with the Implicit Function Theorem we infer that $\mathcal A$ can be locally represented as an hypersurface. In particular, $\mathcal A$ is a closed set with empty interior. Owing to~\eqref{e:aperg}, we have $g = f/a$ on $\R^d \setminus \mathcal A$. Since $g$ is by definition $C^{m-1}$, and in particular locally Lipschitz continuous, it can be uniquely extended to the closure of  $\R^d \setminus \mathcal A$, that is  $\R^d$.
\end{proof}
\begin{corol}
\label{c:dividiamo}
Assume that $\mf f: \R^{k_1} \times \R^{k_2} \to \R^{k_3}$ is a $C^{m}$ function satisfying 
\begin{equation}
\label{e:c:fazero}
    \mf f(\mf x, \mf 0_{k_2}) =0, \quad \text{for every $\mf x \in \R^{k_1}$}.
\end{equation}
Then there is a $C^{m-1}$ function $\mf G: 
\R^{k_1} \times \R^{k_1} \to \mathbb M^{k_3 \times k_2}$ such that 
\begin{equation}
\label{e:c:tesi}
    \mf f(\mf x, \mf y) = \mf G(\mf x, \mf y) \mf y.   
\end{equation}
\end{corol}
\begin{proof}
First, we point out that we can restrict with no loss of generality to the case when 
$k_3=1$. Indeed, if $k_3 >1$ then we apply the result in the case $k_3 =1$ to each component of $\mf f$. Also, for simplicity we assume $k_2=2$ and we denote by $(y_1, y_2)$ the components of $\mf y$. The case $k_2 >2$ does not pose additional challenges. We decompose $\mf f$ as 
\begin{equation}
\label{e:c:decompo}
    \mf f(\mf x, y_1, y_2) =  \mf f(\mf x, y_1, 0) +
    \mf f(\mf x, y_1, y_2) - \mf f(\mf x, y_1, 0). 
\end{equation}
For every fixed $\mf x$, the term $\mf f(\mf x, y_1, 0)$ satisfies the assumptions of Lemma~\ref{l:dividiamo} provided that $a(y_1) = y_1$. Hence, $$
    \mf f(\mf x, y_1, 0) = g_1 (\mf x, y_1) y_1
$$
for some $C^{m-1}$ function $g_1: \R^{k_1} \times \R \to \R$. We now consider the second term on the right hand side of~\eqref{e:c:decompo} and we point out that for every fixed $\mf x$, $y_1$ it satisfies the assumptions of Lemma~\ref{l:dividiamo} provided that $a (y_2) =y_2$. This yields to  
$
     \mf f(\mf x, y_1, y_2) - \mf f(\mf x, y_1, 0) = g_2 (\mf x, y_1, y_2) y_2. 
$
By plugging the above equalities into~\eqref{e:c:decompo}  we arrive at~\eqref{e:c:tesi} provided that $\mf G: = (g_1, g_2)$. This concludes the proof of the corollary.   
\end{proof}
\section{Proof of Lemma~\ref{l:key}}
\label{s:proof}
\subsection{Preliminary results}
We first quote two linear algebra results.
\begin{lemma}
\label{l:sym}
         Let $\mf S$, $\mf T \in \mathbb M^{d \times d}$ be two real symmetric matrices and assume that $\mf T$ is positive definite. 
         Then all the eigenvalues of $\mf T^{-1}\mf S$ are real numbers and the matrix $\mf T^{-1}\mf S$ is diagonalizable through a real matrix. Also, assume that $\mf w_i$ and $\mf w_j$ are eigenvectors associated to different eigenvalues. Then
         \be 
         \label{e:ortho}
                \mf w_i^t  \mf T \mf w_j=0{\comment.}
         \eq  
         \end{lemma}
\begin{proof}
The result is known, but we provide the proof for the sake of completeness. Since $\mf T$ is symmetric and positive definite, then $\mf T = \mf M^t \mf M$ for some invertible matrix $\mf M$. For every $\mf r \in \R^d$, $\lambda \in \R$ we have  
$$
    \mf T^{-1} \mf S \mf r = \lambda \mf r \iff \mf S \mf r = \lambda \mf T \mf r \iff \mf S (\mf M)^{-1} \mf M \mf r =
    \lambda \mf M^t \mf M \mf r \iff (\mf M^t)^{-1} \mf S (\mf M)^{-1} \mf M \mf r =
    \lambda \mf M \mf r
$$
This implies that $(\lambda, \mf r)$ is an eigencouple for $\mf T^{-1} \mf S$ if and only if $(\lambda, \mf M \mf r)$ is an eigencouple for  $(\mf M^t)^{-1} \mf S (\mf M)^{-1}$. Since $(\mf M^t)^{-1} \mf S (\mf M)^{-1}$ is a symmetric matrix, then it has  real eigenvalues. Also, it is diagonalizable. We term $\mf r_1, \dots, \mf r_d$ its eigenvectors, which are linearly independent. Since $\mf M \mf r_1, \dots,\mf M \mf r_d$ are also linearly independent, we conclude that $\mf T^{-1} \mf S$ is also diagonalizable. 

We are left to establish~\eqref{e:ortho}. We fix two eigenvalues $\lambda_i \neq \lambda_j$ and we consider the relations $\mf S \mf w_i = \lambda_i \mf T \mf w_i$ and $\mf S \mf w_j = \lambda_j \mf T \mf w_j$, which owing to the fact that $\mf S$ is symmetric yields $\lambda_i \mf w_j^t \mf T \mf w_i= \lambda_j \mf w_i^t \mf T \mf w_j$. Since $\mf T$ is also symmetric, this implies~\eqref{e:ortho}. 
\end{proof}
We now quote a result concerning the signature of $\mf T^{-1} \mf S$. The proof is based on an homotopy argument, see~\cite[Lemma~3.1]{BianchiniSpinolo:ARMA} and~\cite{BenzoniGavageSerreZumbrun,MajdaPego}. 
\begin{lemma}
\label{l:counting}
Under the same assumptions as in Lemma~\ref{l:sym}, the signature of $\mf T^{-1} \mf S$ is the same as the signature of $\mf S$. In other words, $\mf T^{-1} \mf S$ has the same number of strictly negative and strictly positive eigenvalues as $\mf S$. Also, $\mf T^{-1}\mf S$ admits the eigenvalue $0$ if and only if $\mf S$ does, and the multiplicity is the same. 
\end{lemma} 
We now quote a particular case of~\cite[Lemma 4.7]{BianchiniSpinolo:ARMA} and for completeness we provide a sketch of the proof.
\begin{lemma}
\label{l:pieta}
Assume that $\mf A$ and $\mf B$ satisfy Hypothesis~\ref{h:ks} with $h=1$ and evaluate them at a point $\mf u^\ast$ satisfying~\eqref{e:uast}. 
Then all the roots of the polynomial   
\begin{equation}
\label{e:pieta}
   \mathcal P(s) = \mathrm{det} (\mf A  - s \mf B)
\end{equation}
are real numbers. Also, $k-2$ are strictly negative, $N-k-1$ are strictly positive and the root $0$ has multiplicity one. 
\end{lemma}
\begin{proof}
By using the block decomposition of $\mf B$ and of $\mf A$, recalling that $a_{11} (\mf u^\ast)=0$ and developing the determinant from the first row we conclude that $\mathcal P$ is a polynomial of degree $N-2$. Next, we introduce a perturbation argument.  We define the function 
          $d: \mathbb R^2 \to \R$ by setting  
          \begin{equation}
          \label{e:psigma}
                   d(w, s) := \mathrm{det} 
                   \Big(  
                   \mf A - s  (\mf B + w \mf I_N )  \Big).
          \end{equation} 
          Owing to Hypothesis~\ref{h:ks}, the matrices $\mf A$ and $\mf B + w \mf I_N$ are both 
          symmetric. Also, the matrix $\mf B + w \mf I_N$ is positive definite provided 
          that $w >0$.  Owing to Lemma~\ref{l:sym}, for every $w>0$ the equation $d(w, s)=0$ has $N$ real roots: we term them $s_1 (w), \dots, s_N (w)$  (as usual, each root is counted according to its multiplicity). We now fix $i=1, \dots, N$ and we investigate the behavior of $s_i (w)$ 
          for $w \to 0^+$.  Owing to classical results on algebraic functions (see for instance~\cite[Chapter 5]{Knopp}) there are only two possibilities: either $s_i(w)$ is continuous at $w =0$ and $s_i (0)$ is a root of the polynomial $\mathcal P$ 
          defined at~\eqref{e:pieta} or ${\lim_{w \to 0^+} |s_i (w) | = + \infty}$. Also, every root of $\mathcal P$ can be obtained as the limit $\lim_{j \to + \infty}
          s_j(w)$ for some $j=1, \dots, N$. If a root $\bar s$ has multiplicity $m$, then there are exactly $m$ functions $s_{j_1}(w), \dots,s_{j_m} (w)$ such that 
           $$
             \lim_{w \to 0^+} s_{j_1} (w) = 
             \dots = \lim_{w \to 0^+} s_{j_m} (w) = \bar s.  
          $$
          We draw two conclusions from the previous considerations: i) 
          since the functions 
          $s_1 (w), \dots, s_N(w)$ are real numbers, then all the roots of $\mathcal P$ 
          are real numbers;
          ii) since the polynomial $\mathcal P$ has degree $N-2$, there are 
          exactly two functions among $s_1 (w), \dots, s_N(w)$ that are unbounded for $w \to 0^+$. Up to a change in the order, we can assume that 
          \begin{equation}
          \label{e:vanainfty}
               \lim_{w \to 0^+} | s_{1} (w)| =
                \lim_{w \to 0^+} | s_{N} (w)| = + \infty.  
          \end{equation}
          Note that, owing to Lemma~\ref{l:counting} and to the 
          definition~\eqref{e:psigma} of $d$, when $w>0$ there are exactly $k-1$ 
          functions among $s_1 (w), \dots, s_N(w)$ that attain
          strictly 
          negative values, exactly $N-k$ that attain strictly positive values and one function which is identically $0$. 
          By recalling~\eqref{e:vanainfty}, to conclude the proof of the lemma we are left to show that (up to a change in the order) 
          \begin{equation}
          \label{e:larma:fg}
                   \lim_{ w \to 0^+ } s_1 (w) = - \infty \qquad \text{and} 
                   \qquad 
                    \lim_{ w \to 0^+ } s_N (w) = + \infty. 
          \end{equation}        
        To this end, we study the behavior for $w \to 0^+$ of $s_1$ and 
        $s_N$  satisfying~\eqref{e:vanainfty}. We set $\zeta_1 (w) := 
        1/ s_1 (w)$ and $\zeta_N (w) : = 1/s_N(w)$ and we 
        point out that they are both well defined since, owing to~\eqref{e:vanainfty}, both $s_1$ and $s_N$ are bounded away from $0$. Also, 
        \begin{equation}
        \label{e:vanazero}
                \lim_{w \to 0^+} \zeta_1(w) = 
                 \lim_{w \to 0^+} \zeta_N (w) =0. 
        \end{equation}
        Since 
        $s_1$ and $s_N$ are both roots of~\eqref{e:psigma}, then we arrive at  
        \begin{equation}
        \label{e:2det}
            \mathrm{det} \big(\zeta_1 (w) \mf A - \mf B - w \mf I_N \big) =0, \qquad 
            \mathrm{det} \big(\zeta_N  (w)\mf A- \mf B - w \mf I_N \big) =0. 
        \end{equation}
        We now study the eigenvalue problem 
        $
            \mathrm{det} \big(\zeta \mf A - \mf B - w(\zeta) \mf I_N \big) =0, 
        $
        namely the problem of determining the eigenvalues of the matrix 
        $\zeta \mf A-\mf B$ as functions of $\zeta$. Motivated by~\eqref{e:vanazero}, 
        we investigate the limit $\zeta \to 0$.
        We term $w_1(\zeta), \dots, w_N(\zeta)$
        the eigenvalues of  $\zeta \mf A - \mf B$ and, by relying again on classical 
        results on algebraic functions~\cite[Chapter 5]{Knopp}, 
        we conclude that we can order 
        $w_1(\zeta), \dots, w_N(\zeta)$ in such a way that the behavior for $\zeta \to 0^+$ is as follows: 
        the eigenvalues $w_1(\zeta), \dots, w_{N-1}(\zeta)$ converge
        to the $N-1$ strictly negative eigenvalues of $-\mf B$ (i.e. of $-\mf B_{22}$);
        the eigenvalue $w_N (\zeta)$ converge for $\zeta \to 0$ to $0$
        (the remaining eigenvalue of $-\mf B$). 
        By relying on the analysis in~\cite{BianchiniHanouzetNatalini}
        one can show (see also~\cite{BianchiniSpinolo:ARMA}) that 
        $w_N (\zeta)$ has the following Taylor expansion:
         \begin{equation}
    \label{e:ab}
             w_N (\zeta) = \mf a^t_{21} (\mf B_{22})^{-1} \mf a_{21} \zeta^2 + o(\zeta^2) \quad \text{as $\zeta \to 0$}.
    \end{equation}    
    Since $\mf B_{22}$ is positive definite and $\mf a_{21} \neq  \mf 0_{N-1}$ owing to Lemma~\ref{l:kaw}, then 
    $
        \mf a^t_{21} (\mf B_{22})^{-1} \mf a_{21} >0. 
    $
    Hence, there are exactly two distinct functions such that $\zeta_1 (w) <0$ and $\zeta_N (w)>0$ for $w>0$ and 
    $
        w_N (\zeta_1 (w) )= w, \;  w_N (\zeta_N (w) )= w.$ 
     By recovering $s_1$ and $s_N$ as $s_1 (w) = 1/ \zeta_1 (w)$, $s_N (w) =
      1/ \zeta_N (w)$ we eventually establish~\eqref{e:larma:fg}. 
\end{proof}
To complete the proof of Lemma~\ref{l:key} we need the following 
\begin{lemma}
\label{l:iff}Assume that $\mf A$ and $\mf B$ satisfy Hypotheses~\ref{h:ks} and~\eqref{e:jde} with $h=1$ and evaluate them at a point $\mf u^\ast$ satisfying~\eqref{e:uast}. 
For a given $\varsigma \in \R$, the following statements are equivalent: 
\begin{itemize}
\item[i)] $\varsigma$ is an eigenvalue of the matrix $(\mf R_0)^t \mf A_{22} \mf R_0(\mf u^\ast, \mf 0_{N-2}, 0)$;
\item[ii)] $\varsigma$ is a root of the polynomial $\mathcal P$ defined as in~\eqref{e:pieta}. 
\end{itemize}
\end{lemma}
\begin{proof}
{\sc Step 1:} we establish the implication  i)$\implies$ii). First, we recall that by definition the columns of $\mf R_0$ generate the hyperspace of $\R^{N-1}$ orthogonal to $\mf a_{21}$. This implies that   $\mf a_{21}$ and the 
columns of the matrix $\mf R_0$ form a basis of $\R^{N-1}$. We term $\mf p \in \R^{N-2}$ an eigenvector of 
$(\mf R_0)^t \mf A_{22} \mf R_0$ associated to $\varsigma$. We have  
$
   \big[ \mf A_{22} - \varsigma \mf B_{22} \big] \mf R_0 \mf p = \mf R_0 \mf c +  \mf a_{21} c     
$ 
for some $\mf c \in \R^{N-2}$, $c \in \R$. By left multiplying the above expression times $(\mf R_0)^t$ and recalling that $(\varsigma, \mf p)$ is an eigencouple for $(\mf R_0)^t \mf A_{22} \mf R_0$ and using~\eqref{e:identity} and~\eqref{e:definiamod} we arrive at   
$
  (\mf R_0)^t  \mf R_0 \mf c = \mf 0_{N-2},  
$
which implies that $\mf c = \mf 0_{N-2}$ and that 
$
     \big[ \mf A_{22} - \varsigma \mf B_{22} \big] \mf R_0 \mf p =  \mf a_{21} c.      
$
This in turn implies that
$$
\left(
\begin{array}{cc}
0            &    \mf a^t_{21} \\
\mf a_{21}   &    \mf A_{22} - \varsigma \mf B_{22}\\
\end{array}
\right)
\left(
\begin{array}{cc}
- c  \\
\mf R_0 \mf p
\end{array}
\right) =
\left(
\begin{array}{cc}
0 \\
  \big[ \mf A_{22} - \varsigma \mf B_{22} \big] \mf R_0 \mf p - 
   \mf a_{21} c
\end{array}
\right)
=
\left(
\begin{array}{cc}
0 \\
\mf 0_{N-1} \\
\end{array}
\right)
$$
and that $\varsigma$ is a root of the polynomial $\mathcal P$ defined as in~\eqref{e:pieta} because $\mf R_0 \mf p \neq \mf 0_{N-1}$ since $\mf p \neq \mf 0_{N-2}$ and the columns of $\mf R_0$ are linearly independent vectors.  
\\
{\sc Step 2:} we establish the implication ii)$\implies$i). Since $\varsigma$ is a root of the polynomial $\mathcal P$, then there is $\mf b \in \R^N$, $\mf b \neq \mf 0_N$, such that $[\mf A - \varsigma \mf B] \mf b=\mf 0_{N}$. We write $\mf b: = (b_1, \mf b_2)^t$, with $b_1 \in \R$, $\mf b_2 \in \R^{N-1}$ and we point out that the relation  $[\mf A - \varsigma \mf B] \mf b=\mf 0_{N}$ implies that 
\be
\label{e:fanno0}
\mf a_{21}^t \mf b_2  = 0, \qquad 
 \mf a_{21}  b_1 + \big[ \mf A_{22} - \varsigma \mf B_{22} \big] 
 \mf b_2
= \mf  0_{N-1} 
\eq 
From the first equality we infer that $\mf b_2 = \mf R_0 \mf b_3$ for some $\mf b_3 \in \R^{N-2}$. By plugging this relation into the second equality in~\eqref{e:fanno0}, left multiplying times $(\mf R_0 )^t$ and using~\eqref{e:identity} we arrive at 
$\big[ (\mf R_0 )^t \mf A_{22} \mf R_0- \varsigma \mf I_{N-2} \big] \mf b_3 = \mf 0_{N-2} $. This implies that $\varsigma$ is an eigenvalue provided that we show that $\mf b_3 \neq \mf 0_{N-2}.$ Assume by contradiction that $\mf b_3 = \mf 0_{N-2}$, then $\mf b_2 = \mf 0_{N-1}$ and from the second equality in~\eqref{e:fanno0} and the inequality $ \mf a_{21} \neq \mf 0_{N-1}$ we infer that $b_1 =0$, which implies that $\mf b = \mf 0_N$ and contradicts our assumption. This concludes the proof of the lemma.
\end{proof}
\subsection{Conclusion of the proof of Lemma~\ref{l:key}}
If all the all the eigenvalues of $ (\mf R_0)^t \mf A_{22} \mf R_0$ are distinct then Lemma~\ref{l:key} directly follows from~Lemma~\ref{l:pieta} and Lemma~\ref{l:iff}.  To complete the proof, we tackle the case of eigenvalues with higher multiplicity by relying on a perturbation argument.  We proceed according to the following steps. \\
{\sc Step 1:} we show that the multiplicity of $0$ as an eigenvalue of $ (\mf R_0)^t \mf A_{22} \mf R_0$ is exactly $1$. First, we point out that $0$ is an eigenvalue of $ (\mf R_0)^t \mf A_{22} \mf R_0$ by Lemma~\ref{l:iff} because it is a root of $\mathcal P$ by Lemma~\ref{l:pieta}. Next, we 
assume by contradiction that there are $\mf p_1$ and $\mf p_2$, linearly independent, such that $ (\mf R_0)^t \mf A_{22} \mf R_0
\mf p_1 = \mf 0_{N-2}$, $(\mf R_0)^t \mf A_{22} \mf R_0
\mf p_2 = \mf 0_{N-2}$. By arguing as in {\sc Step 1} of the proof of Lemma~\ref{l:iff} we infer that 
there are two linearly independent vectors $\mf q_1$ and $\mf q_2$ such that $\mf A \mf q_1 = \mf A \mf q_2 = \mf 0_{N}$. This implies that $0$ has multiplicity $2$ as an eigenvector of $(\mf E)^{-1} \mf A$ and hence contradicts the strict hyperbolicity.  \\
{\sc Step 2:} we point out that, for every $\nu >0$, there is $\mf A_{22}^\nu$ such that 
\begin{itemize}
\item[i)] $\mf A_{22}^\nu$ is symmetric;
\item[ii)] $(\mf R_0)^t\mf A_{22}^\nu \mf R_0$  has $N-2$ distinct eigenvalues, and one of them is $0$. 
\item[iii)] $\| \mf A_{22}^\nu - \mf A_{22} \| < \nu$. Here $\| \cdot \|$ denotes the Frobenius norm on the space of $(N-1) \times (N-1)$ matrices (any equivalent norm works). 
\end{itemize}
{\sc Step 3:} we now construct the matrix $\mf A^\nu$ by replacing in the block decomposition~\eqref{e:blockae} the block $\mf A_{22}$ with the block $\mf A_{22}^\nu$. We define the polynomial $\mathcal P^\nu(s): = \mathrm{det} (\mf A^\nu - s \mf B )$. By applying Lemma~\ref{l:iff} and recalling property ii) in {\sc Step 2} we conclude that $\mathcal P^\nu$ has $N-2$ distinct roots, and one of them is $0$. On the other hand, if the constant $\nu$ is sufficiently small, then the coefficients of $\mathcal P^\nu$ are close to those of $\mathcal P$ and hence by Lemma~\ref{l:pieta} $\mathcal P^\nu$ has at least $k-1$ strictly negative roots, and $N-k-1$ strictly positive roots. Since one eigenvalue is $0$, we conclude that $\mathcal P^\nu$ has exactly the $k-1$ strictly negative and $N-k-1$ strictly positive roots. By Lemma~\ref{l:iff}, the same holds for the eigenvalues of  $(\mf R_0 )^t \mf A^\nu_{22}\mf R_0$. By letting $\nu \to 0^+$, using {\sc Step 1} and recalling the continuity of eigenvalues we eventually conclude the proof of Lemma~\ref{l:key}. 
\section{The Slaving Manifold Lemma} 
\label{s:slaving}
In this section we discuss a slaving manifold lemma that we have used in the previous analysis. We refer to the classical book by Katok and Hasselblatt~\cite{KatokHasselblatt} for a comprehensive introduction. See also~\cite{BianchiniSpinolo:JDE}. 

Assume that $\mathbf{g}: \R^d \to \R^d$ is a smooth function and consider the ODE
\begin{equation}
\label{e:ode}
       \mf v ' = \mf g ( \mf v). 
\end{equation}
We assume that $\mf v^\ast \in \R^d$ is an equilibrium, i.e.
$
       \mf g(\mf v^\ast) = \mf 0.  
$
We consider the Jacobian matrix $\mf D \mf g (\mf v^\ast)$ and we term $n_-$, $n_+$ and $n_0$  the number of eigenvalues with strictly negative, strictly positive and zero real part, respectively. We assume that  $n_- \neq 0$, $n_0 \neq 0$.
Also, we term $10 \gamma$ the spectral gap, namely 
\begin{equation}
\label{e:gamma}
  10 \gamma: =  \max \Big\{ | \mathrm{Re} \lambda|: \; \text{$\lambda$ is an eigenvalue of $\mf D \mf g (\mf v^\ast)$, $\mathrm{Re} \lambda \neq 0$} \Big\}. 
\end{equation} 
Note that, by the continuity of the eigenvalues, the number of  eigenvalues of $\mf D \mf g (\check{\mf v})$ satisfying $\mathrm{Re}(\lambda) < -8 \gamma$ is $n_-$, provided that $\delta$ is sufficiently small. We now want to state an elementary extension of the Stable Manifold Theorem, and we refer to~\cite[\S2]{Perko} for the classical statement of the Stable Manifold Theorem. 
Note furthermore that we say that $\mf p \in \R^d$, $\mf p \neq \mf 0_d$ is a generalized eigenvector associated to the eigenvalue $\lambda \in \R$ of a given matrix $\mf F \in \mathbb{M}^{d \times d}$ if there is $m \in \mathbb{N}$ such that $(\mf F - \lambda \mf I_d )^m \mf p = \mf 0_d$. 
\begin{lemma}
\label{l:uniformdecay} There is a constant $\delta$, which only depends on $\mf g$ and $\mf v^\ast$, such that the following holds. For every $\check{\mf v} \in \R^d$ such that $|\check{\mf v} - \mf v^\ast|\leq \delta$ and $\mf g (\check{\mf  v}) = \mf 0$, there is an invariant manifold for~\eqref{e:ode} which contains all the orbits of~\eqref{e:ode} satisfying 
\be
\label{e:expdecay}
   \lim_{t \to + \infty} |\mf v(t) - \check{\mf v}| e^{4\gamma t} =0 {\comment .}
\eq
The manifold is parameterized by a map $\mf m_-(\check{\mf v}, \cdot): \R^{n_-}  \to \R^d$ which is continuously differentiable and satisfies 
$\mf m_-(\check{\mf v}, \mf 0_{n_-})= \check{\mf v}$. The columns of the Jacobian matrix $\mf D \mf m_- (\check{\mf v}, \mf 0_{n_-} )$ generate the space of the generalized eigenvectors associated with eigenvalues $\lambda$ satisfying $\mathrm{Re}(\lambda) < -8 \gamma$. Also, the map $\mf m_-$ Lipschitz continuously depends on the variable $\check{\mf v}$.
\end{lemma} 
Since $n_0 \neq 0$, by linearizing~\eqref{e:ode} at the equilibrium $\mf v^\ast$ we can construct a center manifold. We fix an orbit $\mf v_0$ entirely lying on the center manifold and confined in a sufficiently small neighborhood of $\mf v^\ast$. We also fix an orbit lying on the manifold constructed in Lemma~\ref{l:uniformdecay} and we term it $\mf v_-$. We now want to construct a function $\mf v_p: \R \to \R^d$ in such a way that by setting $\mf v: = \mf v_0 + \mf v_- - \check{\mf v} + \mf v_p$ we obtain a solution of~\eqref{e:ode} satisfying 
\be
\label{e:expdecay2}
    \lim_{t \to + \infty} |\mf v(t) - \mf v_0 (t) | e^{2 \gamma t} =0.  
\eq
The following results can be established by relying on the same techniques as in~\cite{Bressan:cm}. See also~\cite[Theorem 3.1]{BianchiniSpinolo:JDE}. 
\begin{lemma}[Slaving Manifold Lemma]
\label{l:slaving}
There is a constant $\delta$, which only depends on $\mf g$ and $\mf v^\ast$, such that the following holds. Let $\check{\mf v}$ be as in the statement of Lemma~\ref{l:uniformdecay} and assume that the orbit $\mf v_0$ satisfies $|\mf v_0 (t) - \mf v^\ast| \leq 4 \delta$, for every $t \in \R$. Then there is a map 
$\mf m_p (\mf v_0 (0), \check{\mf v}, \cdot): \R^{n_-} \to \R^d$ such that for every $\mf x \in \R^{n_-}$ the solution of the Cauchy problem obtained by coupling~\eqref{e:ode} with the initial datum 
\be 
\label{e:dato}
    \mf v(0) = \mf v_0(0) + \mf m_- (\check{\mf v}, \mf x) - \check{\mf v} +   \mf m_p (\mf v_0, \check{\mf v}, \mf x)
\eq
satisfies~\eqref{e:expdecay2}. In the above expression,  $\mf m_-$ is the same as in Lemma~\ref{l:uniformdecay}. Also, $\mf m_p$ depends Lipschitz continuously on both $\check{\mf v}$ and $\mf x$ and satisfies 
\be
\label{e:diffat0}
     | \mf m_p (\mf v_0, \check{\mf v}, \mf x)| \leq \unpo | \mf v_0 (0)- \check{\mf v} | \ |\mf x|.  
\eq 
\end{lemma}
Note that the initial point $\mf v_0(0)$ uniquely determines the orbit $\mf v_0$. 
We are now left to discuss how the map $\mf m_p$ depends on the orbit $\mf v_0(0)$.
\begin{lemma}
\label{l:dipendenzav0}
Under the same assumptions as in the statement of Lemma~\ref{l:slaving}, assume that $\mf v_{01}$ and $\mf v_{02}$ are two orbits satisfying $|\mf v_{01} (t) - \mf v^\ast|  \leq 4 \delta$,  $|\mf v_{02} (t) - \mf v^\ast| \leq 4 \delta$, for every $t \in \R$. Assume furthermore that 
\be 
\label{e:distanza}
    |\mf v_{01} (t) - \mf v_{02} (t)| \leq L  |\mf v_{01} (0) - \mf v_{02} (0)| e^{\gamma |t|}, \quad 
    \text{for every $t \in \R$.}
\eq 
Then 
\be 
\label{e:lipschitzv0}
     | \mf m_p (\mf v_{01}(0), \check{\mf v}, \mf x)  -
     \mf m_p (\mf v_{02}(0), \check{\mf v}, \mf x)| \leq \unpo L  |\mf v_{01} (0) - \mf v_{02} (0)|.
\eq 
\end{lemma}
\section*{Acknowledgments}
The authors wish to thank Denis Serre for interesting discussions. 
Both authors are members of the GNAMPA group of INDAM and of the PRIN National
Project ``Hyperbolic Systems of Conservation Laws and Fluid
Dynamics: Analysis and Applications''.
\bibliographystyle{plain}
\bibliography{boundary}

\begin{thebibliography}{10}

\bibitem{AnconaBianchini}
Fabio Ancona and Stefano Bianchini.
\newblock Vanishing viscosity solutions of hyperbolic systems of conservation
  laws with boundary.
\newblock {\em Preprint}.

\bibitem{AnconaMarson}
Fabio Ancona and Andrea Marson.
\newblock Existence theory by front tracking for general nonlinear hyperbolic
  systems.
\newblock {\em Arch. Ration. Mech. Anal.}, 185(2):287--340, 2007.

\bibitem{BenzoniGavageSerreZumbrun}
Sylvie Benzoni-Gavage, Denis Serre, and Kevin Zumbrun.
\newblock Alternate {E}vans functions and viscous shock waves.
\newblock {\em SIAM J. Math. Anal.}, 32(5):929--962, 2001.

\bibitem{Bianchini}
Stefano Bianchini.
\newblock On the {R}iemann problem for non-conservative hyperbolic systems.
\newblock {\em Arch. Ration. Mech. Anal.}, 166(1):1--26, 2003.

\bibitem{BianchiniBressan}
Stefano Bianchini and Alberto Bressan.
\newblock Vanishing viscosity solutions of nonlinear hyperbolic systems.
\newblock {\em Ann. of Math. (2)}, 161(1):223--342, 2005.

\bibitem{BianchiniHanouzetNatalini}
Stefano Bianchini, Bernard Hanouzet, and Roberto Natalini.
\newblock Asymptotic behavior of smooth solutions for partially dissipative
  hyperbolic systems with a convex entropy.
\newblock {\em Comm. Pure Appl. Math.}, 60(11):1559--1622, 2007.

\bibitem{BianchiniSpinolo:ARMA}
Stefano Bianchini and Laura~V. Spinolo.
\newblock The boundary {R}iemann solver coming from the real vanishing
  viscosity approximation.
\newblock {\em Arch. Ration. Mech. Anal.}, 191(1):1--96, 2009.

\bibitem{BianchiniSpinolo:JDE}
Stefano Bianchini and Laura~V. Spinolo.
\newblock Invariant manifolds for a singular ordinary differential equation.
\newblock {\em J. Differential Equations}, 250(4):1788--1827, 2011.

\bibitem{Bressan:book}
Alberto Bressan.
\newblock {\em Hyperbolic systems of conservation laws. The one-dimensional
  Cauchy problem}, volume~20 of {\em Oxford Lecture Series in Mathematics and
  its Applications}.
\newblock Oxford University Press, Oxford, 2000.

\bibitem{Bressan:cm}
Alberto Bressan.
\newblock Tutorial on the center manifold theorem.
\newblock In {\em Hyperbolic systems of balance laws}, volume 1911 of {\em
  Lecture Notes in Math.}, pages 327--344. Springer, Berlin, 2007.

\bibitem{ChristoforouSpinolo}
Cleopatra Christoforou and Laura~V. Spinolo.
\newblock A uniqueness criterion for viscous limits of boundary {R}iemann
  problems.
\newblock {\em J. Hyperbolic Differ. Equ.}, 8(3):507--544, 2011.

\bibitem{Clarke}
Frank~H. Clarke.
\newblock {\em Optimization and nonsmooth analysis}.
\newblock Canadian Mathematical Society Series of Monographs and Advanced
  Texts. John Wiley \& Sons Inc., New York, 1983.
\newblock A Wiley-Interscience Publication.

\bibitem{Dafermos}
Constantine~M. Dafermos.
\newblock {\em Hyperbolic conservation laws in continuum physics}, volume 325
  of {\em Grundlehren der Mathematischen Wissenschaften [Fundamental Principles
  of Mathematical Sciences]}.
\newblock Springer-Verlag, Berlin, fourth edition, 2016.

\bibitem{Gisclon}
Marguerite Gisclon.
\newblock {\'Etude des conditions aux limites pour un syst\`eme strictement
  hyperbolique, via l'approximation parabolique}.
\newblock {\em J. Math. Pures Appl. (9)}, 75(5):485--508, 1996.

\bibitem{Grenier}
Emmanuel Grenier.
\newblock Boundary layers.
\newblock In {\em Handbook of mathematical fluid dynamics. {V}ol. {III}}, pages
  245--309. North-Holland, Amsterdam, 2004.

\bibitem{GrenierRousset}
Emmanuel Grenier and Fr\'ed\'eric Rousset.
\newblock Stability of one-dimensional boundary layers by using {G}reen's
  functions.
\newblock {\em Comm. Pure Appl. Math.}, 54(11):1343--1385, 2001.

\bibitem{GuesMetivierWilliamsZumbrun}
Olivier Gu\`es, Guy M\'etivier, Mark Williams, and Kevin Zumbrun.
\newblock Existence and stability of noncharacteristic boundary layers for the
  compressible {N}avier-{S}tokes and viscous {MHD} equations.
\newblock {\em Arch. Ration. Mech. Anal.}, 197(1):1--87, 2010.

\bibitem{JosephLeFloch}
Kayyunnapara~T. Joseph and Philippe~G. LeFloch.
\newblock Boundary layers in weak solutions of hyperbolic conservation laws.
\newblock {\em Arch. Ration. Mech. Anal.}, 147(1):47--88, 1999.

\bibitem{JosephLeFloch2}
Kayyunnapara~T. Joseph and Philippe~G. LeFloch.
\newblock Boundary layers in weak solutions of hyperbolic conservation laws.
  {II}. {S}elf-similar vanishing diffusion limits.
\newblock {\em Commun. Pure Appl. Anal.}, 1(1):51--76, 2002.

\bibitem{KatokHasselblatt}
Anatole Katok and Boris Hasselblatt.
\newblock {\em Introduction to the modern theory of dynamical systems},
  volume~54 of {\em Encyclopedia of Mathematics and its Applications}.
\newblock Cambridge University Press, Cambridge, 1995.
\newblock With a supplementary chapter by Katok and Leonardo Mendoza.

\bibitem{KawashimaShizuta1}
Shuichi Kawashima and Yasushi Shizuta.
\newblock On the normal form of the symmetric hyperbolic-parabolic systems
  associated with the conservation laws.
\newblock {\em Tohoku Math. J. (2)}, 40(3):449--464, 1988.

\bibitem{Knopp}
Konrad Knopp.
\newblock {\em Theory of {F}unctions. {II}. {A}pplications and {C}ontinuation
  of the {G}eneral {T}heory}.
\newblock Dover Publications, New York, 1947.

\bibitem{Lax}
Peter~D. Lax.
\newblock Hyperbolic systems of conservation laws. {II}.
\newblock {\em Comm. Pure Appl. Math.}, 10:537--566, 1957.

\bibitem{LeVeque}
Randall~J. LeVeque.
\newblock {\em Numerical methods for conservation laws}.
\newblock Lectures in Mathematics ETH Z\"urich. Birkh\"auser Verlag, Basel,
  1990.

\bibitem{Liu1}
Tai~Ping Liu.
\newblock The {R}iemann problem for general {$2\times 2$} conservation laws.
\newblock {\em Trans. Amer. Math. Soc.}, 199:89--112, 1974.

\bibitem{Liu2}
Tai~Ping Liu.
\newblock Existence and uniqueness theorems for {R}iemann problems.
\newblock {\em Trans. Amer. Math. Soc.}, 212:375--382, 1975.

\bibitem{MajdaPego}
Andrew Majda and Robert~L. Pego.
\newblock Stable viscosity matrices for systems of conservation laws.
\newblock {\em J. Differential Equations}, 56(2):229--262, 1985.

\bibitem{MatsumuraNishida}
Akitaka Matsumura and Takaaki Nishida.
\newblock Initial-boundary value problems for the equations of motion of
  compressible viscous and heat-conductive fluids.
\newblock {\em Comm. Math. Phys.}, 89(4):445--464, 1983.

\bibitem{MetivierZumbrun}
Guy M\'etivier and Kevin Zumbrun.
\newblock Large viscous boundary layers for noncharacteristic nonlinear
  hyperbolic problems.
\newblock {\em Mem. Amer. Math. Soc.}, 175(826):vi+107, 2005.

\bibitem{MishraSpinolo}
Siddhartha Mishra and Laura~V. Spinolo.
\newblock Accurate numerical schemes for approximating initial-boundary value
  problems for systems of conservation laws.
\newblock {\em J. Hyperbolic Differ. Equ.}, 12(1):61--86, 2015.

\bibitem{NakamuraNishibata}
Tohru Nakamura and Shinya Nishibata.
\newblock Existence and asymptotic stability of stationary waves for symmetric
  hyperbolic-parabolic systems in half-line.
\newblock {\em Math. Models Methods Appl. Sci.}, 27(11):2071--2110, 2017.

\bibitem{Perko}
Lawrence Perko.
\newblock {\em Differential equations and dynamical systems}, volume~7 of {\em
  Texts in Applied Mathematics}.
\newblock Springer-Verlag, New York, third edition, 2001.

\bibitem{Rousset2}
Fr\'ed\'eric Rousset.
\newblock Stability of small amplitude boundary layers for mixed
  hyperbolic-parabolic systems.
\newblock {\em Trans. Amer. Math. Soc.}, 355(7):2991--3008, 2003.

\bibitem{Rousset}
Fr\'ed\'eric Rousset.
\newblock Characteristic boundary layers in real vanishing viscosity limits.
\newblock {\em J. Differential Equations}, 210(1):25--64, 2005.

\bibitem{Serre1}
Denis Serre.
\newblock {\em Systems of conservation laws. 1 \& 2}.
\newblock Cambridge University Press, Cambridge, 1999.
\newblock Translated from the 1996 French original by I. N. Sneddon.

\bibitem{SerreZumbrun}
Denis Serre and Kevin Zumbrun.
\newblock Boundary layer stability in real vanishing viscosity limit.
\newblock {\em Comm. Math. Phys.}, 221(2):267--292, 2001.

\bibitem{Spinolo:2bvp}
Laura~V. Spinolo.
\newblock Vanishing viscosity solutions of a {$2\times 2$} triangular
  hyperbolic system with {D}irichlet conditions on two boundaries.
\newblock {\em Indiana Univ. Math. J.}, 56(1):279--364, 2007.

\bibitem{Tzavaras}
Athanasios~E. Tzavaras.
\newblock Wave interactions and variation estimates for self-similar
  zero-viscosity limits in systems of conservation laws.
\newblock {\em Arch. Rational Mech. Anal.}, 135(1):1--60, 1996.

\bibitem{Xin}
Zhouping Xin.
\newblock Viscous boundary layers and their stability. {I}.
\newblock {\em J. Partial Differential Equations}, 11(2):97--124, 1998.

\end{thebibliography}
\end{document}